\documentclass[10pt,reqno]{amsart} 
\usepackage{enumitem}
\usepackage{amscd}
\usepackage{amssymb}
\usepackage{graphicx} 
\usepackage{hyperref}
\usepackage[all]{xy}
\usepackage{fge}

\usepackage[bbgreekl]{mathbbol}

\usepackage[utf8]{inputenc}
\RequirePackage[T1]{fontenc}
\RequirePackage{amsfonts,latexsym,amssymb}

\RequirePackage{mathrsfs}
\let\mathcal\mathscr
\let\cal\mathcal

\let\bb\mathbb

\usepackage{bbm}
\usepackage[textsize=tiny]{todonotes}

\newtheorem{theorem}[equation]{Theorem}
 \newtheorem{lemma}[equation]{Lemma}
 \newtheorem{proposition}[equation]{Proposition}
 \newtheorem{corollary}[equation]{Corollary}

\theoremstyle{definition}
\newtheorem{definition}[equation]{Definition}

\newtheorem{remark}[equation]{Remark}
\newtheorem{example}[equation]{Example}
\theoremstyle{remark}
\newtheorem*{acknowledgments}{Acknowledgments}

\def\jcdot{\scriptscriptstyle\bullet}

\def\invlim{\mathop{\vtop{\ialign{##\crcr$\hfill{\lim}\hfil$\crcr
\noalign{\kern1pt\nointerlineskip}\leftarrowfill\crcr\noalign
{\kern -3pt}}}}\limits}
\def\dirlim{\mathop{\vtop{\ialign{##\crcr$\hfill{\lim}\hfil$\crcr
\noalign{\kern1pt\nointerlineskip}\rightarrowfill\crcr\noalign
{\kern -3pt}}}}\limits} 
\def\lomapr#1{\smash{\mathop{\relbar\joinrel\longrightarrow}\limits^{#1}}}
 \def\verylomapr#1{\smash{\mathop{\relbar\joinrel\relbar\joinrel\relbar\joinrel\longrightarrow}\limits^{#1}}}

\def\phi{\varphi}
\def\epsilon{\varepsilon}

\newcommand{\an}{\operatorname{an}}

\newcommand{\Sp}{\operatorname{Sp}}
\newcommand{\Sm}{\operatorname{Sm}}
\newcommand{\Spa}{\operatorname{Spa}}
\newcommand{\mode}{\operatorname{mod}}
\newcommand{\ovk}{\overline{K} }
\newcommand{\tr}{\operatorname{tr} }

\newcommand{\dr}{\operatorname{dR} } 
  \newcommand{\hk}{\operatorname{HK} }   
   \newcommand{\nr}{\operatorname{nr} }

 \newcommand{\colim}{\operatorname{colim} }

 \newcommand{\proeet}{\operatorname{pro\acute{e}t} } 
 \newcommand{\eet}{\operatorname{\acute{e}t} }
 \newcommand{\Spec}{\operatorname{Spec} }
\newcommand{\Spf}{\operatorname{Spf} }
 \newcommand{\Hom}{{\rm{Hom}} }
 \newcommand{\Ext}{\operatorname{Ext} }

\newcommand{\Gal}{\operatorname{Gal} }
\newcommand{\can}{ \operatorname{can} }
\newcommand{\synt}{ \operatorname{syn} }
 
\newcommand{\st}{\operatorname{st} }
 \newcommand{\kker}{\operatorname{Ker} } 
  \newcommand{\HK}{\operatorname{HK} } 
 \newcommand{\coker}{\operatorname{coker} }  
  \newcommand{\DR}{\operatorname{DR} }  
    \newcommand{\Bcrp}{\operatorname{{\mathbf B}^+_{\rm cr}} } 
      \newcommand{\Bstp}{\operatorname{{\mathbf B}^+_{\st}} } 
    
 \newcommand{\crr}{\operatorname{cr} }

 \newcommand{\sff}{{\mathcal{F}}}

 \newcommand{\sg}{{\mathcal{G}}}

 \newcommand{\scc}{{\mathcal{C}}}

 \newcommand{\so}{{\mathcal O}}

 \newcommand{\sa}{{\mathcal{A}}}

\newcommand{\sd}{{\mathcal{D}}}

 \newcommand{\wt}{\widetilde}
    
 \newcommand{\wh}{\widehat}
   \numberwithin{equation}{section}

\def\R{{\mathrm R}}
\def\LL{{\mathrm L}}

 \def\A{{\bf A}} \def\B{{\bf B}}
\def\Q{{\bf Q}} \def\Z{{\bf Z}}

\def\N{{\bf N}}

\def\G{{\cal G}}

\def\bwedge{\underset{\widehat{\phantom{\ \ }}}{\phantom{\_}}}

\def\rg{{\rm R}\Gamma}

\def\epsilon{\varepsilon}

\newcommand\AC{\mathcal{A}}
\newcommand\FC{\mathcal{F}}

\newcommand\IC{\mathcal{I}}
\newcommand\JC{\mathcal{J}}

\newcommand\OC{\mathcal{O}}

\def\AS{{\mathscr{A}}}

\def\DS{{\mathscr{D}}}

\def\FS{{\mathscr{F}}}

\def\IS{{\mathscr{I}}}
\def\JS{{\mathscr{J}}}

\def\MS{{\mathscr{M}}}
\def\OS{{\mathscr{O}}}

\def\US{{\mathscr{U}}}

\def\XS{{\mathscr{X}}}

\newcommand\VF{\mathfrak{V}}
\newcommand\XF{\mathfrak{X}}
\newcommand\YF{\mathfrak{Y}}
\newcommand\ZF{\mathfrak{Z}}

\newcommand{\solid}{\otimes^{\scalebox{0.5}{$\square$}}}
\newcommand{\dsolid}{\otimes^{\LL {\scalebox{0.5}{$\square$}}}}

\newcommand{\wiesia}[1]{\textcolor{red}{[Wiesia: #1]}}

\numberwithin{equation}{section}
 \makeatletter
\let\@wraptoccontribs\wraptoccontribs
\makeatother

\begin{document}
\title[Compactly supported $p$-adic pro-\'etale cohomology of analytic varieties]
 {Compactly supported $p$-adic pro-\'etale cohomology of analytic varieties}
 \author{Piotr Achinger}
\address{Instytut Matematyczny PAN, 
ul.\ \'Sniadeckich 8, 
00-656 Warszawa, Poland}
\email{pachinger@impan.pl}
\author{Sally Gilles}
\address{Universit\"at Duisburg-Essen, Fakult\"at f\"ur Mathematik, Thea-Leymann-Str. 9, 
45127 Essen, Germany}
\email{sally.gilles@uni-due.de}
\author{Wies{\l}awa Nizio{\l}}
\address{CNRS, IMJ-PRG, Sorbonne Universit\'e, 4 place Jussieu, 75005 Paris, France}
\email{wieslawa.niziol@imj-prg.fr}

 \date{\today}
\thanks{W.\ N.'s  research was supported in part by the grant NR-19-CE40-0015-02 COLOSS, the NSF grant No. DMS-1440140, and   the Simons Foundation. 
P.A.\ and W.N.\ were supported by the project KAPIBARA funded by the European Research Council (ERC) under the European Union's Horizon 2020 research and innovation programme (grant agreement No 802787). S.G.’s research was partially
supported by the National Science Foundation under Grant No. DMS-1926686 and by the ERC under the European Union’s Horizon
2020 research and innovation programme (grant agreement No 804176).}
 \begin{abstract}
We study  properties of compactly supported $p$-adic pro-\'etale cohomology of smooth partially proper  rigid analytic varieties. In particular, we prove a comparison theorem, in a stable range,  with compactly supported syntomic cohomology, which is built from compactly  supported Hyodo-Kato and $\B^+_{\dr}$-cohomologies. We derive from that a (limited version of a) fundamental diagram.  This paper should be thought of as a compactly supported counterpart of the series \cite{CN3}, \cite{CN4}, \cite{CN5} but, in particular, as a continuation of \cite{CN4}.
 \end{abstract}

\maketitle

 \tableofcontents

 \section{Introduction} We wrote this paper to familiarize ourselves (and the reader) with the definition of compactly supported $p$-adic pro-\'etale cohomology of smooth partially proper  rigid analytic varieties, to study its basic properties, and to present some computations.  This paper should be thought of as a compactly supported counterpart of the series \cite{CN3}, \cite{CN4}, \cite{CN5} but, in particular, as a continuation of \cite{CN4}.
 
  Let $K$ be a complete discrete valuation field  of characteristic 0  with a perfect
residue field $k$ of characteristic $p$.  Let $C$ be the  $p$-adic completion of an algebraic closure of $K$.  Let $\sg_K$ be the Galois group of $K$.
Let $F={\rm Frac}(W(k))$  and 
let $\varphi$ be the absolute
Frobenius on $W(\overline{k})$. 

 \subsection{Definitions} The study of $p$-adic \'etale cohomology of analytic varieties over $C$ is a relatively recent endeavour. It started in earnest with \cite{Sch13}, where Scholze   proved, among other things,  finiteness of $p$-adic \'etale cohomology of proper rigid analytic varieties. For compactly supported $p$-adic pro-\'etale cohomology of smooth partially proper varieties, it was not clear initially what the right definitions should be. There were two immediate  guesses: 
 \begin{align}\label{main-def1}
 \R\Gamma^{(1)}_{\proeet,c}(X,\Q_p) & :=(\R\lim_n\R\Gamma_{\proeet,c}(X,\Z/p^n))\otimes_{\Z_p}^{\LL_{\Box}}\Q_p;\\
 \text{or }\R\Gamma^{(2)}_{\proeet,c}(X,\Q_p) &: =[\R\Gamma_{\proeet}(X,\Q_p)\to \R\Gamma_{\proeet}(\partial X,\Q_p)],\notag\\
   & \hphantom{aaa} \text{with } \R\Gamma(\partial X,\Q_p) =\colim_{Z\in\Phi_X}\R\Gamma(X\setminus Z,\Q_p),\notag
 \end{align}
 where $\Phi_X$ denotes the set of quasi-compact opens in $X$ and    the brackets $[...]$ denote the mapping fiber. The first definition  is just a continuous version of $\Z/p^n$-cohomology.
  The heuristic for the second definition comes from the (possible) comparison with de Rham cohomology and the classical definition of van der Put of compactly supported version of the latter \cite{VdP92}. It came  to us as a surprise that this definition agrees with Huber's definition of compactly supported $p$-adic \'etale (sic !) cohomology (see Theorem \ref{main1} below): this follows from the  canonical quasi-isomorphism
  $$\R\Gamma^{(2)}_{\proeet,c}(X,\Z_p)_{\Q_p}\stackrel{\sim}{\to}\R\Gamma^{(2)}_{\proeet,c}(X,\Q_p).$$ The fact that this might be the right definition comes from the work on dualities which pairs this cohomology with $p$-adic pro-\'etale cohomology (see \cite{CGN}, \cite{CGN2}, \cite{ZLi}, \cite{ALBM}). Hence we set
  $$
   \R\Gamma_{\proeet,c}(X,\Q_p) :=\R\Gamma^{(2)}_{\proeet,c}(X,\Q_p).
  $$
  
  We note that, if we take \'etale instead of pro-\'etale in the first definition of \eqref{main-def1},  we obtain an isomorphic object. The indication that this is actually a  good definition of compactly supported $p$-adic \'etale cohomology (sic !) comes from two sources:
 \begin{enumerate}
 \item in \cite{CDHN}, it was shown that, for $p$-adic period domains, it  behaves like $\ell$-adic cohomology for $\ell\neq p$;
 \item in \cite{LLZ}, it was shown that, for  smooth almost proper varieties (i.e., $X=\overline{X}\setminus Z$, for a smooth  proper variety $\overline{X}$ and a closed Zariski subvariety $Z$), it  satisfies Poincar\'e duality vis a vis $p$-adic \'etale cohomology. 
 \end{enumerate}
 Hence we set
  $$
   \R\Gamma_{\eet,c}(X,\Q_p) :=\R\Gamma^{(1)}_{\proeet,c}(X,\Q_p).
  $$
It is a nonobvious fact but    there exists a canonical map $\R\Gamma_{\proeet,c}(X,\Q_p)\to \R\Gamma_{\eet,c}(X,\Q_p)$. 
 \subsection{Main results} The main result of the paper  can be summed up in the following  theorem-construction.  We write $\sd(\Q_{p,\Box})$ for the $\infty$-derived category 
  of solid $\Q_p$-modules and $\sff\sd(\Q_{p,\Box})$ for the corresponding filtered $\infty$-derived category.
    \begin{theorem}
 \label{main1}
 To any smooth partially proper rigid analytic  variety $X$ over $C$ there are  naturally  associated: 
   \begin{enumerate}
  \item The pro-\'etale cohomology with compact support  $ \R\Gamma_{\proeet,c}(X,\Q_p)\in \sd(\Q_{p,\Box})$. There is a natural quasi-isomorphism
  in $\sd(\Q_{p,\Box})$ with Huber's \'etale cohomology with compact support
  \begin{equation}\label{war11}
  \R\Gamma_{\proeet,c}(X,\Q_p)\simeq \R\Gamma_{\eet,c, {\rm Hu}}(X,\Q_p).
  \end{equation}
\item The syntomic cohomology with compact support $\rg_{\synt,c}(X,\Q_p(r))\in \sd(\Q_{p,\Box})$, $r\in \N$,  with a natural period morphism in $\sd(\Q_{p,\Box})$
 \begin{equation}
 \label{period-now}
 \alpha_r: \rg_{\synt,c}(X,\Q_p(r))\to \rg_{\proeet,c}(X,\Q_p(r)),
 \end{equation}
 which  is a quasi-isomorphism after truncation $\tau_{\leq r}$. 
 \item The de Rham cohomology $\R\Gamma_{\dr,c}(X)$ and $\B^+_{\dr}$-cohomology $\R\Gamma_{\dr,c}(X/\B^+_{\dr})$. These are filtered dg $C$- and $\B^+_{\dr}$-algebras, respectively, 
 such that in $\sff\sd(\Q_{p,\Box})$
 $$
 \R\Gamma_{\dr,c}(X/\B^+_{\dr})\otimes^{\LL_{\Box}}_{\B^+_{\dr}}C\stackrel{\sim}{\to}\R\Gamma_{\dr,c}(X).
 $$
 \item The Hyodo-Kato cohomology 
$\rg_{\hk,c}(X)$. This is  a dg $F^{\nr}$-algebra\footnote{We define $F^{\rm nr}$ as the maximal unramified extension of $F$ in $\ovk$.}
equipped with a Frobenius $\phi$ and a monodromy operator $N$. We have :
\begin{enumerate}
\item Natural Hyodo-Kato  quasi-isomorphisms in $\sd(C_{\Box})$ and $\sd(\B^+_{\dr,\Box})$, respectively:
\begin{align*}
  \iota_{\hk} &: \rg_{\hk,c}(X){\otimes}^{\LL_{\Box}}_{F^{\nr}}C\stackrel{\sim}{\to} \rg_{\dr,c}(X),\\
\iota_{\hk}&: \rg_{\hk,c}(X){\otimes}^{\LL_{\Box}}_{F^{\nr}}\B^+_{\dr}\stackrel{\sim}{\to} \rg_{\dr,c}(X/\B^+_{\dr}).
\end{align*}
\item {\rm (Poincar\'e duality)} If $X$ is everywhere of dimension $d$ then the natural pairing of solid $(\phi,N)$-modules over $F^{\rm nr}$
$$
H^i_{\hk}(X)\otimes^{\Box}_{F_L} H^{2d-i}_{\hk,c}(X)\stackrel{\cup}{\to} H^{2d}_{\hk,c}(X)\xrightarrow{{\rm tr}_{\hk}} F^{\rm nr}\{-d\}
$$
is a perfect duality, i.e., we have  induced isomorphisms of solid $(\phi,N)$-modules over~$F^{\rm nr}$
\begin{align*}
& H^i_{\hk}(X)\stackrel{\sim}{\to} \Hom_{F^{\rm nr}_{\Box}} (H^{2d-i}_{\hk,c}(X),F^{\rm nr}\{-d\}),\\
& H^i_{\hk,c}(X)\stackrel{\sim}{\to} \Hom_{F^{\rm nr}_{\Box}} (H^{2d-i}_{\hk}(X),F^{\rm nr}\{-d\}).
\end{align*}
These duality isomorphisms are compatible, via the Hyodo-Kato isomorphism,  with the de Rham duality.
\end{enumerate}
 \item A  distinguished triangle in  $\sd(\Q_{p,\Box})$
  \begin{equation}
  \label{triangle11}
 \rg_{\synt,c}(X,\Q_p(r))\lomapr{}  [\rg_{\hk,c}(X){\otimes}^{\LL_{\Box}}_{F^{\nr}}\B^+_{\st}]^{N=0,\phi=p^r}\lomapr{\iota_{\hk}} \rg_{\dr,c}(X/\B^+_{\dr})/F^r,
 \end{equation}
 which can be lifted to the derived category of Topological  Vector Spaces.  
 \item {\rm(Trace maps)} If $X$ is of dimension $d$ and irreducible, the syntomic trace map induces via \eqref{period-now}  a natural trace map
 $$
 {\rm tr}_{\proeet,X}:  H^{2d}_{\proeet,c}(X,\Q_p(d))\to \Q_p,
 $$
 that  is compatible with  Berkovich's  trace map  (via \eqref{war11}).
 \end{enumerate}
 \end{theorem}
  Some of the constructions in the above theorem follow in a rather straightforward, if technically tedious\footnote{Tedious especially in the functional analytic part.},  manner   from the ones for usual cohomologies. Others,  like the comparison with the Huber's definition, the definition of the trace maps, the Poincar\'e duality for Hyodo-Kato cohomology, and the related computations require subtler arguments. The duality for Hyodo-Kato cohomology is induced from the one for de Rham cohomology (a classical result) but the information about Frobenius and monodromy (as well as Galois action if needed) is obtained by a reduction (via excision) to the proper case. The pro-\'etale trace map is defined by passing to the syntomic cohomology via period quasi-isomorphism and then using the Hyodo-Kato trace map (via  the distinguished triangle \eqref{triangle11}). That it agrees with Berkovich's trace map can be checked by a reduction to the algebraic case (again, via excision) where this amounts to the compatibility of the algebraic $p$-adic comparison morphism with traces.

    From Theorem \ref{main1}, we derive the following key result: 
 \begin{corollary} \label{main3} Let $X$ be a smooth partially proper  rigid analytic  variety over $C$. 
 \begin{enumerate}
 \item  {\rm ({\rm Basic comparison theorem})}
  \label{basic}
 Let $r\geq 0$. There is a natural  quasi-isomorphism {\rm (period isomorphism)} in $\sd(\Q_{p,\Box})$:
  \begin{align}
  \label{cos1}
\tau_{\leq r}&\rg_{\proeet,c}(X,\Q_p(r))\simeq \tau_{\leq r}  
\big[[\rg_{\hk,c}(X){\otimes}^{\LL_{\Box}}_{F^{\nr}}\B^+_{\st}]^{N=0,\phi=p^r}\lomapr{\iota_{\hk}}\rg_{\dr,c}(X/\B^+_{\dr})/F^r\big].
  \end{align}
 \item  {\rm ({\rm Fundamental diagram})} If $X$ is   Stein  of dimension $d$ over $K$ and $r\geq 0$, then,  under a strong  assumption on the slopes of Frobenius on Hyodo-Kato cohomology (see Theorem \ref{diagSYNc} for details), there is a map of exact sequences of solid $\Q_p$-modules
{\small\[ \xymatrix{
0 \ar[r] & H_c^d(X_C, \Omega^{r-d-1})/ \kker d(d) \ar[r] \ar@{=}[d] & H_{\synt,c}^{r}(X_C, \Q_p(r)) \ar[d] \ar[r] & (H^{r}_{\HK,c}(X_C) \solid_{F^{\rm nr}} t^d \Bstp)^{\varphi=p^{r}, N=0} \ar[r] \ar[d] & 0 \\
0 \ar[r] & H_c^d(X_C, \Omega^{r-d-1})/ \kker d(d) \ar[r] & H_c^d(X_C, \Omega^{r-d})^{d=0}(d)  \ar[r] & H^r_{\dr,c}(X_C)(d) \ar[r] & 0 
}\]}
 \end{enumerate}
 \end{corollary}
The Frobenius slope condition above is very restrictive  but it holds for tori and Drinfeld spaces, for example (probably also for general $p$-adic period domains) so we can use the fundamental diagram to compute compactly supported cohomology in these cases. Recall that for the usual cohomology we have a fundamental diagram with no restrictions on the slopes\footnote{This means that the authors of \cite{CDN1} got very lucky that they started computations with the usual cohomology and not with the compactly supported one as is customary  in the local Langlands program.}.

 Finally, the paper includes a number of concrete examples of computations of compactly supported $p$-adic pro-\'etale cohomology: Drinfeld spaces and tori (already mentioned above) but also affine spaces, Stein curves, and general Stein spaces (the last one -- up to a degree). 
 \begin{remark}
 \begin{enumerate}
 \item We also prove  analogs of  Theorem \ref{main1} and Corollary \ref{main3}  for  smooth overconvergent varieties and compare them with their rigid analytic versions as in \cite{CN4}. Here, again, it was not clear what should be the right definition of compactly supported $p$-adic pro-\'etale cohomology of an overconvergent variety and our definition is inspired, as in the rigid analytic case, by the definition of compactly supported de Rham cohomology of dagger affinoids of Grosse-Kl\"onne from \cite{GK1}.
 \item For a  further study of properties of compactly supported $p$-adic cohomologies of rigid analytic spaces see the work of Xinyu Shao \cite{Shao}. For a formulation  of $C_{\st}$ conjecture  in the same context   see the work of  Colmez-Gilles-Nizio{\l} \cite{Haw}. 
 \end{enumerate}
 \end{remark}
\begin{acknowledgments} This paper is a continuation of a joint project of W.N. with Pierre Colmez to whom we are very greatful for the inspiration and endless discussions.  We would like to thank Gabriel Dospinescu, 
Zhenghui Li,  Xinyu Shao,  and Bogdan Zavyalov  for  helpful comments and discussions concerning the content of this paper and its earlier drafts. S.G. 
would like to thank the MPIM of Bonn and the IAS of Princeton for their support and hospitality
during the academic years 2022-2023 and 2023-2024, when parts of this paper were written. W.N. would like to thank 
the Banach Center in Warsaw for its support and hospitality  during the Spring semester of 2022, when parts of this paper were written. 
\end{acknowledgments}
  \subsubsection*{Notation and conventions.}\label{Notation}
Let $\so_K$ be a complete discrete valuation ring with fraction field $K$  of characteristic 0 and with perfect
residue field $k$ of characteristic $p$. Let $\ovk$ be an algebraic closure of $K$ and let $\so_{\ovk}$ denote the integral closure of $\so_K$ in $\ovk$. Let $C=\wh{\ovk}$ be the $p$-adic completion of $\ovk$.  Let
$W(k)$ be the ring of Witt vectors of $k$ with 
 fraction field $F$ (i.e., $W(k)=\so_F$); let $e=e_K$ be the ramification index of $K$ over $F$.   Set $\sg_K=\Gal(\overline {K}/K)$ and 
let $\phi$ be the absolute
Frobenius on $W(\overline {k})$. 
We will denote by $\A_{\crr}, \B_{\crr}, \B_{\st},\B_{\dr}$ the crystalline, semistable, and  de Rham period rings of Fontaine. 

 We will denote by $\so_K$,
\index{OK@\OK}$\so_K^{\times}$, and $\so_K^0$, depending on the context,  the scheme $\Spec ({\so_K})$ or the formal scheme $\Spf (\so_K)$ with the trivial, the canonical  (i.e., associated to the closed point), and the induced by $\N\to \so_K, 1\mapsto 0$,
log-structure, respectively.  Unless otherwise stated all   formal schemes are $p$-adic, locally of finite type, and equidimensional. For a ($p$-adic formal) scheme $X$ over $\so_K$, let 
\index{Xn@\Xn}$X_0$ denote
the special fiber of $X$; let $X_n$ denote its reduction modulo $p^{n+1}$.

All rigid analytic spaces considered will be over $K$ or $C$. 
We assume that they are separated, taut\footnote{That is,   for  all quasi-compact  opens 
$V$ of  $X$,  the closure $\overline{V}$ of  $V$ in  $X$ is quasi-compact.}, and countable at infinity. 
If $L=K,C$, we \index{Smo@\Smo}let ${\rm Sm}_L$ (resp.~${\rm Sm}_L^\dagger$)
be the category of smooth rigid analytic (resp.~dagger) varieties over $L$, and 
we denote \index{Perf@\Perf}by ${\rm Perf}_C$ the category of perfectoid spaces over $C$.

Unless otherwise stated, we work in the derived (stable) $\infty$-category 
\index{Dcal@\Dcal}$\sd(A)$ of left-bounded complexes of a quasi-abelian category $A$ (the latter will be clear from the context).
  Many of our constructions will involve (pre)sheaves of objects from $\sd(A)$. 
  We will use a shorthand for certain homotopy limits:
 if $f:C\to C'$ is a map  in the derived $\infty$-category of a quasi-abelian  category, we set
$$[\xymatrix{C\ar[r]^f&C'}]:=\lim(C\to C^{\prime}\leftarrow 0).$$ 
For an operator $F$ acting on $C$, we will use the brackets $[C ]^F$ to denote the derived eigenspaces and the brackets $(C)^F$ or simply $C^F$  to denote the non-derived ones. 
We will also write
\[ \left[\vcenter{\xymatrix @C=2cm{ 
C_1\ar[r]^-{f_1} \ar[d] & K_1 \ar[d] \\
C_2 \ar[r]^-{f_2} & K_2}} \right] := [[C_1 \xrightarrow{f_1} K_1] \to [C_2 \xrightarrow{f_2} K_2]]. \]

 For functional analysis, we will work in the context of condensed mathematics. We write $\sd(\Q_{p,\Box})$ for the $\infty$-derived category of solid $\Q_p$-modules; (a bit abusively) we will write $\Q_{p,\Box}$ for the category of solid $\Q_p$-modules. We will use similar notation for modules over a solid $\Q_p$-ring etc. 

 \section{Pro-\'etale cohomology with compact support}\label{proetale-c}

Let $\ell$ be a prime number. In this chapter, we define compactly supported $\ell$-adic pro-\'etale cohomology for partially proper smooth rigid analytic varieties and smooth dagger varieties over $K$ or $C$.  This  cohomology was  used in  \cite{CGN} to obtain a Poincar\'e duality for arithmetic $p$-adic pro-\'etale cohomology of rigid analytic Stein curves. 
\subsection{Rigid analytic compact support functor}\label{compact1} We start with a preliminary discussion. 
Let  $L= K, C$ and let   $X$ be a smooth rigid analytic variety over $L$. Let $\Phi$ denote the set of  quasi-compact opens of $X$. We write $\partial X$ for the 
pro-object $\{X\setminus Z\}_{Z\in\Phi}$ indexed by $\Phi^{\rm op}$. We note that $X\setminus Z$ is an admissible open of $X$ (see \cite[Lemma 1.1]{VdP92}). For a presheaf $F$ on smooth rigid analytic varieties with values in an $\infty$-category $\scc$ admitting filtered colimits, we set
$$
F(\partial X):=\colim_{Z\in\Phi} F(X\setminus Z).
$$
We have a natural restriction map $F(X)\to F(\partial X)$. Suppose that $\scc$ is stable and set 
$$
F_c(X):=[F(X)\to F(\partial X)],
$$
so that we have a distinguished triangle
$$
F_c(X)\to F(X)\to F(\partial X). 
$$
Note that if $X$ is quasi-compact then $\Phi$ admits a final object $X$ and if, moreover, $F(\emptyset )$ is the final object, then we have 
$F_c(X)\stackrel{\sim}{\to}F(X)$.

 Suppose now that $F$ is a sheaf. Let $U\subset X$ be an open. For every $Z\in\Phi$ whose closure is contained in $U$, i.e., $Z\Subset_XU$, covering $X$ by $X\setminus Z$ and $U$, we have a cartesian square
 $$
 \xymatrix{
 F(X)\ar[r]\ar[d] & F(X\setminus Z)\ar[d]\\
  F(U)\ar[r] & F(U\setminus Z)
 }
 $$
 and hence, setting $F_Z(-):=[F(-)\to F(-\setminus Z)]$, we get the excision
 \begin{equation}
 \label{excision1}
 F_Z(X)\stackrel{\sim}{\to}F_Z(U).
 \end{equation}
 
  Now suppose that $j:U\to X$ is partially proper, so that for every quasi-compact open $Z$ of $U$, the closure of $Z$ in $X$ is also contained in $U$. Letting $Z$ run over $\Phi_U$ (quasi-compact opens for $U$), we get the pushforward map 
  $$j_*:  F_c(U)\to  F_c(X)$$ as the composition:
  $$
  F_c(U)=\colim_{Z\in\Phi_U}F_Z(U)\simeq \colim_{Z\in\Phi_U}F_Z(X)\to \colim_{Z\in\Phi}F_Z(X)=F_c(X).
  $$
  We used here \eqref{excision1}. Clearly, the map $j_*$ is functorial in $F$. 

\subsection{Rigid analytic setting}
Following the same process as in Section \ref{compact1}, we define compactly supported (pro-)\'etale cohomology of a smooth rigid analytic variety $X$ over $K$ or $C$. We give the definition for torsion coefficients as well as $\ell$-adic ones. We show that our construction coincides with the one of Huber (\cite{Hub98}) in the case of partially proper varieties.

\subsubsection{Definition}

Let  $L=K,C$. Let  $X$ be a smooth rigid analytic variety over $L$. We define compactly supported cohomology with torsion coefficients of $X$ by the same formulas as in Section \ref{compact1}, i.e.,
\begin{align*}
 \R \Gamma_{*,c}(X, \Z/\ell^n\Z) & := [ \R \Gamma_{*}(X, \Z/\ell^n\Z) \to \R \Gamma_{*}(\partial X, \Z/\ell^n\Z)],
\end{align*}
where $*=\eet,\proeet$ and 
\begin{align*}
 \R \Gamma_{*}(\partial X, \Z/\ell^n \Z) & := \colim_{Z\in\Phi} \R \Gamma_{*}(X \setminus Z, \Z/\ell^n\Z). 
\end{align*}
We then define the integral cohomologies by
\begin{align*}
\R \Gamma_{\proeet,c}(X, \Z_\ell) & := [ \R \Gamma_{\proeet}(X, \Z_{\ell}) \to \R \Gamma_{\proeet}(\partial X, \Z_\ell)],\\
 \R \Gamma_{\eet,c}(X, \Z_\ell)  &  :=\R\lim_n  \R \Gamma_{\eet,c}(X, \Z/\ell^n\Z). 
\end{align*}
We obtain complexes in ${\cal D}(\Z_{\ell,\Box})$.

\begin{remark}
If we define
\[ \R\Gamma_{\eet,c, {\rm naive}}(X, \Z_{\ell}):= [ \R \Gamma_{\eet}(X, \Z_{\ell}) \to \R \Gamma_{\eet}(\partial X, \Z_{\ell})],\]
the natural map \[ \colim_{Z} \R \lim_n \R \Gamma_{\eet}(X \setminus Z, \Z/\ell^n\Z) \to \R \lim_n \colim_{Z} \\R \Gamma_{\eet}(X \setminus Z, \Z/\ell^n\Z)\]  induces a map
$\R \Gamma_{\eet,c,{\rm naive}}(X, \Z_{\ell}) \to \R \Gamma_{\eet,c}(X, \Z_{\ell})$. In general, it is not a quasi-isomorphism. 
\end{remark}

For the rational theory, we set:
\begin{align*}
\R \Gamma_{\eet,c}(X, \Q_\ell) &:= \R \Gamma_{\eet,c}(X, \Z_\ell) \dsolid_{\Z_{\ell}}\Q_\ell,\\
\R \Gamma_{\proeet,c}(X, \Q_\ell) & := [ \R \Gamma_{\proeet}(X, \Q_{\ell}) \to \R \Gamma_{\proeet}(\partial X, \Q_\ell)]. 
\end{align*}
These are complexes in ${\cal D}(\Q_{\ell,\Box})$.

\begin{lemma}
\label{proet-et-entier}
\begin{enumerate}
\item The natural morphism $\R \Gamma_{\eet,c, {\rm naive}}(X, \Z_{\ell}) \to \R \Gamma_{\proeet,c}(X, \Z_{\ell})$ is a quasi-isomorphism in ${\cal D}(\Z_{\ell,\Box})$.
\item It induces a natural change of topology morphism in  ${\cal D}(\Z_{\ell,\Box})$:
\begin{equation}\label{niedziela1}
\R \Gamma_{\proeet,c}(X, \Z_\ell) \to \R \Gamma_{\eet,c}(X, \Z_\ell).
\end{equation}
 \end{enumerate}
\end{lemma}
\begin{remark}If $X$ is partially proper, we will show below (see Theorem \ref{thm:Huber-Qp}) that the natural morphism in $\sd(\Q_{\ell,\Box})$
$$
\R \Gamma_{\proeet,c}(X, \Z_\ell)\otimes_{\Z_\ell}^{\LL_{\Box}}\Q_\ell\to \R \Gamma_{\proeet,c}(X, \Q_\ell)
$$
is a quasi-isomorphism. It follows that the map \eqref{niedziela1} induces a natural morphism
in $\sd(\Q_{\ell,\Box})$:
$$ 
 \R \Gamma_{\proeet,c}(X, \Q_\ell) \to \R \Gamma_{\eet,c}(X, \Q_\ell). $$
\end{remark}
\begin{proof}
We first show that we have a quasi-isomorphism:
\begin{equation}
\label{banach1}
f_Y: \R \Gamma_{\eet}(Y, \Z_{\ell}) \xrightarrow{\sim} \R \Gamma_{\proeet}(Y, \Z_{\ell}),
 \end{equation}
for any smooth  rigid analytic variety $Y$.  For that, choose a covering of $Y$ by an increasing sequence of  quasi-compact opens $\{U_i\}_{i\in\N}$. Since
$$
\R \Gamma_{*}(Y, \Z_{\ell})\simeq \R\lim_i\R\Gamma_{*}(U_i,\Z_{\ell}), 
$$
 it suffices to show  the quasi-isomorphism  (\ref{banach1}) for  a quasi-compact $Y$. 
 
   Assume thus that $Y$ is quasi-compact. Then we have 
\[ \R \Gamma_{\eet}(Y, \Z_{\ell})= \R \lim_n \R \Gamma_{\eet}(Y, \Z/\ell^n) \xrightarrow{\sim}  \R \lim_n \R \Gamma_{\proeet}(Y, \Z/\ell^n) \]
using \cite[(3.17)]{Sch13}. Since the functors  $\R \Gamma$ and  $\R \lim$ commute, we have 
\[ \R \lim_n \R \Gamma_{\proeet}(Y, \Z/\ell^n) \xleftarrow{\sim} \R \Gamma_{\proeet}(Y, \R \lim_n \Z/\ell^n) \xrightarrow{\sim} \R \Gamma_{\proeet}(Y,  \Z_{\ell}), \]
where the last  quasi-isomorphism follows from \cite[(8.2)]{Sch13}. Hence we get  the quasi-isomorphism (\ref{banach1}), as wanted.

  By the definitions of $\rg_{\proeet,c}(X, \Z_{\ell})$ and $\rg_{\eet,c, {\rm naive}}(X, \Z_{\ell})$, this induces a quasi-isomorphism: 
\[\rg_{\proeet,c}(X, \Z_{\ell}) \xleftarrow[f]{\sim} \rg_{\eet,c, {\rm naive}}(X, \Z_{\ell}) \]  
and composing with the natural morphism $\rg_{\eet,c, {\rm naive}}(X, \Z_{\ell}) \to \rg_{\eet,c}(X, \Z_{\ell})$, we obtain the morphism from the second claim of the lemma.
 \end{proof}

The above  cohomologies satisfy all the properties from Section~\ref{compact1}, i.e.,  for  $\Lambda \in \{\Z/\ell^n\Z,\Z_\ell,\Q_\ell\}$, we have (here $*=\eet, \proeet$): 
\begin{enumerate}
\item \label{tri-dist} We have the distinguished triangles: 
\[ \R \Gamma_{*,c}(X, \Lambda) \to \R \Gamma_{*}(X, \Lambda) \to \R \Gamma_{*}(\partial X, \Lambda). \]
\item \label{propre} If  $X$ is  quasi-compact, we have quasi-isomorphisms: 
\[ \R \Gamma_{*,c}(X, \Lambda) \xrightarrow{\sim} \R \Gamma_{*}(X, \Lambda). \]
In particular, if $X$ is  proper its  cohomology with compact  support  coincides with the usual cohomology. 
\item \label{jstar} Let $j:U\hookrightarrow X$ is a relatively  partially proper open subspace of $X$, we have a canonical map $$j_*: \quad \R\Gamma_{*,c}(U, \Lambda) \to \R\Gamma_{*,c}(X, \Lambda) .$$ 
\end{enumerate}

\begin{remark} As we have seen above, 
if  $X$ is  quasi-compact, the cohomology of the boundary of $X$ is  trivial, {i.e.},  
 $\R \Gamma_{\eet}(\partial X, \Lambda) =0$. Hence, in this case,  our definition is not the correct one. For torsion coefficients,    the \'etale cohomology with compact support of Huber    $H^i_{\eet,c, {\rm Hu}}(X, \Z/\ell^n\Z)$ gives the right groups. For example, 
 if  $\ell\neq p$ and  $X=\Sp(C\{T\})$ is the closed unit ball, we have  (see \cite[Ex 0.4.6]{Hub98}) :
\[ H^i_{\eet, c, {\rm Hu}}(X, \Z/\ell^n)\simeq \begin{cases} 0 & \text{ if } i \neq 2, \\ \Z/ \ell^n & \text{ if } i=2. \end{cases} \]
\end{remark}
\subsubsection{Comparison with the  cohomology of Huber:  integral coefficients}

Compactly supported \'etale cohomology was previously defined by Huber in \cite{Huber}, \cite{Hub98}. The purpose of this section is to prove the following comparison theorem.
\begin{theorem}
\label{Hu-proet1} Let $L=K,C$ and let $\ell$ be a prime number. Assume that  $X$ is partially proper over $L$. Then there exist natural quasi-isomorphisms in $\sd(\Z_{\ell,\Box})$
\[ \R \Gamma_{\eet, c, {\rm Hu}}(X, \Z_\ell) \xrightarrow{\sim} \rg_{\eet,c, {\rm naive}}(X, \Z_{\ell}) \xrightarrow{\sim} \R \Gamma_{\proeet,c}(X, \Z_\ell). \]  
\end{theorem}
Let us start with reviewing the definition of cohomology with compact support of Huber from \cite[Ch. 5]{Huber} (and extended to the $\ell$-adic case in  \cite{Hub98}). 
\vskip.2cm
(i) {\em Torsion coefficients.}
Let  $A$ be a ring and let $f: X \to Y$ be a separated morphism  of separated  rigid analytic varieties. 

\begin{definition}
\label{gamma_c}
For   $\mathcal{F}$, a sheaf of $\mode(X_{\eet}-A)$, we define $\Gamma_c(X/Y, \mathcal{F})$ in the following way:
\begin{enumerate}
\item If  $X$ and  $Y$ are qcqs, we choose an integral model $\tilde{f}: \XF \to \YF$ of $f$ and set:
\begin{align*}
\Gamma_c(X/Y, \mathcal{F}) := \{ s \in \Gamma(X, \mathcal{F}) \; | \; \text{ there exists closed }\ZF \subset \XF \text{  such that } \ZF \to \YF \text{ is  proper  }\\ \text{ and  } s|_{{\rm sp}^{-1}(\XF \setminus \ZF)}=0 \},
\end{align*}
where  ${\rm sp} : X \to \XF$ is the specialization morphism. 

\item If  $Y$ is  qcqs, we set:
\begin{align*}
\Gamma_c(X/Y, \mathcal{F}) := \{ s \in \Gamma(X, \mathcal{F}) \; |& \; \text{there exists } \{U,V\}, \text{ an admissible open covering of } X, \\& \text{ such that  } U \text{ is quasi-compact, } s|_U \in \Gamma_c(U/Y, \mathcal{F}|_U) \text{ and,  } s_{V}=0 \}.
\end{align*}

\item In the general case, we set:
\begin{align*} 
\Gamma_c(X/Y, \mathcal{F}) := \{ s \in \Gamma(X, \mathcal{F}) \; |& \; \text{for all } W \subset Y \text{ admissible open  qcqs,} \\&  s|_{f^{-1}(W)} \in \Gamma_c(f^{-1}(W)/W, \mathcal{F}|_{f^{-1}(W)}) \}.
\end{align*}
\end{enumerate}
\end{definition}

If  $X$ is a rigid analytic variety over $C$, we write simply  $\Gamma_c(X, \mathcal{F})$ pour $\Gamma_c(X/ \Sp(C), \mathcal{F})$.
Huber defined the functor $f_{!} : \mode(X_{\eet}-A) \to \mode(Y_{\eet}-A)$ by setting  \[ f_{!} \mathcal{F}(U):= \Gamma_c(X \times_Y U/U, \mathcal{F}), \] for an \'etale open $U$ of $Y$ (see \cite[(5.6.2)]{Huber}).

\begin{remark}
\label{gamma-r(f)}
If  $X\to Y$ is a morphism of adic spaces, we write  $\Gamma_c(X/Y, \mathcal{F})$ for the set of global sections of   $\FC$ whose support is  proper over $Y$. 
If now  $X \to Y$ is a morphism of  rigid analytic varieties as above and   $r(X) \to r(Y)$ is the morphism of associated adic spaces, we have an isomorphism  (see \cite[(5.6.4)]{Huber}) :
\[ \Gamma_c(X/Y, \FC) \xrightarrow{\sim} \Gamma_c(r(X)/r(Y), \rho^*\FC), \]
where $\rho$ is the natural morphism $r(X)_{\eet} \to X_{\eet}$. 
\end{remark}

\vskip.2cm
(ii) {\em $\ell$-adic coefficients.}
Denote by  $\mode(X_{\eet}-\Z_{\ell}^{\bullet})$ the category of projective systems $(\FC_n)_n$ of  $\Z_{\ell}$-modules on  $X_{\eet}$ such that, for all $n\in \N$, $\ell^n\FC_n=0$. Consider  $X\to Y:= \Spa(L, \so_L)$, a separated  adic space locally of finite type\footnote{In fact, Huber gives a definition for more general morphisms  $X\to Y$, which he calls \emph{locally  of $^+$weakly finite type.}} and  assume  that  $X$ is taut. If  $X$ is  partially proper, we define $\R \Gamma_{\eet, c, {\rm Hu}}(X,-)$ as the right derived functor 
of the functor
$$(\FC_n)_n \mapsto \Gamma_c(X_{\eet}, \varprojlim_n \FC_n): \mode(X_{\eet}-\Z_{\ell}^{\bullet})\to \mode(\Z_{\ell}).
$$ In the general case, there exists a factorization
\[ \xymatrix{ X \ar@{^{(}->}[r]^i \ar[d] & \overline{X} \ar[dl] \\ Y} \] 
where  $i$ is an open immersion and  $\overline{X} \to Y$ is partially proper. We then define
\[ \R \Gamma_{\eet, c, {\rm Hu}}(X,(\FC_n)_n):=\R \Gamma_{\eet,c, {\rm Hu}}(\overline{X}, (i_{!}\FC_n)_n), \]
 where  $i_! \sg$ is the functor $U \mapsto \Gamma_c(X \times_YU, \sg)$. 

For  $X$ a separated  rigid analytic variety over  $C$, we define $\R \Gamma_{\eet, c, {\rm Hu}}(X, (\FC_n)_n):= \R \Gamma_{\eet, c, {\rm Hu}}(r(X),(\rho^*\FC_n)_n)$. If  $X$ is partially proper, $\R \Gamma_{\eet,c, {\rm Hu}}(X,-)$ is equal to the right derived functor  of the functor  $(\FC_n)_n \mapsto \Gamma_c(X_{\eet}, \varprojlim_n \FC_n)$. 

We will need the following well-known technical result: 

\begin{lemma}
\label{inj-res}
Let  $(\mathcal{F}_n)_n\in \mode(X_{\eet}-\Z_{\ell}^{\bullet})$. There exists a monomorphism $(\FC_n)_n \hookrightarrow (\IC_n)_n$ from  $(\FC_n)_n$ to $(\IC_n)_n$,  an object of    $\mode(X_{\eet}-\Z_{\ell}^{\bullet})$ of the form
\[ \IC_n = \prod_{s=1}^n \JC_s, \]
where the maps  $\IC_n \to \IC_{n-1}$ are given by the canonical projections  and $\JC_s$ are injective sheaves of $\Z/\ell^s$-modules  on $X_{\eet}$. 

Moreover,  for all $n$, the sheaf $\IC_n$ is an injective  object of $\mode(X_{\eet}-\Z/\ell^n)$ and the canonical map $\FC_n \to \IC_n$ is a  monomorphism.
\end{lemma}

\begin{proof}
We will recall here the proof from  \cite[(1.1)]{Jann88}. For all $s$, we choose a monomorphism $\FC_s \to \JC_s$ of  $\FC_s$ into an injective object of the category of sheaves of $\Z/\ell^s$-modules and we define the projective system 
 $(\IC_n)_n$ as in the statement of the lemma. Since the $\JC_s$'s are injective, $(\JC_s)_{s\in \N}$ is an injective object of the product category 
\[ C_{\ell}:= \prod_{s \in \N} \mode(X_{\eet}-\Z/\ell^s) \]
and  $(\FC_s)_s \to (\JC_s)_s$ is a  monomorphism in that category. The functor
\[ P :  C_{\ell}  \to  \mode(X_{\eet}-\Z_{\ell}^{\bullet}),\quad (\AC_s)_s  \mapsto (\prod_{s=1}^n \AC_s)_n  \] 
admits left adjoint given by the forgetful functor  
\[ V : \mode(X_{\eet}-\Z_{\ell}^{\bullet}) \to C_{\ell}. \]
Since  $V$ is exact and faithful, the functor $P$ preserves monomorphisms and injective objects. It follows that  $(\IC_n)_n:=P((\JC_s)_s)$ is an injective object and we have a monomorphism  $(\FC_n)_n\to (\IC_n)_n$. 

Conversely, for all  $m\in\N$, the functor 
 \[ V_m :  \mode(X_{\eet}-\Z_{\ell}^{\bullet}) \to  \mode(X_{\eet}-\Z/\ell^m),\quad   (\AC_n)_n  \mapsto \AC_m \]
 also admits an exact  faithful left adjoint: 
\[ U_m :  \mode(X_{\eet}-\Z/\ell^m)  \to  \mode(X_{\eet}-\Z_{\ell}^{\bullet}), \text{ where } U_m((\AC_n)_n)=\begin{cases} 0 & \text{ if } n < m, \\ \AC_m & \text{ if } n \ge m. \end{cases}  \] 
It follows that  $V_n$ preserves monomorphisms  and injectives. 
Hence,  we have obtained that  $\FC_n \to \IC_n= V_n((\IC_m)_m)$ is a monomorphism of  $\FC_n$ into an injective sheaf  $\IC_n$, as wanted.  
\end{proof}
\begin{remark}
\label{inj-res1}
This lemma allows us to construct injective resolutions $(\FC_n)_n \to (\IC_n)^{\bullet}_n$ such that the induced sequences $\FC_n \to \IC_n^{\bullet}$ are also injective resolutions. 
\end{remark}
We can now prove the following proposition:

\begin{proposition}{\rm (Huber, {\cite[(5.3.6)]{Hub98}})}
\label{Hu-!pi}
There exists an isomorphism of exact functors  from  $D^+(\mode(X_{\emph{\text{\'et}}}- \Z_{\ell}^{\bullet}))$ to  $D^+(\mode(\Z_{\ell}))$ : \[ \R \Gamma_{\eet, c, {\rm Hu}} \simeq  \R \Gamma_{!} \circ \R  \pi_*, \]
where $\pi_*$ is the discretization functor
$$
\pi_*: \mode(X_{\emph{\text{\'et}}}- \Z_{\ell}^{\bullet})\to \mode(X_{\emph{\text{\'et}}}- \Z_{\ell}),\quad (\sff_n)_n\mapsto \lim_n\sff_n
$$
and  $\Gamma_{!}$ is the functor  
$$\Gamma_{!}: \mode(X_{\emph{\text{\'et}}}-\Z_{\ell})\to \mode(\Z_{\ell}),\quad \FC \mapsto \Gamma_c(X_{\emph{\text{\'et}}}, \FC).
$$ 
\end{proposition}
\begin{proof}
By definition, $\R \Gamma_{\eet, c, {\rm Hu}}(X,-)$ is equal to the functor $\R (\Gamma_c \circ \pi_*)$. Take a resolution  $(\FC_n)_n \to (\IC_n)^{\bullet}_n$ given by Remark  \ref{inj-res1}.  
The functor $\pi_*$  is given by  $(\FC_n)_n \mapsto \lim_n \FC_n$, hence we have 
\[ \pi_*((\IC^{\bullet}_n)_n)= \prod_{s \in \N} \JC^{\bullet}_s. \] This yields
$$
\R (\Gamma_c \circ \pi_*)((\sff_n)_n)\simeq \Gamma_c(X_{\eet},\prod_{s \in \N} \JC^{\bullet}_s). 
$$

  But, by  \cite[(5.2.6)]{Huber}, every term of the complex $\prod_{s \in \N} \JC^{\bullet}_n$ is  $\Gamma_!$-acyclic, hence we have 
  $$
  (\R \Gamma_{!} \circ \R  \pi_*)((\sff_n)_n)\simeq\R \Gamma_{!}(\prod_{s \in \N} \JC^{\bullet}_s) \simeq \Gamma_{!}(\prod_{s \in \N} \JC^{\bullet}_s)= \Gamma_{c}(X_{\eet}, \prod_{s \in \N} \JC^{\bullet}_s).
  $$
  We have obtained 
   that   $\R (\Gamma_! \circ \pi_*)\simeq \R \Gamma_{!} \circ \R  \pi_*$, as wanted.
\end{proof}

Later on we will need the following lemma:

\begin{lemma}
\label{pi-holim}
For  $(\mathcal{F}_n)_n$ in  $\mode(X_{\eet}-\Z_{\ell}^{\bullet})$, there exists a quasi-isomorphism
\[ \R \Gamma_{\eet}(X, \R \pi_*(\FC_n)_n) \xrightarrow{\sim} \R \lim_n \R \Gamma_{\eet}(X, \FC_n).\]
\end{lemma}

\begin{proof}
As in the proof of the previous proposition, take an  injective resolution  $(\FC_n)_n \to (\IC_n)^{\bullet}_n$ given by Remark \ref{inj-res1}. We have  $\pi_*((\IC_n)^{\bullet}_n)= \prod_{s \in \N} \JC^{\bullet}_s$. But  each term of the complex $\prod_{s \in \N} \JC^{\bullet}_s$ is  $\Gamma$-acyclic and  $\Gamma$ commutes with countable products, hence we get:   
\[ \R \Gamma_{\eet}(X, \R \pi_*(\FC_n))\simeq  \prod_{s \in \N} \Gamma_{\eet}(X, \JC^{\bullet}_s). \]
Remark  \ref{inj-res1} also gives us, for all $n$, an injective resolution  $\FC_n \to \IC^{\bullet}_n$ of  $\FC_n$. Hence
\[ \R \lim_n \R \Gamma_{\eet}(X, \FC_n)\simeq \R \lim_n \Gamma_{\eet}(X, \IC^{\bullet}_n)\simeq  \prod_{s \in \N} \Gamma_{\eet}(X, \JC^{\bullet}_s), \]
which proves our lemma.
\end{proof}

\vskip.2cm
(iii) {\em Comparison theorem.}
Let  $L=K,C$.  
Assume in this section that  $X$ is partially proper over $L$. For  a quasi-compact open $Z$ and  a sheaf of  $\Z_{\ell}$-module $\mathcal{F}$ on   $X_{\eet}$, define
\[ \Gamma_Z(X_{\eet}, \mathcal{F}):= \kker( \Gamma(X_{\eet}, \mathcal{F}) \to \Gamma((X \setminus Z)_{\eet}, \mathcal{F}) ) \]
and  $H^i_Z(X_{\eet}, \mathcal{F}) := H^i \R \Gamma_Z(X_{\eet}, \mathcal{F})$ the associated cohomology. 

\begin{lemma}{\rm (Huber, {\cite[Prop. 5.6.5]{Huber}})}
\label{c-limZ}
For all $\mathcal{F}$ in  $\mode(X_{\emph{\text{\'et}}}-\Z_{\ell})$, we have 
\[ \Gamma_c(X_{\emph{\text{\'et}}}, \mathcal{F})= \{ s \in \Gamma(X_{\emph{\'et}}, \mathcal{F}) \; | \; \exists Z \subset X \text{ an admissible quasi-compact open  such that  } s|_{X \setminus Z}=0 \}. \]  In particular  $\Gamma_c(X_{\eet}, \mathcal{F})= \colim_{Z \in \Phi} \Gamma_Z(X_{\eet}, \mathcal{F})$. 
\end{lemma}

\begin{proof}
If  $s$ is in  $\Gamma_c(X_{\eet}, \mathcal{F})$ then, by definition (see \ref{gamma_c}), there exists an admissible open covering  $(Z, V)$, with $Z$ quasi-compact open,  such that  $s|_V=0$. Since $V\supset X\setminus Z$, we have 
 $s|_{X \setminus Z}=0$ as well.

Conversely, let  $s$ be such that  $s|_{X \setminus Z}=0$,  for   $Z$ admissible open  and quasi-compact. Let  $X= \bigcup_{i \in I} X_i$ be an admissible covering of  $X$ with $X_i$ open affinoid  and such that,  for all  $i$, there exists $\tilde{X}_i$ an open affinoid 
of  $X$ with $X_i \Subset  \tilde{X}_i$. Since  $Z$ is quasi-compact, modulo reindexing, we may assume that $Z \subset X_1 \cup X_2 \cup \dots \cup X_n$, for a constant $n$. Set  $V:=X \setminus (X_1 \cup \dots \cup X_n)$ and  $\tilde{V}:= \tilde{X}_1 \cup \dots \cup \tilde{X}_n$. Then  $\{V, \tilde{V} \}$ is an admissible open  covering of  $X$ and  $s|_{V}=0$. 

   Take   a formal model $\tilde{\VF}$ of  $\tilde{V}$ and the opens  $\XF_i$ and  $\tilde{\XF}_i$ of  $\tilde{\VF}$ with  $\XF_i \subset \tilde{\XF}_i$ and  $X_i= {\rm sp}^{-1}(\XF_i)$, $\tilde{X}_i= { \rm sp}^{-1}(\tilde{\XF}_i)$
    (where  ${\rm sp} :  \tilde{V} \to \tilde{\VF}$ is the specialization map). If we write  $\bar{\XF}_i$ for the closure of  $\XF_i$ in  $\tilde{\XF}_i$ then $\bar{\XF}_i$ is proper over  $\so_L$ (by \cite[2.5]{Lutk90}). Let now  $\ZF:= \bar{\XF}_1 \cup \dots \cup \bar{\XF}_n$. Then  $\ZF$ is proper over  $\so_L$ and  ${\rm sp}^{-1}(\XF \setminus \ZF) \subset (X \setminus Z)$ hence $s|_{ {\rm sp}^{-1}(\XF \setminus \ZF)}=0$. We have proved that  $s \in \Gamma_c(X_{\eet}, \mathcal{F})$.   
\end{proof}

Denote by  $F$ the functor 
$$F: \mode(X_{\eet}, \Z_{\ell})\to  \mode(\Z_{\ell}),\quad 
 F(\mathcal{F}):=\colim_{Z \in \Phi} \Gamma_Z(X_{\eet}, \mathcal{F}).
 $$

\begin{proposition}{\rm (van der Put, \cite[(1.3)]{VdP92})}
\label{limZ-etc}
For all sheaves  $\mathcal{F}$ of  $\mode(X_{\emph{\text{\'et}}} - \Z_{\ell})$,  there exists a natural quasi-isomorphism 
\[ \R F(\mathcal{F}) \xrightarrow{\sim} [\R \Gamma_{\emph{\text{\'et}}}(X,\mathcal{F}) \to  \colim_{Z\in \Phi} \R \Gamma_{\emph{\text{\'et}}}(X \setminus Z, \mathcal{F})]. \]
\end{proposition}

\begin{proof}It suffices to check that, for all quasi-compact opens
$Z$,  we have a distinguished triangle:
\[ \R \Gamma_Z(X_{\eet}, \FC) \to \R \Gamma_{\eet}(X, \FC) \to \R \Gamma_{\eet}(X\setminus Z, \FC), \]
where  $\R \Gamma_Z(X_{\eet}, -)$ is the right derived functor of  $\FC \mapsto \Gamma_Z(X_{\eet}, \mathcal{F})$.

 To prove that, take the injective resolution $\FC\to \IC^{\bullet}(\FC) $ defined as follows. For $n=0$, we set $\IC^{-1}(\FC) =\FC$ and $\IC^0(\FC)(U):= \prod_{x \in U} \widetilde{\FC}_x$, where the product is taken over the geometric points  $x$ of  $X$ such that  $x \in U$ and  $\widetilde{\FC}_x$ is  the injective envelope of  $\FC_x$. We define then  by induction $\IC^{n+1}(\FC):=\IC^0(\IC^{n-1}(\FC)/\IC^{n-2}(\FC))$. We have  that $\Gamma_{\eet}(X, \IC^n(\sff)) \to \Gamma_{\eet}(X \setminus Z, \IC^n(\sff))$ is surjective, for all $n \ge 0$.
 It follows that the sequence of complexes
 $$
 0\to \Gamma_{Z}(X_{\eet}, \IC^{\bullet}(\sff))\to  \Gamma_{\eet}(X, \IC^{\bullet}(\sff)) \to \Gamma_{\eet}(X \setminus Z, \IC^{\bullet}(\sff))\to 0
 $$
 is exact, as wanted.
\end{proof}

\begin{corollary}
\label{!-etc}
For all sheaves $\mathcal{F}\in \mode(X_{\emph{\text{\'et}}}  - \Z_{\ell})$,  we have a natural quasi-isomorphism
\[ \R \Gamma_{!}(X, \FC) \xrightarrow{\sim} [\R \Gamma_{\text{\emph{\'et}}}(X, \FC) \to \colim_{Z\in \Phi} \R \Gamma_{\eet}(X \setminus Z, \mathcal{F})]. \]
\end{corollary}

\begin{proof}
$\R \Gamma_{!}(X,-)$ is the right  derived functor  of the functor $\FC \mapsto \Gamma_c(X_{\eet}, \FC)$ and hence it is equal to  $\R F$ by  Lemma \ref{c-limZ}. The result follows now directly from  Proposition \ref{limZ-etc}.  
\end{proof}

\begin{proof}{\rm({\em of Theorem \ref{Hu-proet1}})}  We want to show that
\[ \R \Gamma_{\eet,c, {\rm Hu}}(X, \Z_\ell) \xrightarrow{\sim} \R \Gamma_{\eet,c,{\rm naive}}(X, \Z_\ell). \]
For that, consider the following quasi-isomorphisms: 
\begin{align*}
\R \Gamma_{\eet, c, {\rm Hu}}(X, \Z_\ell) & \xrightarrow{\sim} \R \Gamma_{!}(X, \R \pi_*(\Z/\ell^n)_n) \\
& \xrightarrow{\sim} [\R \Gamma_{\eet}(X, \R \pi_*(\Z/\ell^n)_n) \to \colim_{Z\in \Phi} \R \Gamma_{\eet} (X\setminus Z, \R \pi_*(\Z/\ell^n)_n)]  \\
& \xrightarrow{\sim} [\R \lim_n \R \Gamma_{\eet}(X, \Z/\ell^n) \to  \colim_{Z\in \Phi} \R \lim_n \R \Gamma_{\eet} (X\setminus Z, \Z/\ell^n)]  \\
& =[\R \Gamma_{\eet}(X, \Z_\ell) \to  \colim_{Z\in \Phi} \R \Gamma_{\eet} (X\setminus Z, \Z_\ell)].
\end{align*}
The first and second quasi-isomorphisms
follow from Proposition~\ref{Hu-!pi} and Corollary~\ref{!-etc}, respectively. The third one follows from Lemma~\ref{pi-holim}. The equality in the fourth row is obtained from the definition of $\ell$-adic \'etale cohomology.
\end{proof}

\begin{example}(Cohomology with compact support of the affine line $\A_C^1$).\label{rain1}
\begin{lemma}We have quasi-isomorphisms of solid $\Z_{\ell}$ and $ \Z_p$-modules, respectively:
\begin{align}
\label{HicA1}
& H^i_{\proeet,c}(\A^1_C, \Z_{\ell}(1)) \simeq H^i_{\eet,c, {\rm naive}} (\A^1_C, \Z_{\ell}(1)) \simeq  \begin{cases} \Z_{\ell} & \text{if } i=2, \\ 0 & \text{otherwise}, \end{cases}   \\
& H^i_{\proeet,c}(\A^1_C, \Z_{p}(1)) \simeq H^i_{\eet,c, {\rm naive}} (\A^1_C , \Z_{p}(1)) \simeq  \begin{cases}  (\so_{{\mathbb P}^1,\infty}/C)\oplus \Z_{p} &\text{if } i=2 ,\\ 0 & \text{otherwise.} \end{cases} \notag
\end{align}
\end{lemma}
\begin{proof}
Recall that for all primes $\ell$ (including $\ell=p$), we have
\begin{equation}
\label{HiA}
H^i_{\eet} (\A_C^1, \Z/ \ell^n ) \simeq  \begin{cases} \Z/ \ell^n & \text{ if } i=0, \\ 0 & \text{ if } i \neq 0. \end{cases}
\end{equation}
For $j \in \N$, consider the closed discs 
\[ D_j:=\{z \in \A^1_C, |z| \le p^j\} \]
and  the associated distinguished triangle \eqref{tri-dist}:  
\[ \R \Gamma_{\eet,c}(\A_C^1, \Z_{\ell}) \to \R \Gamma_{\eet} (\A_C^1, \Z_{\ell}) \to  \colim_{j} \R \Gamma_{\eet} (\A_C^1 \setminus D_j, \Z_{\ell}). \]
  Write $C_j:=\{z \in \A^1_C, |z| \ge p^j\}$, then  
 \[ \colim_{j} \R\Gamma_{\eet}(\A^1_C \setminus D_j, \Z_{\ell}(1))\simeq   \colim_{j} \R\Gamma_{\eet}(C_j, \Z_{\ell}(1)). \] 
 Consider $E_j:=\{z \in \mathbf{P}^1, |z| \ge p^j\}$ so that $C_j=\A^1_C \cap E_j$. The Kummer sequence gives
\[ 
H^i_{\eet} (E_j, \Z/ \ell^n(1)) \simeq  \begin{cases} \Z/ \ell^n\Z & \text{if } i=0, \\ C\{(p^jz)^{-1}\}^{*}/C\{(p^jz)^{-1}\}^{*\ell^n} &\text{if } i=1, \\ 0 & \text{otherwise.} \end{cases}
\]
Using the Mayer-Vietoris sequence for $\mathbf{P}^1_C=E_j \cup \A^1_C$, we get
\[ 
H^i_{\eet} (C_j, \Z/ \ell^n(1)) \simeq  \begin{cases} \Z/ \ell^n\Z & \text{if } i=0, \\ C\{(p^jz)^{-1}\}^{*}/C\{(p^jz)^{-1}\}^{*\ell^n} \oplus \Z/\ell^n &\text{if } i=1, \\ 0 & \text{otherwise.} \end{cases}
\] 

  If  $\ell \neq p$ then $C\{(p^jz)^{-1}\}^{*}$ is $\ell$-divisible and we get: 
\[ 
H^i_{\eet} (C_j, \Z_{\ell}(1)) \simeq  \begin{cases} \Z_{\ell} & \text{if } i=0, \\  \Z_{\ell} &\text{if } i=1, \\ 0 & \text{otherwise} \end{cases}  \quad \text{and} \quad 
H^i_{\eet} (C_j, \Z_{p}(1)) \simeq  \begin{cases} \Z_{p} & \text{if } i=0, \\  C\{(p^jz)^{-1}\}^{*\wedge}\oplus \Z_{p} &\text{if } i=1, \\ 0 & \text{otherwise. } \end{cases}
\] 
Using \eqref{HiA} and  the isomorphism, via the logarithm (see \cite[Example A.2]{CDHN}), 
$
\colim_jC\{(p^jz)^{-1}\}^{*\wedge}\simeq (\so_{{\mathbb P}^1,\infty}/C)
$
this yields 
\eqref{HicA1}, as wanted.
 \end{proof}
 
  In particular, the above result  recovers the groups $H^2_{\eet,c, {\rm Hu}}(\A^1_C, \Z_{\ell}(1))$ and $H^2_{\eet,c, {\rm Hu}}(\A^1_C, \Z_p(1))$ computed in \cite[Ex (A.2)]{CDHN} in accordance with Theorem \ref{Hu-proet1}. 
 
 For the rational cohomology, we will see later (see Example~\ref{dim1-aff}) using syntomic techniques  that we have:  
 \begin{align*}
 H^i_{\proeet,c}(\A^1_C, \Q_p(1))\simeq  \begin{cases} 0 & \text{if } i \neq 2, \\ H^1_c(\A^1_C, \so) \oplus \Q_p & \text{if } i=2. \end{cases}
 \end{align*}
 Using the long exact sequence defining $H^1_c(\A^1_C, \so)$ and the vanishing of the coherent cohomology of $\A^1_C$, we obtain 
 $H^1_c(\A^1_C, \so) \simeq \so(\partial \A^1_C)/  \so(\A^1_C)$ and
  \begin{align*}
 H^i_{\proeet,c}(\A^1_C, \Q_p(1))\simeq \begin{cases} 0 & \text{if } i \neq 2, \\ \so(\partial \A^1_C)/\so(\A^1_C) \oplus \Q_p & \text{if } i=2. \end{cases} 
 \end{align*}
 In particular, we see that $H^i_{\proeet,c}(\A^1_C, \Q_p(1)) \simeq H^i_{\eet,c, {\rm Hu}}(\A^1_C, \Z_p(1)) \otimes^{\Box}_{\Z_p} \Q_p$ for all $i$. As we will show in Theorem \ref{thm:Huber-Qp} this is a general phenomena.
\end{example}

\subsubsection{Comparison with the  cohomology of Huber: rational coefficients}

The goal of this section is to prove the following theorem: 

\begin{theorem} \label{thm:Huber-Qp}
  Let $X$ be a partially proper rigid analytic variety over $L=K,C$. The quasi-isomorphism $\rg_{\eet,c, {\rm Hu}}(X, \Z_{\ell}) \xrightarrow{\sim} \rg_{\proeet,c}(X, \Z_{\ell})$ from Theorem~\ref{Hu-proet1} induces a quasi-isomorphism in $\sd(\Q_{\ell,\Box})$
  \[
    \rg_{\eet,c,{\rm Hu}}(X, \Z_{\ell})_{\Q_{\ell}} \xrightarrow{\sim} \rg_{\proeet,c}(X, \Q_{\ell}),  
  \]
  where we set $ \rg_{\eet,c,{\rm Hu}}(X, \Z_{\ell})_{\Q_{\ell}} :=\rg_{\eet,c,{\rm Hu}}(X, \Z_{\ell}) \solid_{\Z_{\ell}} \Q_{\ell}$. 
\end{theorem}

We precede the proof with a handful of remarks. Thanks to Theorem~\ref{Hu-proet1}, we have a quasi-isomorphism  in $\sd(\Q_{\ell,\Box})$
\[
    \rg_{\eet,c,{\rm Hu}}(X, \Z_{\ell})_{\Q_{\ell}} \xrightarrow{\sim} \rg_{\proeet,c}(X, \Z_{\ell})_{\Q_{\ell}}.
\]
Therefore, Theorem~\ref{thm:Huber-Qp} is equivalent to the following result, which has nothing to do with Huber's compactly supported cohomology anymore.

\begin{lemma} \label{lemma:Qp-vs-Zp}
  Let $X$ be a partially proper rigid analytic variety over $L=K,C$. Then, the natural morphism  in $\sd(\Q_{\ell,\Box})$
  \[ 
     \rg_{\proeet,c}(X, \Z_{\ell})_{\Q_{\ell}} \to \rg_{\proeet,c}(X, \Q_{\ell})
  \]
  is a quasi-isomorphism. 
\end{lemma}

Note that the same result is not true without compact support as the following basic example shows.
\begin{example}
  Let $X=\A^{1}_C$. Then the map 
  \begin{equation} \label{eqn:Zp-to-Qp} 
     \rg_{\proeet}(X, \Z_p)_{\Q_p} \to \rg_{\proeet}(X, \Q_p)
  \end{equation}
  is not a quasi-isomorphism. Indeed, by \cite[Th.~1.8]{CDN3}, the $H^1$ of the target equals
  \[ 
    H^1_{\proeet}(X, \Q_p) \simeq (\mathcal{O}(X)/C)(-1),
  \]
  while the source $H^1_{\proeet}(X, \Z_p)$ vanishes (e.g.\ by comparison with the cohomology of the scheme $H^1_\text{\'et}(\A^1_C, \Z_p)=0$). 
\end{example}

In order to measure the failure of \eqref{eqn:Zp-to-Qp} being a quasi-isomorphism, let us introduce the following notation for its mapping fiber
\[ 
  P(X) = [\rg_{\proeet}(X, \Z_{\ell})_{\Q_{\ell}} \to \rg_{\proeet}(X, \Q_{\ell})].
\]
Naturally, a morphism $Y\to X$ induces a map $P(X)\to P(Y)$.

Below we see rigid spaces as adic spaces and we use the language of adic geometry. 



\begin{lemma} \label{lemma:interior-of-closure} 
  Let $X$ be a rigid space and let $U\subseteq X$ be a retrocompact open. Then the closure $\overline{U}$ of $U$ in $X$ consists of all specializations of points of $U$. Moreover, $U$ is equal to the interior of $\overline{U}$.
\end{lemma}

\begin{proof}
This follows from a result in general topology \cite[Chapter 0, Cor.~2.2.27 and Prop. 2.4.4]{FujiwaraKato} combined with the fact that the underlying topological spaces of rigid spaces (in Fujiwara--Kato's sense) of type $V_\mathbf{R}$ are reflexive valuative spaces \cite[Chapter II, Th. 8.2.19]{FujiwaraKato}. Since the result is less difficult than having to move between different formalisms, we give a proof below.

Let $U'$ be the interior of $\overline{U}$ in $X$. Trivially $U\subseteq U'$ and we need to show the reverse inclusion. The underlying topological space of $X$ is a locally spectral space, so by \cite[Chapter~0, Cor.~2.2.27]{FujiwaraKato}, the closure $\overline{U}$ consists of all specializations of points of $U$ inside $X$. Therefore $U$ and $U'$ have the same closed points, in particular the same classical points. 

Let $V\subseteq U'$ be an affinoid open. As $U\hookrightarrow X$ is quasi-compact, the map $U\cap V\hookrightarrow V$ is an open immersion of qcqs rigid spaces which is surjective on classical points. It follows that $U\cap V = V$, and since $V$ was arbitrary, that $U \subseteq U'$.
\end{proof}




\begin{lemma} \label{lemma:construction-of-nice-opens}
Let $X$ be a qs rigid space, let $U,V\subseteq X$ be quasi-compact opens such that $U\subset V$. Then $U\Subset V$ if and only if there exists a partially proper $U'$ with $U\subseteq U'\subseteq V$. Moreover, one can pick $U'$ so that $(X\setminus U')^\circ$ is a retrocompact open and $(X\setminus U')^\circ \cup V = X$.
\end{lemma}

\begin{proof}
Let $\mathfrak{U}\subseteq\mathfrak{V}$ be an open immersion of admissible formal schemes over $\mathcal{O}_L$ which is a formal model of $U\subseteq V$, which exists because $U$ and $V$ are qcqs. The assumption that $U\Subset V$ translates into the statement that $Z:=\overline{\mathfrak{U}_0}$ (the closure of the special fiber of $\mathfrak{U}$ inside the special fiber of $\mathfrak{V}$) is proper over $k$. Let $U'\subseteq V$ be the Berthelot tube of $Z$, i.e.\ the interior of its preimage ${\rm sp}^{-1}(Z)$ under the specialization map ${\rm sp}\colon V\to\mathfrak{V}_0$. Then $U'$ is a partially proper open of $V$ containing $U$ by construction.

Lemma~\ref{lemma:interior-of-closure} above implies that the quasi-compact open $W = {\rm sp}^{-1}(\mathfrak{V}_0\setminus Z)$ is the interior of its closure, so that ${\rm sp}^{-1}(Z)$ is the closure of its interior. This in turn implies that the interior $(V\setminus U')^\circ$ of $V\setminus U'$ equals $W$ and hence is quasi-compact. 

Finally, we claim that 
\[
  (X\setminus U')^\circ = (X\setminus V)\cup (V\setminus U')^\circ.
\]
Suppose we know this to hold, then $X\setminus V\subseteq (X\setminus U')^\circ$ implies that $X = V\cup (X\setminus U')^\circ$. Moreover, we can check that $(X\setminus U')^\circ$ is retrocompact on the open cover $X = V \cup (X\setminus U')^\circ$, and we are done.

 It remains to show that if $x\in X\setminus V$ belongs to the closure of $V$, then it lies in the closure of the quasi-compact open $W = (V\setminus U')^\circ$. As taking the closure of a qc open amounts to adding specializations (Lemma~\ref{lemma:interior-of-closure} again), there exists a point $y\in V$ specializing to $x$, and $x$ belongs to the closure of $W$ if and only if $y$ does. But if $y$ is not in the closure of $W$, it belongs to the interior of $X\setminus W$, which equals $U'$. Since $U'$ is partially proper, it is closed under specialization in $X$ (by the valuative criterion), and hence $x\in U'$, a contradiction.
\end{proof}

\begin{corollary}  \label{cor:pp-presentation}
  Let $X$ be a partially proper rigid space. Then $X$ can be written as the filtered colimit of opens $U\subseteq X$ with the following properties:
  \begin{enumerate}
    \item $U$ is partially proper,
    \item the closure of $U$ in $X$ is quasi-compact,
    \item $(X\setminus U)^\circ$ is retrocompact in $X$.
  \end{enumerate}
  Moreover, such a family is cofinal with the family of quasi-compact opens of $X$.
\end{corollary} 

\begin{proof}
Since $X$ is partially proper, it can be written as the filtered colimit of a family of quasi-compact opens $\{U_i\}_{i\in I}$ such that for every $i\in I$ there exists a $j\in I$ with $j>i$ such that $U_i\Subset U_j$. Pick such a $j$ for every $i$. Applying Lemma~\ref{lemma:construction-of-nice-opens} to $U_i\Subset U_j\subseteq X$, we obtain an open subset $U'_i$ with $U_i\subset U'_i\subset U_j$ which satisfies (1)--(3). It is easy to see that the resulting family of subsets (possibly no longer indexed by $I$) is cofiltering, and trivially covers $X$.
\end{proof}

\begin{remark} 
The assertion of Corollary~\ref{cor:pp-presentation} is obvious in the Stein case. Embed $X\subseteq \mathbf{A}^{r}_L$ as a closed analytic subspace, let $U_n = X\cap D(0,|p|^{-n})^r$ be the intersection with a large closed polydisk, an affinoid open of $X$, and let $U'_n = X\cap D^\circ(0, |p|^{-n})^r$ be the intersection with the corresponding open polydisc, a pp open of $X$ containing $U_{n-1}$. Then $(X\setminus U'_n)^\circ \cap U_n$ is the product of circles $\{|x_i|=|p|^{-n}\}$, again an affinoid space.
\end{remark}

\begin{lemma} \label{lemma:PX}
  Let $X$ be a rigid analytic variety over $L=K,C$. 
  \begin{enumerate}
    \item If $X$ is quasi-compact, then $P(X) = 0$ (i.e., \eqref{eqn:Zp-to-Qp} is a quasi-isomorphism).
    \item If $X$ is partially proper, then $P(X)\to P(\partial X)$ is an isomorphism.
  \end{enumerate}
\end{lemma}

\begin{proof}
Claim (1) is clear. For claim (2), write $X$ as the filtering colimit of opens $U_i\subseteq X$ ($i\in I$) as in Corollary~\ref{cor:pp-presentation}. Since the family is cofinal with the family of quasi-compact opens of $X$, we have 
\[ 
  P(\partial X) = \varinjlim_{i\in I} \, P((X\setminus U_i)^\circ).
\]
It is therefore enough to show that  the map $P(X)\to P((X\setminus U)^\circ)$ is an isomorphism for every $U$ in the family. To this end, let $V\subseteq X$ be a quasi-compact open containing $U$. Thus $X$ can be  covered by the quasi-compact open $V$ and the non-quasi-compact open $(X\setminus U)^\circ$ whose intersection $W$ is quasi-compact by retrocompactness of $(X\setminus U)^\circ$. We have a morphism of distinguished triangles
\[ 
  \xymatrix{
    \rg_{\proeet}(X, \Z_{\ell})_{\Q_{\ell}} \ar[r]\ar[d] & \rg_{\proeet}((X\setminus U)^\circ, \Z_{\ell})_{\Q_{\ell}} \oplus \rg_{\proeet}(V, \Z_{\ell})_{\Q_{\ell} }\ar[r]\ar[d] & \rg_{\proeet}(W, \Z_{\ell})_{\Q_{\ell}}\ar[d] \\
    \rg_{\proeet}(X, \Q_{\ell}) \ar[r] & \rg_{\proeet}((X\setminus U)^\circ, \Q_{\ell}) \oplus \rg_{\proeet}(V, \Q_{\ell}) \ar[r] & \rg_{\proeet}(W, \Q_{\ell}).
  }
\]
By (1), we have $P(V) = 0 = P(W)$, and it follows that $P(X)\to P((X\setminus U)^\circ)$ is a quasi-isomorphism. 
\end{proof}

\begin{proof}[Proof of Lemma~\ref{lemma:Qp-vs-Zp} (and therefore also of Theorem~\ref{thm:Huber-Qp})]
For every quasi-compact open $U\subset X$, let $U'\subset U$ denote its interior, which is a partially proper open. Since $X$ is partially proper, the subsets $U'$ (as $U$ varies over quasi-compact opens $U$) are cofinal with the subsets $U$ (every $U$ is contained in $V'$ for some $V$), so that we have
\[ 
  \rg_{\proeet,c}(\partial X, \Q_{\ell}) = \varinjlim_U \rg_{\proeet,c}(X\setminus U', \Q_{\ell})
\]
and, similarly, 
\[ 
  \rg_{\proeet,c}(\partial X, \Z_{\ell})_{\Q_{\ell}} = \varinjlim_U \rg_{\proeet,c}(X\setminus U', \Z_{\ell})_{\Q_{\ell}}.
\]
By Lemma~\ref{lemma:PX}, the map $P(X)\to P(\partial X)$ is a quasi-isomorphism. Therefore the square 
\[ 
  \xymatrix{
    \rg_{\proeet}(X, \Z_{\ell})_{\Q_{\ell}} \ar[r] \ar[d] & \ar[d] \rg_{\proeet}(X, \Q_{\ell}) \\
    \rg_{\proeet}(\partial X, \Z_{\ell})_{\Q_{\ell}} \ar[r] & \rg_{\proeet}(\partial X, \Q_{\ell}) 
  }
\]
is cartesian. 
Flipping the square, we deduce that $\rg_{\proeet,c}(X, \Z_{\ell})_{\Q_{\ell}} \to \rg_{\proeet,c}(X, \Q_{\ell})$ is a quasi-isomorphism as well, as wanted.
\end{proof}

\subsubsection{Properties}
We will now list the following properties of Huber's \'etale cohomology with compact support (hence, in the case the spaces  are smooth and partially proper, by Theorem \ref{Hu-proet1} and  Theorem \ref{thm:Huber-Qp},  of the pro-\'etale cohomology with compact support). All the spaces are adic. 
   \begin{enumerate}
   \item If $X$ is proper then we have a natural quasi-isomorphism  in $\sd(\Z_{\ell,\Box})$
   $$ \R\Gamma_{\eet, {\rm c,Hu}}(X,\Z_\ell)\simeq  \R\Gamma_{\eet}(X,\Z_\ell).
   $$ 
   \item If $j:U\hookrightarrow X$ is an open embedding then we have the pushforward  map  in $\sd(\Z_{\ell,\Box})$$$j_*:  \R\Gamma_{\eet, {\rm c,Hu}}(U,\Z_\ell) \to \R\Gamma_{\eet, {\rm c,Hu}}(X,\Z_\ell).$$
     \item If $X$ is quasi-compact, there is an exact sequence \cite[Cor.\ 2.4]{Hub98} of solid $\Z_{\ell}$-modules
   $$
   0\to \R^1\varprojlim_nH^{i-1}_{\eet,c}(X,\Z/\ell^n)\to H^i_{\eet,{\rm c,Hu}}(X,\Z_\ell)\to \varprojlim_nH^i_{\eet,c}(X,\Z/\ell^n)\to 0
   $$
   \item Let $U$ be a taut open subspace of $X$, let $Z=X\setminus U$, and let $i: Z\hookrightarrow X$ be the inclusion. Assume that $X,U$ are partially proper. Then we have a distinguished triangle  in $\sd(\Z_{\ell,\Box})$
   \begin{align*}
  \rg_{\eet,{\rm c,Hu}}(U,\Z_\ell|_{U_n})  \to \rg_{\eet,{\rm c,Hu}}(X, \Z_\ell)\to \rg_{\eet,c}(Z, i^*\R\pi_*\Z_\ell)
   \end{align*}
   \item Let ${\mathbb U} $ be an open covering of $X$ such that every $U \in {\mathbb U} $ is taut and, for every
$U, V\in {\mathbb U}$, there exists a $W \in {\mathbb U}$ such that  $U \cup V\subset W$. Then,  the map of solid $\Z_{\ell}$-modules
$$
\varinjlim_{U\in {\mathbb U}}
H^i_{\eet,{\rm c,Hu}} (U,\Z_\ell|_{U})\to H^i_{\eet, {\rm c,Hu}} (X,\Z_\ell),\quad i\geq 0,$$
is an isomorphism \cite[Prop. 2.1.]{Hub98}.
\item 
Let $\{U\}_{i\in I}$ be an admissible open covering of $X$ by taut open  varieties. 
There is a spectral sequence \cite[Prop. 2.5.]{Hub98}
\begin{equation}\label{yunnan1}E_1^{a,b}=
\bigoplus_{J\subset I, |J|=-a+1}H^b_{\eet, {\rm c, Hu}}(U_J, \Q_\ell(r)|_{U_J}) \Rightarrow   H^{a+b}_{\eet,{\rm c, Hu}}(X, \Q_\ell(r)).
\end{equation}

\item Let $X$ be adic and partially proper and let $G$ be a locally profinite group acting continuously on $X$. 
Then $H^i_{\eet,c}(X,\Z/\ell^n)$, $i\geq 0$, is a smooth $G$-module \cite[Cor.\ 7.8]{B}. 
\end{enumerate}
\subsection{Overconvergent setting}
We define and study compactly supported $\ell$-adic pro-\'etale cohomology of smooth dagger varieties over $K$ or $C$. As usual, we first give a local definition via presentations of dagger affinoids and then globalize it. We show that this definition coincides with the rigid analytic one for partially proper varieties. 

\subsubsection{Definition}

We refer the reader to \cite[Ch. 3]{CN3} for the definition of pro-\'etale cohomology of dagger varieties. Here, we define the overconvergent pro-\'etale cohomology with compact support by the following process: locally, if $X$ is a dagger affinoid over $L=K,C$ with  a presentation ${\rm pres}(X)=\{X_h\}_{h\in\N}$ of the dagger structure\footnote{See the Appendix of \cite{Vez} for the definition and properties of presentations of dagger structures.} and rigid completion $\wh{X}$,
we define  in $\sd(\Q_{\ell,\Box})$
  \begin{align*}
 \rg^{\natural}_{\proeet}(\partial X,\Q_{\ell}(r)) & := \rg_{\proeet}(\partial \{X\}_h,\Q_{\ell}(r)):=\colim_h\rg_{\proeet}(X_h\setminus \wh{X},\Q_{\ell}(r)),\\
\rg^{\natural}_{\proeet,c}(X,\Q_{\ell}(r)) & := \rg_{\proeet,c}(\{X\}_h,\Q_{\ell}(r)):=[\rg^{\natural}_{\proeet}(X,\Q_{\ell}(r))\to  \rg^{\natural}_{\proeet}(\partial X,\Q_{\ell}(r)) ],
  \end{align*}
  where $$\rg^{\natural}_{\proeet}(X,\Q_{\ell}(r)):= \rg_{\proeet}(\{X\}_h,\Q_{\ell}(r)):=\colim_h \rg_{\proeet}(X_h,\Q_{\ell}(r)).
  $$
Putting it together we get
\begin{align}\label{dec1}
\rg^{\natural}_{\proeet,c}(X,\Q_{\ell}(r)) & \simeq \colim_h [\rg_{\proeet}(X_h,\Q_{\ell}(r))\to \rg_{\proeet}(X_h\setminus \wh{X},\Q_{\ell}(r))] \\
 & =:\colim_h \rg_{\proeet,\wh{X}}(X_h,\Q_{\ell}(r)).\notag
\end{align}
This form makes it clear that $\rg^{\natural}_{\proeet,c}(X,\Q_{\ell}(r))$ is covariant for open embeddings.  Indeed, let $U\to X$ be an open embedding of smooth dagger affinoids. Choose compatible presentations $\{U_h\}$ and $\{X_h\}$ (i.e., so that the maps $U_h\to X_h$ are also  open embeddings). We define the pushforward map  $\rg^{\natural}_{\proeet,c}(U,\Q_{\ell}(r))\to \rg^{\natural}_{\proeet,c}(X,\Q_{\ell}(r))$ via the composition
$$
\rg_{\proeet,\wh{U}}(U_h,\Q_{\ell}(r)))\stackrel{\sim}{\leftarrow} \rg_{\proeet,\wh{U}}(X_h,\Q_{\ell}(r)){\rightarrow}\rg_{\proeet,\wh{X}}(X_h,\Q_{\ell}(r)).
$$
 Here the first arrow is a quasi-isomorphism by the excision \eqref{excision1}.

 For a general smooth dagger variety $X$ over $L$, we define the cosheaf $\sa_{\proeet,c}(r)$ on $X$ as the cosheaf\footnote{See \cite[Definition 5.5.4.1]{LurieHA} for a definition of a cosheaf in a $\infty$-category.} associated to the precosheaf defined by $U\mapsto \rg^{\natural}_{\proeet,c}(U,\Q_{\ell}(r)), U\in  {\rm SmAff}^{\dagger}_{L}$, $U\to X$  an open embedding. We set 
  \begin{equation}
  \label{csh-dagger}
  \rg_{\proeet,c}(X,\Q_{\ell}(r)):=\rg(X,\sa_{\proeet,c}(r)) \in {\cal D}(\Q_{\ell,\Box}).
  \end{equation}
If $ L =K$,  it is a dg $\Q_{\ell}$-algebra equipped with  a continuous action of $\sg_K$.

  \begin{lemma}\label{herbata2}Let $L=K,C$. Let $X$ be in ${\rm Sm}_L^{\dagger}$. 
  \begin{enumerate} 
  \item  {\rm (Local-global compatibility)} If $X$ is  a  dagger affinoid then the natural map in $\sd(\Q_{\ell,\Box})$
  $$ \rg_{\proeet,c}(X,\Q_\ell(r))\to \rg^{\natural}_{\proeet,c}(X,\Q_\ell(r))$$
   is a quasi-isomorphism. 
  \item {\rm (Passing to completion)} If $X$ is partially proper then there is  natural quasi-isomorphism  in $\sd(\Q_{\ell,\Box})$
  $$
  \rg_{\proeet,c}({X},\Q_\ell(r))\stackrel{\sim}{\to} \rg_{\proeet,c}(\wh{X},\Q_\ell(r))
  $$
  \end{enumerate}
  \end{lemma}
  
  \begin{proof}
   For the first claim, 
  it suffices to show that for any open covering $X=X_1\cup X_2$ by dagger affinoids, we have the Mayer-Vietoris distinguished triangle
  $$
  \rg^{\natural}_{\proeet,c}(X,\Q_{\ell}(r))\leftarrow \rg^{\natural}_{\proeet,c}(X_1,\Q_{\ell}(r))\oplus \rg^{\natural}_{\proeet,c}(X,\Q_{\ell}(r))\leftarrow \rg^{\natural}_{\proeet,c}(X_1\cap X_2,\Q_{\ell}(r)).
  $$
  Choosing a presentation $\{ X_h \}$ of $X$, an $h\in \N$, and using excision we can write this triangle as
  $$
   \rg_{\proeet,\wh{X}}(X_h,\Q_{\ell}(r))\leftarrow \rg_{\proeet,\wh{X}_1}(X_h,\Q_{\ell}(r))\oplus \rg_{\proeet,\wh{X}_2}(X_h,\Q_{\ell}(r))\leftarrow \rg_{\proeet,\wh{X_1\cap X_2}}(X_h,\Q_{\ell}(r)),
  $$
  which is known to be distinguished. 
  
  Let us now prove claim (2). 
 If $X$ is Stein, take $\{ U_n \}$ an increasing covering of $X$ by admissible dagger affinoids. We have a commutative diagram:
\[ \xymatrix{ \rg_{\proeet,c}(\wh{X},\Q_{\ell}(r)) \ar[r]  & \rg_{\proeet}(\wh{X},\Q_{\ell}(r)) \ar[r]  & \colim_{n}  \rg_{\proeet}(\wh{X}\setminus \wh{U}_n,\Q_{\ell}(r))  \\
\rg_{\proeet,c}(X,\Q_{\ell}(r)) \ar[r] \ar@{-->}[u] & \rg_{\proeet}(X,\Q_{\ell}(r)) \ar[r]  \ar[u]_{\rotatebox{90}{$\sim$}}& \colim_{n}  \rg_{\proeet}(X\setminus {U}_n,\Q_{\ell}(r)) \ar[u]_{\rotatebox{90}{$\sim$}} } \] 
where the first row is the distinguished triangle  \eqref{tri-dist}  and the second one --  the distinguished triangle  \eqref{tri-overcv-proet}. (We note that there is no circular reasoning here: the proof of Proposition \ref{stein-proet} uses only the first claim of this lemma.) The second and third vertical arrows are quasi-isomorphisms by the usual comparison results between overconvergent cohomologies and their rigid analytic analogues (see \cite[Prop. 3.17]{CN3}).  We obtain that the first vertical map (defined by the diagram) is a quasi-isomorphism as well. 

   Since a general partially proper variety $X$ is locally Stein we can conclude by analytic codescent.
  \end{proof}
\begin{remark} {\rm (Comparison with the definition of Grosse-Kl\"onne)} Compactly supported de Rham cohomology of dagger varieties was defined  by Grosse-Kl\"onne in \cite{GK1} and we can consider its pro-\'etale version. For a smooth dagger affinoid $X$ over $L=K,C$,
set in $\sd(\Q_{\ell,\Box})$  (see \cite[4.3]{GK1}) 
$$
 \rg^{\rm GK}_{\proeet,c}(X,\Q_{\ell}(r)):=\R\Gamma_{\proeet,\wh{X}}(X_h,\Q_{\ell}(r)):=[\R\Gamma_{\proeet}(X_h,\Q_{\ell}(r))\to \R\Gamma_{\proeet}(X_h\setminus \wh{X},\Q_{\ell}(r))], \quad h\in\N.
$$
This is independent of $h$ because of the excision \eqref{excision1}.
\begin{proposition} \label{natural1}Let $X$ be a smooth dagger affinoid.  
There is a quasi-isomorphism in ${\cal D}(\Q_{\ell,\Box})$:
$$
 \rg^{\rm GK}_{\proeet,c}(X,\Q_{\ell}(r))\stackrel{\sim}{\to}  \rg^{\natural}_{\proeet,c}(X,\Q_{\ell}(r)).
$$
\end{proposition}
\begin{proof} This follows from the excision \eqref{excision1}. Consider the following maps of distinguished triangles in ${\cal D}(\Q_{\ell,\Box})$
$$
\xymatrix{
 \rg^{\rm GK}_{\proeet,c}(X,\Q_{\ell}(r))\ar[r] \ar[d]^{\wr}&  \rg_{\proeet}(X_{h_0},\Q_{\ell}(r)) \ar[r] \ar[d]& \rg_{\proeet}(X_{h_0}\setminus \wh{X},\Q_{\ell}(r))\ar[d]\\
  \colim_h\rg^{\rm GK}_{\proeet,c}(X,\Q_{\ell}(r))\ar@{..>}[d] \ar[r] &  \colim_h\rg_{\proeet}(X_h,\Q_{\ell}(r)) \ar[r] \ar[d]^{\wr} &  \colim_h\rg_{\proeet}(X_h\setminus \wh{X},\Q_{\ell}(r))\ar@{=}[d]\\
    \rg^{\natural}_{\proeet,c}(X,\Q_{\ell}(r))\ar[r]  & \rg_{\proeet}(X,\Q_{\ell}(r))\ar[r] &  \colim_h\rg_{\proeet}(X_h\setminus \wh{X},\Q_{\ell}(r)).
  }
  $$
The existences of the first and the second distinguished triangle follow from the definition of $\rg^{\rm GK}_{\proeet,c}(X,\Q_{\ell}(r))$;  of the last one from the definition of $\rg^{\natural}_{\proeet,c}(X,\Q_{\ell}(r))$.
The left vertical map is a  quasi-isomorphisms by the excision \eqref{excision1}. Also the dotted arrow exists and clearly  is a quasi-isomorphism. This concludes the proof. 
\end{proof}
\end{remark}

\begin{proposition}
\label{stein-proet}
Let $X$ be a smooth dagger Stein variety over $L=K,C$ and let $\{U_n\}_{n}$ be an admissible strict covering of $X$ with dagger affinoids $U_n \Subset U_{n+1}$. Then there are natural quasi-isomorphisms in $\sd(\Q_{\ell,\Box})$: 
\[ \rg_{\proeet,c}(X,\Q_{\ell}(r)) \stackrel{\sim}{\to}  \colim_n \rg_{\proeet,U_n}(X,\Q_{\ell}(r))\stackrel{\sim}{\to} \colim_n \rg_{\proeet,U_n}(U_{n+1},\Q_{\ell}(r)). \]
\end{proposition}
\begin{remark}
In particular, if we define the cohomology of the boundary of $X$ as 
$$ \rg_{\proeet}(\partial X,\Q_{\ell}(r)):= \colim_{n} \rg_{\proeet}(X \setminus U_n,\Q_{\ell}(r)),$$
we have a distinguished triangle:
\begin{equation}
\label{tri-overcv-proet} 
\rg_{\proeet,c}(X,\Q_{\ell}(r)) \to \rg_{\proeet}(X,\Q_{\ell}(r)) \to \rg_{\proeet}(\partial X,\Q_{\ell}(r)). 
\end{equation}
As in the rigid analytic set-up. 
\end{remark}
\begin{proof}
Let $X$ and $\{U_n\}$ be as in the statement of the proposition. For each $n$, let $\{U_{n,h} \}_h$ be a presentation of $U_n$. Then, 
\begin{align*}
\rg_{\proeet,c}(X,\Q_{\ell}(r)) \xleftarrow{\sim} \colim_n \rg_{\proeet,c}(U_n,\Q_{\ell}(r)) \xleftarrow{\sim} \colim_n \rg^{\natural}_{\proeet,c}(U_n,\Q_{\ell}(r)) \xleftarrow{\sim}  \colim_n \rg^{\rm GK}_{\proeet,c}(U_n,\Q_{\ell}(r)) 
 \end{align*} 
 where the first quasi-isomorphism comes from the fact that $\AS_{\proeet,c}$ is a cosheaf, the second one is the local-global compatibility from Lemma~\ref{herbata2}, and the third one is the one from Proposition~\ref{natural1}. 
For $h \in \N$, from the excision  \eqref{excision1}, we have 
\[ \rg_{\proeet,c}^{\rm GK}(U_n,\Q_{\ell}(r))= \rg_{\proeet, \wh{U}_n}(U_{n,h},\Q_{\ell}(r)) \xleftarrow{\sim} \rg_{\proeet, \wh{U}_n}(\wh{X},\Q_{\ell}(r)). \] 
We conclude the proof using the quasi-isomorphisms $\rg_{\proeet}(\wh{X},\Q_{\ell}(r))\xleftarrow{\sim}\rg_{\proeet}(X,\Q_{\ell}(r))$ and $\rg_{\proeet}(\wh{X}\setminus \wh{U}_n,\Q_{\ell}(r))\xleftarrow{\sim}\rg_{\proeet}(X \setminus U_n,\Q_{\ell}(r))$.
\end{proof}

  \subsubsection{ Mayer-Vietoris sequence}  By definition, our compactly supported pro-\'etale cohomology satisfies a Mayer-Vietoris distinguished triangle:
  
\begin{proposition}{\rm (Mayer-Vietoris sequence for dagger varieties.)}
\label{mv-ov}
Let $X$ be a dagger variety over $L=K,C$ and let $X=U \cup V$ an admissible open covering. Let $r\geq 0$. We have a distinguished triangle in $\sd(\Q_{\ell,\Box})$:
\[ \rg_{\proeet, c}(U\cap V, \Q_\ell(r)) \to \rg_{\proeet,c}(U, \Q_\ell(r)) \oplus \rg_{\proeet,c}(V, \Q_\ell(r)) \to \rg_{\proeet,c}(X, \Q_\ell(r)).\]
\end{proposition}

\begin{proof}
The proposition follows from the formula \eqref{csh-dagger} and the cosheaf condition for $\sa_{\proeet,c}(r)$.
\end{proof}

\section{Rigid analytic de Rham and Hyodo-Kato cohomologies with compact support}

In this section we define compactly supported de Rham, $\B_{\rm dR}^+$-, and Hyodo-Kato cohomologies for smooth rigid analytic varieties over $K$ or $C$. Our construction is inspired by the one of Chiarellotto in  \cite{Ch90} (see also \cite{VdP92}, \cite{Bey97}). The usual $\B_{\rm dR}^+$- and Hyodo-Kato cohomologies we use here are the ones defined by Colmez-Nizio{\l} in \cite{CN3}, \cite{CN4}. 

We prove various properties and comparison theorems satisfied by the above  cohomologies with compact support. Most of these results are the compactly supported analogs of the ones we have in the usual case (see \cite{CN3}, \cite{CN4}).   
  
\subsection{De Rham cohomology with compact support}

We begin by reviewing the definition of compactly supported de Rham cohomology of rigid analytic varieties introduced by Chiarellotto \cite{Ch90}. 
\subsubsection{Definition of de Rham cohomology with compact support}
Let  $L= K, C$ and let $X$ be a smooth rigid analytic variety over $L$. 
  Recall that  the de Rham cohomology of $X \in {\rm Sm}_L$ is defined as 
\[ \rg_{\dr}(X):= \rg(X, \Omega_X^{\bullet}) \in \sff{\cal D}(L_{\Box}), \]
where $\Omega_X^{\bullet}$ is the complex of sheaves  of differential forms and $ \sff{\cal D}(L_{\Box})$ is filtered $\infty$-category of $ {\cal D}(L_{\Box})$. 
As in Section \ref{compact1}, we set in  $ \sff{\cal D}(L_{\Box})$
 \begin{align}
 \label{sliwka1}
 \R\Gamma_{\dr}(\partial X)  :=\colim_{Z\in \Phi} \R\Gamma_{\dr}(X \setminus Z),\quad 
\R\Gamma_{\dr,c}(X) & := [ \R\Gamma_{\dr}(X) \to  \R\Gamma_{\dr}(\partial X) ],
\end{align}
and write $H^i_{\dr,c}(X):=H^i\R\Gamma_{\dr,c}(X)$ for  the associated cohomology groups. We have $$\R\Gamma_{\dr,c}(X) \simeq  \R\Gamma_c(X, \Omega^{\bullet}_X),$$
where the compactly supported coherent cohomology is defined by an analog of formula (\ref{sliwka1}).
We have equipped $\rg_{\dr,c}(X)$ with the Hodge filtration
\[ F^r \R\Gamma_{\dr,c}(X):=\R\Gamma_c(X, \Omega_X^{\ge r}). \]
Hence $\R\Gamma_{\dr,c}(X)$ is a (solid) filtered dg $L$-algebra. We refer the reader to Section \ref{compact1} for its basic properties. 

\subsubsection{Definition of ${\mathbf B}^+_{\dr}$-cohomology with compact support} 
\label{Bdr-coho}
The construction is analogous to the one above, using the definition of the rigid analytic ${\mathbf B}^+_{\dr}$-cohomology from \cite[Section 3.3]{CN4}.
\vskip.2cm
(i) {\em Rigid analytic ${\mathbf B}^+_{\dr}$-cohomology}. Let us first recall how the usual $\B_{\dr}^+$-cohomology is defined. Denote by $\MS_C^{\rm ss}$ the category of semistable formal models and let $\XS \in \MS_C^{\rm ss}$. By regarding the crystalline structure sheaf $\OS_{\crr}$ and the divided power sheaves $\JS^{[r]}$, $r \ge 0$, as sheaves with values in condensed abelian groups, we can define the absolute crystalline cohomology $\rg_{\crr}(\XS)_{\Q_p}$, $\rg_{\crr}(\XS, \JS^{[r]})_{\Q_p}$ in ${\cal D}(\Q_{p,\Box})$. By completing with respect to its Hodge filtration and then taking the $\eta$-sheafification, we define sheaves $F^r\AS_{\crr}^{\bwedge}$, $r\ge0$ on ${\rm Sm}_{C, \eet}$ with values in ${\cal D}(\B_{\dr,\Box}^+)$. The ${\mathbf B}^+_{\dr}$-cohomology is then given by: 
\begin{align*}
& \rg_{\dr}(X/\B^+_{\dr}):= \rg_{\eet}(X, \AS_{\crr}^{\bwedge}) \in {\cal D}({\rm Mod}_{\B_{\dr}^+}^{\rm solid}) \\ & F^r\rg_{\dr}(X/\B^+_{\dr}):= \rg_{\eet}(X, F^r{\AS}_{\crr}^{\bwedge})\in {\cal D}({\rm Mod}_{\B_{\dr}^+}^{\rm solid}), \quad \text{for } r \ge 0. \end{align*} 
It is a (solid) dg filtered $\B^+_{\dr}$-algebra equipped with a continuous action of $\sg_K$ if $X$ is defined over $K$.

We have a canonical map in $\sd(C_{\Box})$
\begin{align*}
\vartheta: \quad & F^r\rg_{\dr}(X/\B^+_{\dr})\to F^r\rg_{\dr}(X),\quad r\geq 0.
\end{align*}
\vskip.2cm
(ii) {\em Compactly supported ${\mathbf B}^+_{\dr}$-cohomology}. Let $X\in {\rm Sm}_C$. 
Proceeding as in Section \ref{compact1}, we obtain compactly supported rigid analytic
${\mathbf B}^+_{\dr}$-cohomology  in $ {\cal D}(\B_{\dr,\Box}^+) $
$$\rg_{\dr,c}(X/\B^+_{\dr}),\quad F^r\rg_{\dr,c}(X/\B^+_{\dr}),\quad r\geq 0,
$$ starting with the rigid analytic ${\mathbf B}^+_{\dr}$-cohomology from the previous paragraph. It is a solid dg filtered $\B^+_{\dr}$-algebra  equipped with a continuous action of
 $\sg_K$ if $X$ is defined over $K$. It satisfies the properties listed in Section  \ref{compact1}. 
 We have a canonical map  in $\sd(C_{\Box})$
 \begin{align*}
\vartheta_c: \quad & F^r\rg_{\dr,c}(X/\B^+_{\dr})\to F^r\rg_{\dr,c}(X),\quad r\geq 0.
\end{align*}
 
 \subsubsection{Comparison results}    In this section, we give the analogs for compactly supported $\B_{\dr}^+$-cohomology of the comparison results that we have in the classical case. We first recall the results from \cite[Section 3.3.2]{CN4} formulating  them in the solid world.  
 
 \vskip.2cm
  (i)  {\em Product formula}. 
 Let $X$ be a smooth rigid analytic variety  over $K$. We define filtrations on $\R\Gamma_{\dr}(X)\dsolid_K \B^+_{\dr}$ and $\R\Gamma_{\dr,c}(X)\dsolid_K \B^+_{\dr}$ by the formulas
\begin{align*}
F^r(\R\Gamma_{\dr}(X)\dsolid_{K} \B^+_{\dr})&:= \colim_{i+j \ge r}  F^i\rg_{\dr}(X) \dsolid_K F^j\B_{\dr}^+, \\
F^r(\R\Gamma_{\dr,c}(X)\dsolid_{K} \B^+_{\dr})&:=\colim_{i+j \ge r}  F^i\rg_{\dr,c}(X) \dsolid_K F^j\B_{\dr}^+.
\end{align*}

 \begin{remark}
 \label{Rul1}
For $r \ge 0$, we clearly have a distinguished triangle in $\sd({\B^+_{\dr,\Box}})$: 
 \[ F^r(\R\Gamma_{\dr,c}(X)\dsolid_{K} \B^+_{\dr}) \to F^r(\R\Gamma_{\dr}(X)\dsolid_{K} \B^+_{\dr}) \to F^r(\R\Gamma_{\dr}(\partial X)\dsolid_{K} \B^+_{\dr}). \]
Moreover, since  $ \colim_Z(-)$ commutes with $(-)\dsolid_K \B_{\dr}^+$, we have  a canonical quasi-isomorphism
$$F^r(\R\Gamma_{\dr}(\partial X)\dsolid_{K} \B^+_{\dr})\stackrel{\sim}{\leftarrow} \colim_Z F^r(\R\Gamma_{\dr}(X\setminus Z)\dsolid_{K} \B^+_{\dr}).
$$ 
 \end{remark}
 
 Recall that in the classical case we have: 
 
 \begin{lemma}{\rm (\cite[Lemma 3.17]{CN4})}
 \label{Boa1}
Let $X$ be a smooth rigid analytic variety over $K$. Then, there is
a filtered quasi-isomorphism in ${\cal D}(\B_{\dr,\Box}^{+})$: 
 $$
\iota_{\rm BK}:  \quad \rg_{\dr}(X)\dsolid_K\B^+_{\dr}\xrightarrow{\sim} \rg_{\dr}(X_C/\B^+_{\dr}).
 $$
 \end{lemma}

 \begin{proof}
We can work \'etale locally.  Consider the maps:
\begin{align} 
\label{Boa}
\iota_{{\rm BK},r}:  
(\rg_{\dr}(\XS_L) \dsolid_L \B_{\dr}^+)/F^r & \to (\rg_{\crr}(\XS_{\OS_L}/\OS_L^{\times})_{\Q_p} \dsolid_L \rg_{\crr}(\OS_{\overline{K}}^{\times}/\OS_L^{\times})_{\Q_p})/F^r  \\ 
& \xrightarrow{\cup} (\rg_{\crr}(\XS/\OS_{L}^{\times})_{\Q_p})/F^r \leftarrow (\rg_{\crr}(\XS)_{\Q_p})/F^r \notag
\end{align}
for $\XS$ the base change to $\OS_C$ of a  semistable model $\XS_{\OS_L}$ over a finite extension $L$ of $K$, and where we have defined
\[ F^r(\rg_{\dr}(\XS_L) \dsolid_L \B_{\dr}^+):=(\OS(\XS_L) \solid_L F^r\B_{\dr}^+ \to \Omega^1(\XS_L) \solid_L F^{r-1}\B_{\dr}^+ \to \Omega^2(\XS_L) \solid_L F^{r-2}\B_{\dr}^+ \to \cdots).  \] 
The first arrow in \eqref{Boa} is a quasi-isomorphism (use the fact that   the ring $\B_{\dr}^+$ is the Hodge completed  crystalline cohomology of the base ring $\OS_{\overline{K}}$). For the second map, we know that, by the K{\"u}nneth formula, we have quasi-isomorphisms:
\[ \rg_{\crr}(\XS_{\OS_L,n}/\OS_{L,n}^{\times}) \otimes^{\LL}_{\OS_L} \rg_{\crr}(\OS_{\overline{K},n}^{\times}/\OS_{L,n}^{\times}) \xrightarrow[\sim]{\cup} \rg_{\crr}(\XS_n/\OS_{L,n}^{\times}), \quad \text{ for } n \ge 1. \]
By taking the derived inverse limit over $n\ge 1$, we get a quasi-isomorphisms: 
 \[ \rg_{\crr}(\XS_{\OS_L}/\OS_{L}^{\times}) \wh{\otimes}^{\LL}_{\OS_L} \rg_{\crr}(\OS_{\overline{K}}^{\times}/\OS_{L}^{\times}) \xrightarrow[\sim]{\cup} \rg_{\crr}(\XS/\OS_{L}^{\times}).\]
We then use \cite[Lemma A.9]{Bos22} to identify the left hand side with the solid tensor product: 
 \[ \rg_{\crr}(\XS_{\OS_L}/\OS_{L}^{\times}) \wh{\otimes}^{\LL}_{\OS_L} \rg_{\crr}(\OS_{\overline{K}}^{\times}/\OS_{L}^{\times}) = \rg_{\crr}(\XS_{\OS_L}/\OS_{L}^{\times}) \dsolid_{\OS_L} \rg_{\crr}(\OS_{\overline{K}}^{\times}/\OS_{L}^{\times}). \]
 Inverting $p$, we obtain that the second map in \eqref{Boa} is a quasi-isomorphism.  
 The third arrow in \eqref{Boa} is given by the base change from $\OS_F^{\times}$ to $\OS_L^{\times}$ and it is a quasi-isomorphism by \cite[Corollary 2.4]{NN16}. 
  By taking the derived limit $\R\lim_r \iota_{{\rm BK},r}$, we get a quasi-isomorphism:
\[
\hat{\iota}_{\rm BK} :  \R\lim_r(\rg_{\dr}(\XS_L) \dsolid_L \B_{\dr}^+)/F^r  \stackrel{\sim}{\to} \rg_{\crr}(\XS_L)_{\Q_p}^{\bwedge}. 
\]
To conclude it suffices to show that we have a quasi-isomorphism:
\[\rg_{\dr}(\XS_L) \dsolid_L \B_{\dr}^+ \xrightarrow{\sim} \R\lim_r(\rg_{\dr}(\XS_L) \dsolid_L \B_{\dr}^+)/F^r. \]
But this can be done as in \cite[Lemma 3.17]{CN4}, by writing the explicit definition of the filtration and arguing degreewise. 
  \end{proof}
  
 
Applying the previous lemma to the varieties $X$ and $X \setminus Z$, for $Z\in\Phi_X$, and  using that the tensor product commutes with colimits, we obtain a filtered quasi-isomorphism in $\sd(\B^+_{\dr,\Box})$  for the compactly supported cohomologies:
\begin{equation}
 \label{Boa2}
\iota_{{\rm BK},c}:  \quad \rg_{\dr,c}(X)\dsolid_K\B^+_{\dr}\xrightarrow{\sim} \rg_{\dr,c}(X_C/\B^+_{\dr}). 
 \end{equation}
 Note that the compatibility with the filtrations follows from the Remark~\ref{Rul1}.

 \vskip.2cm 
  (ii) {\em Projection from $\B^+_{\dr}$-cohomology to de Rham cohomology}.    We can now compare the $\B^+_{\dr}$- and de Rham cohomologies.  
  
 Recall that we have: 
  \begin{proposition} {\rm (\cite[Prop. 3.11]{CN3})}
Let $X$ be a smooth rigid analytic variety over $C$. Then, 
 \begin{enumerate}
 \item
We have a natural quasi-isomorphism in ${\cal D}(C_{\Box})$
 \begin{equation}
  \label{projection}
 \vartheta:\quad \rg_{\dr}(X/\B^+_{\dr}) \dsolid_{\B^+_{\dr}} C \stackrel{\sim}{\to}\rg_{\dr}(X).
\end{equation}
 \item More generally, for $r\geq 0$, we have a natural distinguished triangle in  ${\cal D}({\B_{\dr,\Box}^+}) $
 \begin{equation}
 \label{detail1}
 F^{r-1}\rg_{\dr}(X/\B^+_{\dr})\lomapr{t} F^r\rg_{\dr}(X/\B^+_{\dr})\lomapr{\vartheta}F^r\rg_{\dr}(X)
 \end{equation}
 \item For $r\geq 0$, we have a natural distinguished triangle in  ${\cal D}({\B_{\dr,\Box}^+}) $
 \begin{equation}
 \label{coffee1}
  F^{r+1}\rg_{\dr}(X/\B^+_{\dr})\to F^r\rg_{\dr}(X/\B^+_{\dr})\stackrel{\beta_X}{\to} \bigoplus_{i\leq r}\rg(X,\Omega^i_X)(r-i)[-i]
 \end{equation}
 \end{enumerate}
 \end{proposition}  
 
 \begin{proof}
The proof from \cite[Prop. 3.11]{CN3} still works in the solid setting: working \'etale locally, we use the $\iota_{\rm BK}$ from Lemma~\ref{Boa1} to reduce to similar statement for the de Rham cohomology instead of the crystalline one. 
 \end{proof}

We deduce from the previous proposition its analog in the compactly supported case:
   
 \begin{proposition}\label{projection-c}
Let $X$ be a smooth rigid analytic variety over $C$. Then, 
 \begin{enumerate}
 \item
We have a natural quasi-isomorphism in ${\cal D}(C_{\Box})$
 $$
 \vartheta_c:\quad \rg_{\dr,c}(X/\B^+_{\dr}) \dsolid_{\B^+_{\dr}} C \stackrel{\sim}{\to}\rg_{\dr,c}(X).
 $$
 \item More generally, for $r\geq 0$, we have a natural distinguished triangle in  ${\cal D}({\B_{\dr,\Box}^+}) $
 \begin{equation}
 \label{detail1-c}
 F^{r-1}\rg_{\dr,c}(X/\B^+_{\dr})\lomapr{t} F^r\rg_{\dr,c}(X/\B^+_{\dr})\lomapr{\vartheta_c}F^r\rg_{\dr,c}(X)
 \end{equation}
 \item For $r\geq 0$, we have a natural distinguished triangle in $ {\cal D}({\B_{\dr,\Box}^+}) $
 \begin{equation}
 \label{coffee1-c}
  F^{r+1}\rg_{\dr,c}(X/\B^+_{\dr})\to F^r\rg_{\dr,c}(X/\B^+_{\dr})\stackrel{\beta_X}{\to} \bigoplus_{i\leq r}\rg_c(X,\Omega^i_X)(r-i)[-i]
 \end{equation}
 \end{enumerate}
 \end{proposition} 
 
 \begin{proof}
 The proposition follows from Proposition~\ref{projection}, applying it to the smooth varieties $X$ and $X \setminus Z$, for $Z$ quasi-compact. As usual, we use that the solid tensor product commutes with colimits for the first claim. The second and third points come from the fact that, by definition, we have distinguished triangles  in  ${\cal D}({\B_{\dr,\Box}^+}) $ and in ${\cal D}(C_{\Box})$, respectively:
\begin{align*}
F^r \rg_{\dr,c}(X/ \B_{\dr}^+) & \to F^r \rg_{\dr}(X/ \B_{\dr}^+)   \to  F^r \rg_{\dr}(\partial X/ \B_{\dr}^+),  \\
F^r \rg_{\dr,c}(X)& \to F^r \rg_{\dr}(X)  \to  F^r \rg_{\dr}(\partial X),  \quad \text{ for } r \ge 0.
\end{align*}
 \end{proof}

 \subsubsection{Computations} In this section, we will give an explicit formula to compute the quotients of the compactly supported $\B_{\dr}^+$-cohomology of $X$ when $X$ is a Stein variety (and in particular, when $X$ is the affine space). These quotients will intervene in the computation of the compactly supported pro-\'etale cohomology of $X$, and the following results will allow us to define a trace map for the compactly supported pro-\'etale cohomology in Section~\ref{trace-maps}.  
 
\vskip.2cm  
(i) {\em Stein varieties.} Let $X$ be a smooth Stein variety of dimension $d$ over $K$. 
We will see later that in that case $H^i_c(X, \Omega^j_X)=0$ for $i \neq d$ and all $j$ in $\{0, \cdots, d\}$ (see Proposition~\ref{comp-vanish}). Thus, we obtain that the de Rham cohomology $\R\Gamma_{\dr,c}(X)$ is computed by the complex:  
\begin{align}\label{buis2}
 & H^d_c(X, \Omega_{X}^{\bullet}):=(H^d_c(X, \OC_X)  \to H^d_c(X, \Omega_X^1) \to \dots \to H^d_c(X, \Omega_X^d) )[-d],\\
 & F^iH^d_c(X, \Omega_{X}^{\bullet}):=H^d_c(X, \Omega_{X}^{\geq i})=(H^d_c(X, \Omega_X^i) \to \dots \to H^d_c(X, \Omega_X^d) )[-d-i].\notag
 \end{align}
It follows that in $\sd(\B^+_{\dr,\Box})$:
\[
\R\Gamma_{\dr,c}(X) \dsolid_K \B^+_{\dr} \simeq   (H^d_c(X, \OC_X) \solid_K \B^+_{\dr} \to H^d_c(X,\Omega_X^1) \solid_K \B^+_{\dr} \to \dots \to H^d_c(X, \Omega_X^d)\solid_K \B^+_{\dr} )[-d].
\]
Concerning the filtration, we have in  $\sd(\B^+_{\dr,\Box})$
\begin{align*}
F^r(\R\Gamma_{\dr,c}(X) \dsolid_K \B^+_{\dr}) & \simeq   \colim_{i+j \ge r} F^{i}H^d_c(X, \Omega_X^{\bullet}) \dsolid_K F^j \B^+_{\dr} \\ 
& \simeq   (H^d_c(X, \OC_X) \solid_K F^{r}\B^+_{\dr} \to H^d_c(X, \Omega_X^1)\solid_K F^{r-1}\B^+_{\dr}  
\to \cdots )[-d]. 
\end{align*}
This yields in $\sd(\B^+_{\dr,\Box})$
\begin{align*}
\DR_c(X,r) &  :=  (\R\Gamma_{\dr,c}(X) \dsolid_K \B^+_{\dr})/F^r & \\ 
& \simeq 
(H^d_c(X, \OC_X) \solid_K (\B^+_{\dr}/F^{r}) \to H^d_c(X, \Omega_X^1)\solid_K (\B^+_{\dr}/F^{r-1}) \\ & \hspace{5cm} \to \dots \to H^d_c(X, \Omega_X^{d})\solid_K (\B^+_{\dr}/F^{r-d}))[-d].
\end{align*}

\begin{lemma}
\label{DRc}
\begin{enumerate}
\item
We have an exact sequence in $\B^+_{\dr,\Box}$:
\begin{equation*}
\begin{split}
 0 \to (H^d_c(X_C, \Omega^{i-d})/ {\rm Im}\, d)  \solid_K  {\rm gr}^{r-i+d-1}_F\B^+_{\dr}\to H^{i} \DR_c(X,r) \to H^{i}_{\dr,c}(X) \solid_K  (\B^+_{\dr}/ F^{r-i+d-1}) \to 0.
\end{split}
\end{equation*}
\item  In degree  $2d$, we have an isomorphism in $\B^+_{\dr,\Box}$:
$$ H^{2d}\DR_c(X,r) \simeq H_{\dr,c}^{2d}(X) \solid_K (\B^+_{\dr}/ F^{r-d}).$$  
\end{enumerate}
\end{lemma}

\begin{proof}This follows from the fact that, 
for  $0 \le j\le d$,  the differential $\tilde{d}_{j+d}$ of the complex $\DR_c(X,r)$
factorizes as :
 \[ H^d_c(X,\Omega_X^j) \solid_K (\B^+_{\dr}/F^{r-j})   \stackrel{{\rm Id}\otimes{\rm Id}}{\longrightarrow} H^d_c(X,\Omega_X^{j})  \solid_K (\B^+_{\dr}/F^{r-j-1})  \stackrel{d_j\otimes{\rm Id}}{\longrightarrow}
 H^d_c(X,\Omega_X^{j+1})  \solid_K (\B^+_{\dr}/F^{r-j-1}) \]
 where the first arrow is surjective with kernel $H^d_c(X,\Omega_X^{j}) \solid_K (F^{r-j-1} \B^+_{\dr}/F^{r-j})$.  
 \end{proof}

\vskip.2cm
(ii) {\em Affine space.} 
For example, let $X={\mathbb A}_K^d$ be the rigid analytic affine space over $K$ of dimension $d\geq 1$. We will see in Section~\ref{dualities} that we have:
\begin{equation}
\label{mdm4}
 H^i_{\dr,c}(X)\simeq  \begin{cases} 0 \text{ for } i \neq 2d, \\ K \text{ for } i=2d .\end{cases}
\end{equation}

\begin{lemma}
\label{DRcAff} Let $r^{\prime}:=r-i+d-1$. 
We have isomorphisms in $\B^+_{\dr,\Box}$:
$$
 H^{i}\DR_c(X,r) \simeq 
 \begin{cases}
 ( H_c^d(X_C, \Omega^{i-d})/{\rm Im}\, d) \otimes^{\Box}_F{\rm gr}^{r^{\prime}}\B^+_{\dr}\simeq  ( H_c^d(X_C, \Omega^{i-d})/{\rm Ker}\,  d )\otimes^{\Box}_F{\rm gr}^{r^{\prime}}\B^+_{\dr}  & \mbox{ if } d \le i \le 2d-1,\\
 \B^+_{\dr}/ F^{r-d} & \mbox{ if } i=2d,\\
 0 & \mbox{ otherwise}.
 \end{cases}
 $$
\end{lemma}
\begin{proof}It suffices to use Lemma \ref{DRc} and the isomorphisms (\ref{mdm4}).
\end{proof}

\subsection{Rigid analytic Hyodo-Kato cohomology with compact support} We now define the Hyodo-Kato cohomology with compact support for smooth rigid analytic varieties over $K,C$. The construction is as in Section \ref{compact1}, using the definitions of \cite{CN3},\cite{CN4} of  the rigid analytic Hyodo-Kato cohomologies. We then prove, using the classical case,  that we have a compactly supported version of the Hyodo-Kato (iso)morphisms. 
 
\subsubsection{Definition}
We will give three definitions of Hyodo-Kato cohomologies: the arithmetic one (when the variety is defined over $K$) and the geometric and completed ones (when the variety is defined over $C$). 

\vskip.2cm
(i) {\em Rigid analytic Hyodo-Kato cohomology.}
 Let $L=K,C$ and $X \in {\rm Sm}_L$. As in section \ref{Bdr-coho} (i.e., by regarding the crystalline sheaves as sheaves with values in the solid world), we can define the solid versions of the Hyodo-Kato cohomologies from \cite[4.2, 4.3]{CN3}:
\[
\rg_{\hk}(X) \in {\cal D}_{\phi,N}(F_{\Box}), \text{ if } L=K, \quad \text{ and } \quad
\rg_{\hk}(X) \in {\cal D}_{\phi,N}(F^{\rm nr}_{\Box}), \text{ if } L=C, \]
and, if $X \in {\rm Sm}_C$, the completed Hyodo-Kato cohomology (see \cite[Sec. 4.1.1]{CN4}): 
\[ \rg_{\hk, \breve{F}}(X) \in {\cal D}_{\phi,N}(\breve{F}_{\Box}). \]  
Here $\sd_{\phi,N} (-) $ denotes  the derived $\infty$-category of solid $(\phi, N)$-modules; similarly we define  the category ${\cal D}_{\phi,N,\sg_K}(-)$.

\begin{lemma}
\label{kit-1}
For any $X \in {\rm Sm}_C$, we have a canonical quasi-isomorphism in ${\cal D}_{\phi,N} (\breve{F}_{\Box})$:
\[\beta:\quad \rg_{\hk}(X)\dsolid_{F^{\nr}}\breve{F}\stackrel{\sim}{\to} \rg_{\hk,\breve{F}}(X).\]
\end{lemma}

\begin{proof}
We can work \'etale locally. Let $\XS$ be the base change of a formal model $\XS_{\OS_L}$ defined over $L$ a finite extension of $K$. We have a quasi-isomorphism: 
\[ \rg_{\crr}(\XS_0/\OS_{F_L,n}^0)\otimes^{\LL}_{\OS_{F_L,n}} \OS_{\breve{F},n} \stackrel{\sim}{\to} \rg_{\crr}(\XS_0/\OS_{\breve{F},n}), \quad \text{ for } n \ge 1 \]
where $\OS_{F_L}:=W(k_L)$. By taking the derived inverse limit over $n$ and inverting $p$, we get a quasi-ismorphism:
 \[ \rg_{\hk}(\XS_0) \wh{\otimes}^{\LL}_{{F_L}} {\breve{F}} \stackrel{\sim}{\to} \rg_{\hk, \breve{F}}(\XS_0),\] 
where $\wh{\otimes}^{\LL}$ is the derived completion. Since $ \rg_{\hk}(\XS_0) $ is built from Banach spaces, we can  pass  to the solid tensor product.
\end{proof}

\vskip.2cm
(ii) {\em Compactly supported Hyodo-Kato cohomology.}
Proceeding as in Section \ref{compact1}, we obtain compactly supported Hyodo-Kato cohomology $\R\Gamma_{\hk,c}(X)$ from $\R\Gamma_{\hk}(X)$. This is a solid dg $F$-algebra if $L=K$, and a dg $F^{\nr}$-algebra if $L=C$, equipped with a Frobenius $\phi$, a monodromy operator $N$ and a continuous action of $\sg_K$ if $X$ is defined over $K$. It satisfies all of the  properties in Section~\ref{compact1}. 
   
If $X \in {\rm Sm}_C$, we define as above its completed version $\R\Gamma_{\hk,\breve{F},c}(X)$, using $\R\Gamma_{\hk,\breve{F}}(X)$. It is a solid dg $\breve{F}$-algebra equipped with a Frobenius $\phi$, a monodromy operator $N$ and a continuous action of $\sg_K$ if $X$ is defined over $K$. It satisfies all of the properties in Section \ref{compact1}. Moreover, since the tensor product commutes with  colimits, the quasi-isomorphism from Lemma~\ref{kit-1} applied to $X, X\setminus Z$, for $Z$ quasi-compact, induces a quasi-isomorphism in $\sd_{\phi,N} (\breve{F}_{\Box})$:
   $$
   \beta_c:\quad \rg_{\hk,c}(X)\dsolid_{F^{\nr}}\breve{F}\stackrel{\sim}{\to} \rg_{\hk,\breve{F},c}(X).
   $$
 
\subsubsection{Hyodo-Kato morphisms}
\label{HK-ch}

We will define two types of Hyodo-Kato morphisms: Hyodo-Kato-to-de Rham and Hyodo-Kato-to-$\B^+_{\dr}$.
Recall the  result for usual cohomology:
\begin{proposition}{\rm (\cite[Theorem 4.6]{CN4})}
\label{kit-2}
For $X \in {\rm Sm}_C$, we have natural quasi-isomorphisms 
in ${\cal D}(C_{\Box})$ and $\sd({\B^+_{\dr,\Box}})$, respectively,
  $$
  \iota_{\hk}: \quad \R\Gamma_{\hk,*}(X)\dsolid_{F_L}C\stackrel{\sim}{\to}\R\Gamma_{\dr}(X),\quad   \iota_{\hk}: \quad \R\Gamma_{\hk,*}(X)\dsolid_{F_L}\B^+_{\dr}\stackrel{\sim}{\to}\R\Gamma_{\dr}(X/\B^+_{\dr}),
  $$
where $F_L=F^{\nr}, \breve{F}$ and $*=-, \breve{F}$, that are compatible with the maps $\theta: \B_{\dr}^+ \to C$ and $\vartheta$ from \eqref{projection}.
\end{proposition}

\begin{proof}
We can work \'etale locally. The local quasi-isomorphisms are constructed in \cite{CN4} and the construction goes through in the condensed setting. Recall, that the main point is to show that we have a quasi-isomorphism:
$$
\epsilon^{\hk}_{\st}:\rg_{\hk}(\XS_1^0)_{\Q_p}\dsolid_{\breve{F}}\B^+_{\st}\stackrel{\sim}{\to} \rg_{\crr}(\XS_1)_{\Q_p}\dsolid_{\B^+_{\crr}}\B^+_{\st}. 
$$
\end{proof}

We deduce:
  \begin{proposition}
  \label{sains1}
 Let  $X$ be a smooth rigid analytic variety over $L=K,C$. 
  \begin{enumerate} 
  \item There is a Hyodo-Kato morphism in $\sd(F_{L,\Box})$
  $$
   \iota_{\hk,c}: \quad \R\Gamma_{\hk,*,c}(X){\to}\R\Gamma_{\dr,c}(X)
  $$
 (where $*=-, \breve{F}$) that is compatible with its usual version $
    \iota_{\hk}:  \R\Gamma_{\hk,*}(X){\to}\R\Gamma_{\dr}(X).
  $
  \item If $L=C$, it induces quasi-isomorphisms in ${\cal D}(C_{\Box})$ and in ${\B^+_{\dr,\Box}}$, respectively:
  \begin{align*}
  & \iota_{\hk,c}: \R\Gamma_{\hk,*,c}(X)\dsolid_{F_L}C\stackrel{\sim}{\to}\R\Gamma_{\dr,c}(X),\\
  &  \iota_{\hk,c}: \R\Gamma_{\hk,*,c}(X)\dsolid_{F_L}\B^+_{\dr}\stackrel{\sim}{\to}\R\Gamma_{\dr,c}(X/\B^+_{\dr}),
  \end{align*}
 where $F_L=F^{\nr}, \breve{F}$ and $*=-, \breve{F}$,  that are compatible with the maps $\theta: \B_{\dr}^+ \to C$ and $\vartheta_c$ from \eqref{projection-c}.
  \end{enumerate}
  \end{proposition}
  
  \begin{proof} The proposition follows immediately from Proposition~\ref{kit-2} applied to $Y=X, X\setminus Z$, for $Z$ quasi-compact.
  \end{proof}

\section{Overconvergent de Rham and Hyodo-Kato cohomologies with compact support}
In this section, we define compactly supported de Rham, $\B_{\dr}^+$- and Hyodo-Kato cohomologies for smooth dagger varieties over $L=K,C$. To do that we first define the compactly supported cohomology of a dagger affinoid $X$ as the  colimit of the cohomologies of the rigid affinoids $X_h$ forming a presentation of the dagger structure, but supported in the rigid completion $\wh{X}$ of $X$. We then globalize this definition. 

We prove various properties satisfied by these cohomologies, which are, as before, compactly supported versions of the results from \cite{CN3},\cite{CN4}. In particular, we compare them with their rigid analytic analogs defined in the previous section.   
 
 \subsection{Definition of de Rham cohomology with compact support}
We first introduce the definition of overconvergent de Rham cohomology with compact support using presentations of dagger structures (see \cite[Appendix]{Vez}, \cite[Section 3.2.1]{CN3} for the definition and properties of presentations of dagger structures). 
  
\vskip.2cm
    (i) {\em Local definition}. Let $X$ be a dagger affinoid over $L=K,C$. Let ${\rm pres}(X)=\{X_h\}$ and let $\wh{X}$ be the rigid completion of $X$. 
\index{rghk@\rghk}Define in ${\cal D}(L_{\Box})$:
  \begin{align*}
 \rg^{\natural}_{\dr}(\partial X) & := \rg_{\dr}(\partial \{X\}_h):=\colim_h\rg_{\dr}(X_h\setminus \wh{X}),\\
\rg^{\natural}_{\dr,c}(X) & := \rg_{\dr,c}(\{X\}_h):=[\rg^{\natural}_{\dr}(X)\to  \rg^{\natural}_{\dr}(\partial X) ]
  \end{align*}
  where $\rg^{\natural}_{\dr}(X):= \colim_h \rg_{\dr}(X_h)$ is the overconvergent de Rham cohomology from \cite{CN4}. We have a natural filtration on $\rg^{\natural}_{\dr,c}(X)$:
  \[ F^r\rg^{\natural}_{\dr,c}(X):=[F^r\rg^{\natural}_{\dr}(X) \to \colim_h F^r\rg_{\dr}(X_h\setminus \wh{X})], \quad  r \ge 0. \]
We can write
\begin{equation}\label{trick1}
 F^r\rg^{\natural}_{\dr,c}(X):= \colim_h[F^r\rg_{\dr}(X_h)\to F^r\rg_{\dr}(X_h\setminus \wh{X})]=:\colim_h F^r\rg_{\dr,\wh{X}}(X_h).
\end{equation}
As in the pro-\'etale case, this makes it clear that the functor $F^r\rg^{\natural}_{\dr,c}(X)$ is covariant for open embeddings.  
\vskip.2cm
  (ii) {\em Globalization}. For a general smooth dagger variety $X$ over $L$, using the natural equivalence
  \begin{equation}
  \label{equiv1}
  {\rm CSh}({\rm SmAff}^{\dagger}_{L,{\rm an}})\stackrel{\sim}{\to} {\rm CSh}({\rm Sm}^{\dagger}_{L,{\rm an}})\
  \end{equation}
  we define the cosheaves\footnote{See \cite[Definition 5.5.4.1]{LurieHA} for a definition of a cosheaf in a $\infty$-category.} $\sa_{\dr,c}$ and $F^r\sa_{\dr,c}$, $r\ge 0$, on $X_{\an}$ with values in ${\cal D}(L_{\Box})$ as the cosheaves associated to the precosheaves defined by $U\mapsto \rg^{\natural}_{\dr,c}(U)$ and $U\mapsto F^r\rg^{\natural}_{\dr,c}(U)$, $r \ge0$, for $U\in  {\rm SmAff}^{\dagger}_{L}$, $U\subset X$  an admissible open. We set  $$
  \rg_{\dr,c}(X):=\rg(X,\sa_{\dr,c}) \quad \text{ and } F^r\rg_{\dr,c}(X):=  \rg_{\dr,c}(X):=\rg(X,F^r\sa_{\dr,c}), \quad r\ge0. 
    $$
We obtain a solid filtered dg $L$-algebra. 

\begin{lemma}\label{pois1}
\begin{enumerate} 
\item {\rm (Local-global compatibility)}
If $X$ is a smooth dagger affinoid over $L=K,C$, there is a  quasi-isomorphism in $\sff{\cal D}(L_{\Box})$: 
\[ \rg_{\dr,c}(X) \stackrel{\sim}{\to} \rg^{\natural}_{\dr,c}(X). \]
\item {\rm (Passing to completion)}  If $X$ is smooth and partially proper over $L=K,C$ then there is a natural quasi-isomorphism in $\sff{\cal D}(L_{\Box})$
  $$
  \rg_{\dr,c}({X})\stackrel{\sim}{\to} \rg_{\dr,c}(\wh{X}).
  $$
  \end{enumerate}
\end{lemma}

\begin{proof}
The proof of this lemma is analogous to  the one of Lemma~\ref{herbata2}.
 \end{proof}

\begin{proposition}
\label{stein-dr}
Let $X$ be a smooth dagger Stein variety over $L=K,C$ and let $\{U_n\}_{n}$ be an admissible strict covering of $X$ with dagger affinoids $U_n \Subset U_{n+1}$. Then there are natural quasi-isomorphisms in $\sd(L_{\Box})$: 
\[ \rg_{\dr,c}(X) \stackrel{\sim}{\to}  \colim_n \rg_{\dr,U_n}(X)\stackrel{\sim}{\to} \colim_n \rg_{\dr,U_n}(U_{n+1}). \]
\end{proposition}
\begin{remark}In particular, if we define the cohomology of the boundary of $X$ from Proposition \ref{stein-dr} as 
$$ \rg_{\dr}(\partial X):= \colim_{n} \rg_{\dr}(X \setminus U_n),$$
we have a distinguished triangle:
\begin{equation}
\label{tri-overcv} 
\rg_{\dr,c}(X) \to \rg_{\dr}(X) \to \rg_{\dr}(\partial X). 
\end{equation}
As in the rigid analytic set-up. 
\end{remark}
\begin{proof}
Let $X$ and $\{U_n\}$ be as in the statement of the proposition. For each $n$, let $\{U_{n,h} \}_h$ be a presentation of $U_n$. Then, 
\begin{align*}
\rg_{\dr,c}(X) \xrightarrow{\sim} \colim_n \rg_{\dr,c}(U_n) \xleftarrow{\sim} \colim_n \rg^{\natural}_{\dr,c}(U_n) \xleftarrow{\sim}  \colim_n \rg^{\rm GK}_{\dr,c}(U_n) 
 \end{align*} 
 where the first quasi-isomorphism comes from the fact that $\AS_{\dr,c}$ is a cosheaf, the second one is the local-global compatibility from Lemma~\ref{pois1}, and the third one is the de Rham analog of the one from Proposition~\ref{natural1}. 
For $h \in \N$, from the excision  \eqref{excision1}, we have 
\[ \rg_{\dr,c}^{\rm GK}(U_n)= \rg_{\dr, \wh{U}_n}(U_{n,h}) \xleftarrow{\sim} \rg_{\dr, \wh{U}_n}(\wh{X}). \] 
We conclude the proof using the quasi-isomorphisms $\rg_{\dr}(\wh{X})\xleftarrow{\sim}\rg_{\dr}(X)$ and $\rg_{\dr}(\wh{X}\setminus \wh{U}_n)\xleftarrow{\sim}\rg_{\dr}(X \setminus U_n)$.
\end{proof}

\subsection{Overconvergent ${\mathbf B}^+_{\dr}$-cohomology with compact support}
Proceeding as for the de Rham cohomology, we define compactly supported $\B_{\dr}^+$-cohomology via presentation of dagger structures.  

\subsubsection{Definition}
Let $X$ be a smooth dagger affinoid over $C$. Taking a presentation ${\rm pres}(X)=\{X_h\}$ of $X$ as above we define overconvergent ${\mathbf B}^+_{\dr}$-cohomology with compact support
 $$     F^r\rg^{\natural}_{\dr,c}(X/\B^+_{\dr}) \in {\cal D}(\B_{\dr,\Box}^+), \quad \mbox{for }r\geq 0,
 $$
   using the overconvergent ${\mathbf B}^+_{\dr}$-cohomology 
 \[ F^r\rg^{\natural}_{\dr}(X/\B_{\dr}^+):= \colim_h F^r\rg_{\dr}(X_h/\B_{\dr}^+)\]
  from \cite[Sec.  3.4.1]{CN3}. The functor $ F^r\rg^{\natural}_{\dr,c}(X/\B^+_{\dr}) $ is covariant for open embeddings. 
  
  The $\eta$-analytic cosheafification of $ F^r\rg^{\natural}_{\dr,c}(X/\B^+_{\dr})$ gives us a cosheaf $F^r\sa^{\bwedge}_{\crr,c}$ on any smooth dagger variety $X$ over $C$, with values in ${\cal D}(\B_{\dr,\Box}^+)$. The filtered $\B^+_{\dr}$-cohomology
with compact support is then defined as 
$$ F^r \rg_{\dr,c}(X/\B^+_{\dr}):=\rg(X,F^r\sa^{\bwedge}_{\crr,c}) \in {\cal D}(\B_{\dr,\Box}^+) ,\quad r\geq 0.$$
\begin{lemma}
\label{pois1-Bdr}
\begin{enumerate}
\item {\rm (Local-global compatibility)}
If $X$ is a smooth dagger affinoid over $C$, there is a  quasi-isomorphism in $\sff{\cal D}({\B_{\dr,\Box}^+})$: 
\[ \rg_{\dr,c}(X/\B^+_{\dr}) \stackrel{\sim}{\to} \rg^{\natural}_{\dr,c}(X/\B^+_{\dr}). \]
\item   {\rm (Passing to completion)} If $X$ is smooth and partially proper over $C$ then there is a  natural quasi-isomorphism in $\sff{\cal D}({\B_{\dr,\Box}^+})$
  $$
  \rg_{\dr,c}({X}/\B^+_{\dr})\stackrel{\sim}{\to} \rg_{\dr,c}(\wh{X}/\B^+_{\dr}).
  $$

\end{enumerate}
\end{lemma}

\begin{proof}
The proof  of this lemma is analogous to  the one of Lemma~\ref{herbata2}. 
 \end{proof}

 We have an analog of  Proposition~\ref{stein-dr}  in this setting (using \cite[Cor. 4.32]{CN4})  hence we see that if $X$ is Stein and smooth, we have a distinguished triangle 
\begin{equation}
\label{tri-overcv1} 
\rg_{\dr,c}(X/ \B_{\dr}^+) \to \rg_{\dr}(X/ \B_{\dr}^+) \to \rg_{\dr}(\partial X/ \B_{\dr}^+) 
\end{equation}
where $ \rg_{\dr}(\partial X/ \B_{\dr}^+):= \colim_{n} \rg_{\dr}((X \setminus U_n)/ \B_{\dr}^+)$,  for $\{U_n\}$ a Stein covering of $X$.
\subsubsection{Properties} \label{ov-Bdr000}
We will now prove properties of overconvergent $\B^+_{\dr}$-cohomology that do not require Hyodo-Kato cohomology.
 We have canonical maps, in $\sd(\B^+_{\dr,\Box})$ and $\sff\sd(K_{\Box})$, respectively:
\begin{align*}
\vartheta_c: \quad & F^r\rg_{\dr,c}(X/\B^+_{\dr})\to F^r\rg_{\dr,c}(X),\quad X\in{\rm Sm}^{\dagger}_C, \quad r\geq 0,\\
\iota_{{\rm BK},c}:  \quad & \rg_{\dr,c}(X)\to\rg_{\dr,c}(X_C/\B^+_{\dr}),\quad X\in{\rm Sm}^{\dagger}_K,
\end{align*}
induced by their rigid analytic analogs. 
 \begin{proposition}\label{prison-c}
 \begin{enumerate}
 \item  {\rm (Projection)} Let $X$ be a smooth dagger variety over $C$. Then, 
\begin{enumerate}
\item The map $\vartheta_c$ yields  a natural quasi-isomorphism in ${\cal D}(C_{\Box})$:
 $$
 \vartheta_c:\quad \rg_{\dr,c}(X/\B^+_{\dr})\dsolid_{\B^+_{\dr}}C \stackrel{\sim}{\to}\rg_{\dr}(X).
 $$
  \item More generally, for $r\geq 0$, we have a natural distinguished triangle in  ${\cal D}({\B_{\dr,\Box}^+}) $
 \begin{equation}
 \label{detail11-c}
 F^{r-1}\rg_{\dr,c}(X/\B^+_{\dr,c})\lomapr{t} F^r\rg_{\dr,c}(X/\B^+_{\dr})\lomapr{\vartheta_c}F^r\rg_{\dr,c}(X)
 \end{equation}
 \item  {\rm (Devissage)} For $r\geq 0$, we have a natural distinguished triangle in $ {\cal D}({\B_{\dr,\Box}^+}) $
 \begin{equation}
 \label{niedziela10-c}
 F^{r+1}\rg_{\dr,c}(X/\B^+_{\dr})\to  F^r\rg_{\dr,c}(X/\B^+_{\dr})\to \bigoplus_{i\leq r}\rg_c(X,\Omega^i_X)(r-i)[-i]
 \end{equation}
\end{enumerate}
 \item  {\rm (Product  formula)} Let $X$ be a smooth dagger variety over $K$.  The map $\iota_{{\rm BK},c}$ 
defined above yields a natural  quasi-isomorphism in $\sff{\cal D}(\B^+_{\dr,\Box})$:
       $$
\iota_{{\rm BK},c}:  \quad \rg_{\dr,c}(X)\dsolid_K\B^+_{\dr}\to\rg_{\dr,c}(X_C/\B^+_{\dr}).
 $$
\end{enumerate}
\end{proposition}

\begin{remark}  In Proposition \ref{prison-c} (2), 
the filtration on $\R\Gamma_{\dr,c}(X)\dsolid_K \B^+_{\dr}$ is defined by the formula
$$
F^r(\R\Gamma_{\dr,c}(X)\dsolid_{K} \B^+_{\dr}):= \colim_{i+j \ge r}  F^i\rg_{\dr,c}(X) \dsolid_K F^j\B_{\dr}^+,
$$
and, if $X$ is a dagger affinoid, we also define
$$
F^r(\R\Gamma^{\natural}_{\dr,c}(X)\dsolid_{K} \B^+_{\dr}):= \colim_{i+j \ge r}  F^i\rg^{\natural}_{\dr,c}(X) \dsolid_K F^j\B_{\dr}^+.
$$ 
Note that we have a local-global compatibility between the two, using Lemma~\ref{pois1}. 
By the same argument as in Remark~\ref{Rul1}, we see that we have a distinguished triangle 
\[ F^r(\R\Gamma^{\natural}_{\dr,c}(X)\dsolid_{K} \B^+_{\dr}) \to F^r(\R\Gamma^{\natural}_{\dr}(X)\dsolid_{K} \B^+_{\dr}) \to F^r(\R\Gamma^{\natural}_{\dr}(\partial X)\dsolid_{K} \B^+_{\dr}). \]
 We also have  $F^r(\R\Gamma^{\natural}_{\dr}(\partial X)\dsolid_{K} \B^+_{\dr}):= \colim_h F^r(\R\Gamma^{\natural}_{\dr}(X_h\setminus \wh{U})\dsolid_{K} \B^+_{\dr})$, for $\{ X_h\}$ a dagger presentation of $X$. 
\end{remark}

\begin{proof}
To prove the point (1)(a) of the proposition, it suffices to argue locally, then apply the analogous quasi-ismorphism obtained in the rigid analytic case \eqref{projection} to the varieties $X_h$ and $X_h \setminus \wh{X}$ and we conclude using that the solid tensor product commutes with colimits. 

The arguments are similar for the distinguished triangles (1)(b) (using \eqref{detail1}) and (1)(c) (using \eqref{coffee1}), and for the quasi-isomorphism (2) (using \eqref{Boa1} and the previous remark for the compatibility with the filtrations).     
\end{proof}

\subsection{Overconvergent Hyodo-Kato cohomology with compact support} We now define overconvergent Hyodo-Kato cohomologies with compact support via similar methods. As before, we have three definitions (the arithmetic Hyodo-Kato cohomology, if $X \in {\rm Sm}_K^{\dagger}$, and the geometric and completed ones, if $X$ is defined over $C$). We will obtain Hyodo-Kato (iso)morphisms, compatible with the maps from the above proposition. 
 
\subsubsection{Definition}   
Let $L=K,C$ and $X$ be a smooth dagger affinoid variety over $L$. Starting from the overconvergent Hyodo-Kato cohomologies
\[  \rg^{\natural}_{\hk,*}(X):= \colim_h \rg_{\hk,*}(X_h) \]
from \cite[\textsection 4.2.1]{CN4}  we obtain, as in the de Rham case, the local compactly supported Hyodo-Kato cohomologies:
\[
\rg^{\natural}_{\hk,c}(X) \in {\cal D}_{\phi,N} (F_{\Box}), \text{ if } L=K, \quad \text{ and } \quad
\rg^{\natural}_{\hk,c}(X) \in {\cal D}_{\phi,N} (F^{\rm nr}_{\Box}), \text{ if } L=C, \]
and, if $X$ is over $C$, the completed Hyodo-Kato cohomology: 
\[ \rg^{\natural}_{\hk, \breve{F},c}(X) \in {\cal D}_{\phi,N} ({\breve{F}}_{\Box}). \]
These functors are  covariant for open embeddings. 

  This allows us to define cosheaves $\sa_{\hk,c}$ and $\sa_{\hk,\breve{F},c}$ on any smooth dagger variety $X$ over $L$. We define  $$ \rg_{\hk,*,c}(X):=\rg(X,\sa_{\hk,*,c}) \in {\cal D}_{\phi,N}({F_L}_{\Box}),  \text{ with } F_L=F, F^{\rm nr}, \breve{F} \text{ respectively, and } *=-,\breve{F}.$$
It is a solid dg $F_L$-algebra equipped with a Frobenius, monodromy action, and a continuous action of $\sg_K$ if $X$ is defined over $K$. If $X$ is Stein,  we have a distinguished triangle 
\begin{equation}
\label{tri-overcv2} 
\rg_{\hk,*,c}(X) \to \rg_{\hk,*}(X) \to \rg_{\hk,*}(\partial X)
\end{equation}
where $*=-,\breve{F}$ and $\rg_{\hk,*}(\partial X):=  \colim_n \rg_{\hk,*}(X\setminus U_n)$,  for a Stein covering $\{U_n\}$.

\begin{lemma}
\label{kit-1-dagger}
For any $X \in {\rm Sm}^{\dagger}_C$, we have a canonical quasi-isomorphism in ${\cal D}_{\phi,N}({\breve{F}}_{\Box})$:
\[\beta_c:\quad \rg_{\hk,c}(X)\dsolid_{F^{\nr}}\breve{F}\stackrel{\sim}{\to} \rg_{\hk,\breve{F},c}(X).\]
\end{lemma}

\begin{proof}
We can work locally.  Let $X$ be a smooth dagger affinoid over $C$.  Then, 
\begin{align*} 
\rg^{\natural}_{\hk,c}(X)\dsolid_{F_L} \breve{F} & = \colim_h [\rg_{\hk}(X_h) \to \rg_{\hk}(X_h \setminus \wh{X})] \dsolid_{F_L} \breve{F} \\
 & \xrightarrow{\sim} \colim_h [\rg_{\hk}(X_h)\dsolid_{F_L} \breve{F} \to \rg_{\hk}(X_h \setminus \wh{X})\dsolid_{F_L} \breve{F}] \\
 &  \xrightarrow[\beta]{\sim} \colim_h [\rg_{\hk, \breve{F}}(X_h) \to \rg_{\hk, \breve{F}}(X_h \setminus \wh{X})]  = \rg^{\natural}_{\hk,\breve{F},c}(X)
 \end{align*}
where $F_L=W(k_L)$. The first quasi-isomorphism is obtained by commutativity of the solid tensor product with colimits and the second one is the map $\beta$ from Lemma~\ref{kit-1}.  
\end{proof}

\subsubsection{Overconvergent Hyodo-Kato morphisms} 

Let $X$ be a smooth dagger affinoid over $K$ or $C$. The rigid analytic Hyodo-Kato morphism from Section~\ref{HK-ch} induces a morphism: 
\begin{align*}
\iota^{\natural}_{{\rm HK},c} : \rg^{\natural}_{\hk,*,c}(X) \to \rg^{\natural}_{\dr,c}(X) 
\end{align*}
coming from the commutative diagram: 
\[ \xymatrix{
\rg_{\hk,*}(X_h) \ar[r] \ar[d]^{\iota_{\hk}} & \rg_{\hk,*}(X_h\setminus \wh{X}) \ar[d]^{\iota_{\hk}} \\
\rg_{\dr}(X_h) \ar[r] & \rg_{\dr}(X_h\setminus \wh{X}),} \]
where $\{X_h\}$ is the presentation of $X$, $\wh{X}$--  the rigid analytic completion of $X$, and $*=-, \breve{F}$. 

Let $X$ be a smooth dagger variety over $K$ or $C$. The previous morphism $\iota_{\hk}^{\natural}$ is clearly functorial for open immersions hence globalizes to give a morphism 
\begin{align*}
\iota_{{\rm HK},c} : \rg_{\hk,*,c}(X) \to \rg_{\dr,c}(X) 
\end{align*}
compatible with its classical version $\iota_{\hk}:  \R\Gamma_{\hk,*}(X){\to}\R\Gamma_{\dr}(X).$

 \begin{proposition} Let $X$ be a smooth dagger variety over $C$.
 \label{HK-morph-ov}
The Hyodo-Kato morphism $\iota_{\hk,c}$ induces quasi-isomorphisms (where $F_L=F^{\nr}, \breve{F}$ and $*=-, \breve{F}$) in $\sd(C_{\Box})$ and $\sd({\B^+_{\dr,\Box}})$, respectively:
  $$
  \iota_{\hk,c}: \quad \R\Gamma_{\hk,*,c}(X)\dsolid_{F_L}C\stackrel{\sim}{\to}\R\Gamma_{\dr,c}(X),\quad   \iota_{\hk,c}: \quad \R\Gamma_{\hk,*,c}(X)\dsolid_{F_L}\B^+_{\dr}\stackrel{\sim}{\to}\R\Gamma_{\dr,c}(X/\B^+_{\dr})
  $$
  that are compatible with the maps $\theta: \B_{\dr}^+ \to C$ and $\vartheta_c$ from Section~\ref{ov-Bdr000}.
  \end{proposition}
  
  \begin{proof} 
  It suffices to argue locally. Let $X$ be a smooth dagger affinoid over $L$. Then, 
\begin{align*} 
\rg^{\natural}_{\hk,*,c}(X)\dsolid_{F_L} C & = \colim_h [\rg_{\hk,*}(X_h) \to \rg_{\hk,*}(X_h \setminus \wh{X})] \dsolid_{F_L} C \\
 & \xrightarrow{\sim} \colim_h [\rg_{\hk,*}(X_h)\dsolid_{F_L} C \to \rg_{\hk,*}(X_h \setminus \wh{X})\dsolid_{F_L} C] \\
 &  \xrightarrow[\iota_{\hk}]{\sim} \colim_h [\rg_{\dr}(X_h) \to \rg_{\dr}(X_h \setminus \wh{X})]  = \rg^{\natural}_{\dr,c}(X).
 \end{align*}
The first quasi-isomorphism is obtained by commutativity of the solid tensor product with colimits and the second one is the Hyodo-Kato morphism from Proposition~\ref{kit-2}.
 \end{proof}

 As an immediate application of the overconvergent Hyodo-Kato quasi-isomorphisms we get the local-global compatibility for (geometric and completed) Hyodo-Kato and $\B^+_{\dr}$-cohomology with compact support: 

\begin{corollary}\label{local-global-kwak}
{\rm ({Local-global compatibility})} Let $X$ be a smooth dagger affinoid over $C$. 
\begin{enumerate}
\item There is a natural quasi-isomorphism in $\sd_{\phi,N}(F_{L,\Box})$:
\begin{align*}
\rg_{\hk,*,c}(X) \to \rg^{\natural}_{\hk,*,c}(X)
\end{align*}
where $*=-,\breve{F}$.
\item The canonical morphism in $\sff\sd({\B^+_{\dr,\Box}})$
\begin{align*}
  \rg_{\dr,c}(X/\B^+_{\dr})\to \rg^{\natural}_{\dr,c}(X/\B^+_{\dr})
\end{align*}
is a quasi-isomorphism. 
\end{enumerate}
\end{corollary}

\begin{proof}
We start with the proof for the Hyodo-Kato cohomology. The proof is the same as \cite[Prop. 3.5]{CN4}. It suffices to show that for any analytic dagger affinoid hypercovering $U_{\bullet}$ of $X$, the map : 
\begin{equation}
\label{hk-bullet}
\rg^{\natural}_{\hk,c}(X) \leftarrow \rg^{\natural}_{\hk,c}(U_{\bullet})
\end{equation}
is a quasi-isomorphism. 
Consider the following commutative diagram (where $F_L=F^{\rm nr}, \breve{F}$ and $*=-, \breve{F}$) in $\sd(C_{\Box})$: 
\[ \xymatrix{ 
\rg^{\natural}_{\hk,*,c}(X) \dsolid_{F_L} C \ar[d]_{\iota_{\rm HK}^{\natural}}^{\rotatebox{90}{$\sim$}} &  \rg^{\natural}_{\hk,*,c}(U_{\bullet})\dsolid_{F_L} C \ar[d]^{\iota_{\rm HK}^{\natural}}_{\rotatebox{90}{$\sim$}} \ar[l]\\
\rg^{\natural}_{\dr,c}(X)  & \rg^{\natural}_{\dr,c}(U_{\bullet})\ar[l]_{\sim}}
\]
The  vertical arrows are quasi-isomorphisms by (the proof of) Proposition~\ref{HK-morph-ov} and the bottom horizontal one by (the proof of) Lemma \ref{pois1}. It follows that the upper horizontal arrow is a quasi-isomorphism as well. 
  
   It remains to check that this implies the quasi-isomorphism \eqref{hk-bullet}. For that,  it is enough to show that if $K^{\bullet}$ is in ${\cal D}({F_{1,\Box}})$, for $F_1=\breve{F}$ or a finite unramified extension of $F$,  such that  $K^{\bullet} \dsolid _{F_1}C$ is acyclic then $K^{\bullet}$ itself is acyclic. To see that, writing $C \simeq F_1 \oplus W$ with $W$ a $F_1$-Banach space, we obtain $$0 = H^*(K^{\bullet} \dsolid_{F_1} C)= H^*(K^{\bullet}) \oplus H^*(K^{\bullet} \dsolid_{F_1} W) $$
so $H^*(K^{\bullet})=0$, and this concludes the proof of the first point.   

The second point follows from the first one and from  Lemma~\ref{pois1} by the exact same argument as in \cite[Corollary 4.28]{CN4}. 
\end{proof} 

\subsection{Comparison with the rigid analytic constructions} We can now compare the overconvergent compactly supported cohomologies with the rigid analytic ones. 
Let $X$ be a smooth dagger variety over $L=K,C$ and let $\wh{X}$ be its completion. Then, we have:
 \begin{proposition}
 \label{rig-ov}
 Assume that $X$ is partially proper. Then,
 \begin{enumerate}
 \item There is a  natural quasi-isomorphim in $\sd_{\phi,N}(*_{\Box})$, for $*=F, F^{\nr}$, respectively:
\begin{equation}
 \rg_{\hk,c}({X})\xrightarrow{\sim} \rg_{\hk,c}(\wh{X}). \label{wrzesien1}
\end{equation}
 \item Let $L=C$. There are compatible natural quasi-isomorphisms in $\sd_{\phi,N}(\breve{F}_{\Box})$ and $\sff\sd({\B^+_{\dr,\Box}})$, respectively: 
 \begin{align*}
 \rg_{\hk,\breve{F},c}({X})\xrightarrow{\sim}\rg_{\hk,\breve{F},c}(\wh{X}),\quad 
   \rg_{\dr,c}({X}/\B^+_{\dr})\xrightarrow{\sim}\rg_{\dr,c}(\wh{X}/\B^+_{\dr}). 
 \end{align*}
 They are compatible with the map (\ref{wrzesien1}).
 \item The morphisms in (1) and (2) are compatible with the Hyodo-Kato morphisms, i.e., we have the commutative diagrams (where $*=-, \breve{F}$) in $\sd(F_{L,\Box})$ and $\sd(\breve{F}_{\Box})$, respectively: 
 $$
 \xymatrix@R=6mm{
   \rg_{\hk,*,c}({X})\ar[r]\ar[d]^{\iota_{\hk,c}} &  \rg_{\hk,*,c}(\wh{X})\ar[d]^{\iota_{\hk,c}}\\
   \rg_{\dr,c}({X})\ar[r] & \rg_{\dr,c}(\wh{X}),
 }\quad 
 \xymatrix@R=6mm{
  \rg_{\hk,\breve{F},c}({X})\ar[r]\ar[d]^{\iota_{\hk,c}} &  \rg_{\hk,\breve{F},c}(\wh{X})\ar[d]^{\iota_{\hk,c}}\\
   \rg_{\dr,c}({X}/\B^+_{\dr})\ar[r] & \rg_{\dr,c}(\wh{X}/\B^+_{\dr}).
 }
 $$
 \end{enumerate}
 \end{proposition}
 
 \begin{proof}
 These results  follow from the de Rham  and $\B^+_{\dr}$-cohomology cases (see Lemma \ref{pois1} and Lemma \ref{pois1-Bdr}) using the Hyodo-Kato quasi-isomorphisms from Proposition~\ref{kit-2} and Proposition~\ref{HK-morph-ov}.  
     \end{proof}

\section{Dualities for de Rham and Hyodo-Kato cohomologies}\label{dualities}
We prove Poincar\'e duality results for the de Rham and Hyodo-Kato cohomologies of  partially proper smooth rigid analytic varieties over $L=K,C$. The de Rham duality will follow from the coherent duality theorem for smooth Stein varieties first proved by Chiarellotto in \cite{Ch90}. We then obtain the same results for quasi-compact  smooth dagger varieties over $K$ or $C$. In the dagger affinoid case, the coherent and de Rham dualities were shown by Grosse-Kl\"onne in \cite[Th. 4.4, Th. 4.9]{GK1}.

The trace maps constructed below will allow us to define a trace map for $p$-adic pro-\'etale cohomology in Section~\ref{trace-maps}. This trace morphism will then be used in \cite{CGN} to prove Poincar\'e duality for $p$-adic pro-\'etale cohomology of rigid analytic varieties.

\subsection{Coherent duality: classical approach} We will briefly review the classical coherent dualities for rigid analytic varieties.
\subsubsection{Rigid analytic setting}\label{brompt0}
Let $L=K,C$ and let $X$ be a rigid analytic variety over $L$. In this section we will discuss the following well-known result.

\begin{theorem}{\rm ({\em Serre duality})}
\label{serre-rigid}
Let $X$ be a Stein or a proper rigid analytic variety, smooth of dimension $d$ over $L$. Then
\begin{enumerate}
\item There exists a canonical trace map ${\rm tr}_{\rm coh}: \, H_c^{d}(X, \Omega_X^d)\to L$.
\item For all coherent sheaves $\FS$ on $X$, the pairings
\begin{align*}
& H_c^i(X, \FS) \times {\rm Ext}^{d-i}_{\OS_X}(\FS,\Omega_X^{d}) \to H_c^{d}(X, \Omega_X^d) \xrightarrow{{\rm tr}_{\rm coh}} L \\
& H^i(X, \FS) \times {\rm Ext}^{d-i}_{\OS_X,c}(\FS,\Omega_X^{d}) \to H_c^{d}(X,\Omega_X^d) \xrightarrow{{\rm tr}_{\rm coh}} L 
\end{align*}
induce perfect dualities, i.e., we have the induced isomorphisms:
\begin{align*}
& {\rm Ext}^{d-i}_{\OS_X}(\FS,\Omega_X^{d}) \stackrel{\sim}{\to} \Hom_L(H_c^i(X, \FS),L) & \text{ and } &
& H_c^i(X, \FS) \stackrel{\sim}{\to} \Hom_L( {\rm Ext}^{d-i}_{\OS_X}(\FS,\Omega_X^{d}),L), \\
 & {\rm Ext}^{d-i}_{\OS_X,c}(\FS,\Omega_X^{d}) \stackrel{\sim}{\to} \Hom_L(H^i(X, \FS),L) & \text{ and } &
& H^i(X, \FS) \stackrel{\sim}{\to} \Hom_L( {\rm Ext}^{d-i}_{\OS_X,c}(\FS,\Omega_X^{d}),L). 
\end{align*}
\end{enumerate}
\end{theorem}

This result was proved for smooth Stein varieties over $C$  by Chiarellotto  in  \cite{Ch90}. This was then extended  by van der Put to  include proper rigid analytic varieties over $L$ in \cite{VdP92}.  A different proof for smooth Stein varieties over $L$ was given by  Beyer in  \cite{Bey97}. In the following we will recall the strategy of the proof of \cite{Ch90},\cite{VdP92} and in the next section we will reformulate their main results in the condensed setting.

\vskip.2cm
(i) {\em Duality for the affine space.} Let  ${\mathbb A}^d_L$ be  the rigid affine space of dimension $d$ over $L$. Cover ${\mathbb A}^d_L$ with the discs:
\[ D_j := \{ x=(x_1, \cdots, x_d) \in {\mathbb A}^d_L  \; | \; |x_i| \le {p^j} \text{ for } 1 \le i \le d\} \]
and recall that the compactly supported cohomology of $\Omega^d$ can be computed using the distinguished triangle: 
\begin{equation}
\label{brompt1}
 \rg_c({\mathbb A}_L^d, \Omega^d) \to \rg({\mathbb A}_L^d, \Omega^d) \to \colim_j \rg({\mathbb A}_L^d \setminus D_j,\Omega^d).
\end{equation}

\begin{lemma}
\label{aff-vanish}
We have $H^i_c({\mathbb A}^d_L, \Omega^d)=0$ for   $i \neq d$ and for $i=d$:
\[ H^d_c({\mathbb A}^d_L, \Omega^d) \simeq \{ \sum_{ \nu >0} \frac{ a_{\nu}}{x^{\nu}} dx_1 \wedge \cdots \wedge dx_d \; | \; \exists j \in \N \text{ such that } \lim_{|\nu| \to \infty}  p^{-j| \nu |}|a_{\nu}|=0\}, \]
where $\nu=(\nu_1, \cdots, \nu_d) \in \Z^d$ and $x^{\nu}=x_1^{\nu_1}x_2^{\nu_2} \cdots x_d^{\nu_d}$.  
\end{lemma}

\begin{proof}
The result follows from the long exact sequence \eqref{brompt1} using that  we have
$H^i({\mathbb A}_L^d, \Omega^d) = 0$, for  $i \neq 0$, and
\[H^i({\mathbb A}_L^d \setminus D_j, \Omega^d) \simeq \begin{cases}  H^0({\mathbb A}_L^d, \Omega^d) &\text{ if } i=0 ,\\ \{ \sum_{ \nu >0} \frac{ a_{\nu}}{x^{\nu}} dx_1 \wedge \cdots \wedge dx_d \; | \; \lim_{|\nu| \to \infty}  p^{-j| \nu |}|a_{\nu}|=0 \} &\text{ if } i=d-1, \\0 &\text{ if } i\neq 0, d-1. \end{cases}  \]
To prove the latter formula, consider the admissible covering of ${\mathbb A}_L^d\setminus D_j$ by the open subsets $\{U_i\}_{1 \le i \le d}$ where 
\[ U_i:=\{x=(x_1, \cdots, x_d) \in {\mathbb A}^d_L  \; | \; |x_i| > \frac{1}{p^j}\}, \quad \text{ for } i \in \{ 1, \cdots, d\} ,\]
and use the associated \v{C}ech cohomology. 
\end{proof}

We define the trace map for the affine space as: 
\begin{align}
\label{trace-aff}
{\rm tr}_{\rm coh}: \qquad  H^d_c({\mathbb A}_L^d, \Omega^d) \quad& \to  L \\
 \sum_{ \nu >0} \frac{ a_{\nu}}{x^{\nu}}dx  & \mapsto  a_{(1, \dots, 1)} \notag
\end{align}
where $\nu=(\nu_1, \cdots, \nu_d) \in \Z^d$, $x^{\nu}=x_1^{\nu_1}x_2^{\nu_2} \cdots  x_d^{\nu_d}$ and $dx:=dx_1 \wedge \cdots \wedge dx_d$.

\begin{proposition}\cite[(2.5) and (2.6)]{VdP92}
\label{dual-aff}
Theorem~\ref{serre-rigid} is true for $X={\mathbb A}_L^d$. 
\end{proposition}

\begin{proof}
The trace map is given by the formula \eqref{trace-aff}. 

First assume that $\FS= \OS_X$. For $i \neq d$, both terms of the product are zero, so the pairing is perfect. For $i=d$, the pairing is given by 
\begin{align*}
H_c^d(X, \OS_X) & \times H^0(X,\Omega_X^{d})  \to  L \\
(\sum_{\nu > 0} \frac{ a_{\nu}}{x^{\nu}} & ,\sum_{\nu \ge 0} b_{\nu} x^{\nu} dx) \mapsto \sum_{\nu\ge 0} a_{\nu+1} b_{\nu}
\end{align*}
and it induces a perfect duality. This proves the proposition for $\FS= \OS_X$ and we can deduce from it the result, first,  for $\FS= \OS_X^m$ and then for a general coherent sheaf. 
\end{proof}

\vskip.2cm
(ii) {\em Duality for Stein varieties.} 
We can now deduce from the previous proposition the duality result for general smooth  Stein varieties. Assume that $X$ is Stein and choose a closed immersion $i: X \hookrightarrow {\mathbb A}^N_L$. We have the following classical results: 

\begin{lemma}\cite[(4.15)]{Ch90}
\label{i-coh1}
There is a canonical quasi-isomorphism $\Omega_X^d[d-N] \simeq \R i^{!}\Omega_{{\mathbb A}^N_C}^N$.
\end{lemma} 

\begin{corollary}
\label{i-coh2}
\begin{enumerate}
\item There is a canonical quasi-isomorphism $i_*\Omega_X^d[d-N] \simeq \R\Hom_{\OS_{{\mathbb A}_C^N}}(i_* \OS_X, \Omega_{{\mathbb A}^N_C}^N)$.
\item For all coherent sheaves $\FS$ on $X$, there is a canonical isomorphism $H_c^i(X, \FS) \xrightarrow{\sim} H_c^i({\mathbb A}_C^N, i_*\FS)$, for all $i \ge 0$. 
\item For all coherent sheaves $\FS$ on $X$, there is a canonical quasi-isomorphism $$\R\Hom_{\OS_X}(\FS, \Omega_X^d)[d-N] \xrightarrow{\sim} \R\Hom_{\OS_{{\mathbb A}^N_C}}(i_*\FS, \Omega_{{\mathbb A}_C^N}^N).$$
\end{enumerate} 
\end{corollary}

\begin{proof}
The first claim follows from Lemma~\ref{i-coh1} using the fact  that $i^{!}$ is right adjoint to $i_*$.

To prove the claim (2), note that the result is true for the usual cohomology (without compact support). We deduce the isomorphism for the compactly supported cohomology using that for each open affinoid $V$ of $X$, we can find a finite set of open affinoids $Z_1, \cdots, Z_n$ in the affine space such that $V$ is the intersection of the $Z_i$'s with $X$ (see \cite[(4.17)]{Ch90}).

For the third claim, consider an injective resolution $\IS_{\bullet}$ of $\Omega_{{\mathbb A}_C^N}^N$. Then 
\[ \Hom_{\OS_X}(i_*\FS, \IS^{\bullet}) \simeq \Hom_{\OS_{{\mathbb A}^N_C}}(\FS, i^{!}\IS^{\bullet}). \]
We conclude using that the functor $i^!$ sends injectives  to injectives and applying Lemma~\ref{i-coh1}.
\end{proof}

  We define the trace map for a smooth Stein variety $X$ as the composition: 
\begin{align*}
H^d_c(X, \Omega^d_X) \xrightarrow{\sim} H^d_c({\mathbb A}^N_C,& i_*\Omega_X^d) 
 \simeq {\rm Ext}^N_{\OS_X,c}(i_*\OS_X, \Omega_{{\mathbb A}^N_C}^N) \xrightarrow{\sim} \Hom_L(H^0({\mathbb A}_C^N, i_*\OS_X),L) \to L,   
\end{align*}
where the first and second isomorphisms follow from the previous lemma and the third one from the duality for the affine space.

\begin{lemma} {\rm({\em  \cite[3.5, 3.7]{VdP92}})}
\label{tr-uniq}
Let $X$ be a smooth Stein variety over $L$, of dimension $d$. Then,
\begin{enumerate}
\item  The trace map does not depend of the chosen closed immersion $X \hookrightarrow {\mathbb A}^N_L$.
\item If $X \to Y$ is an open immersion of smooth Stein varieties over $L$ of dimension $d$, the composition 
\[ H^d_c(X, \Omega^d) \to H^d_c(Y, \Omega^d) \xrightarrow{{\rm tr}_{{\rm coh},Y}} L \] 
is equal to ${\rm tr}_{{\rm coh},X}: H^d_c(X, \Omega^d) \to L$.  
\end{enumerate}
\end{lemma}

\begin{proposition}
\label{dual-stein}
The theorem~\ref{serre-rigid} is true for smooth Stein varieties. 
\end{proposition}

\begin{proof} Let $X$ be a smooth Stein variety. Our theorem
 follows from the duality for the affine space (Proposition~\ref{dual-aff}) using the isomorphisms from Corollary~\ref{i-coh2}: 
\begin{align*}
H_c^i(X, \FS) \xrightarrow{\sim} H_c^i({\mathbb A}_C^N, i_*\FS) \quad \text{ and }& {\rm Ext}_{\OS_X}^{d-i}(\FS, \Omega^d_X) \xrightarrow{\sim} {\rm Ext}_{\OS_{{\mathbb A}_C^N}}^{N-i}(i_*\FS, \Omega^N_{{\mathbb A}^N_C}) \\
H^i(X, \FS) \xrightarrow{\sim} H^i({\mathbb A}_C^N, i_*\FS) \quad \text{ and }& {\rm Ext}_{\OS_X,c}^{d-i}(\FS, \Omega^d_X) \xrightarrow{\sim} {\rm Ext}_{\OS_{{\mathbb A}_C^N},c}^{N-i}(i_*\FS, \Omega^N_{{\mathbb A}^N_C}).
\end{align*} 
\end{proof}

In particular, we obtain:  
\begin{proposition}
\label{comp-vanish}
Let $X$ be a smooth Stein variety of dimension $d$ over $L$ and let $\FS$ be a locally free sheaf  on $X$. Then $H^i_c(X, \FS)=0$ if $i \neq d$. 
\end{proposition}

\subsubsection{Overconvergent setting} In the dagger case the coherent duality was studied by Grosse-Kl\"onne. We will review it briefly. 
\begin{theorem}\cite[Th. 4.4]{GK1}
\label{serre-dagger}
Let $X$ be smooth dagger affinoid of dimension $d$ over $L$. Then
\begin{enumerate}
\item There exists a canonical trace map ${\rm tr}_{\rm coh}: \, H_c^{d}(X, \Omega_X^d)\to L$.
\item For all coherent sheaves $\FS$ on $X$, the pairings
\begin{align*}
& H_c^i(X, \FS) \times {\rm Ext}^{d-i}_{\OS_X}(\FS,\Omega_X^{d}) \to H_c^{d}(X, \Omega_X^d) \xrightarrow{{\rm tr}_{\rm coh}} L \\
& H^i(X, \FS) \times {\rm Ext}^{d-i}_{\OS_X,c}(\FS,\Omega_X^{d}) \to H_c^{d}(X, \Omega_X^d) \xrightarrow{{\rm tr}_{\rm coh}} L 
\end{align*}
induce perfect dualities, i.e., we have the induced isomorphisms of $L$-vector spaces:
\begin{align*}
& {\rm Ext}^{d-i}_{\OS_X}(\FS,\Omega_X^{d}) \stackrel{\sim}{\to} \Hom_L(H_c^i(X, \FS),L) & \text{ and } &
& H_c^i(X, \FS) \stackrel{\sim}{\to} \Hom_L( {\rm Ext}^{d-i}_{\OS_X}(\FS,\Omega_X^{d}),L), \\
 & {\rm Ext}^{d-i}_{\OS_X,c}(\FS,\Omega_X^{d}) \stackrel{\sim}{\to} \Hom_L(H^i(X, \FS),L) & \text{ and } &
& H^i(X, \FS) \stackrel{\sim}{\to} \Hom_L( {\rm Ext}^{d-i}_{\OS_X,c}(\FS,\Omega_X^{d}),L). 
\end{align*}
\end{enumerate}
\end{theorem}

The proof is similar to the one in the rigid case. We will first prove the result for $X:={\rm Sp}(L[X_1, \cdots, X_d]^{\dagger})$ and then consider a closed immersion $i: X \hookrightarrow {\rm Sp}(L[X_1, \cdots, X_d]^{\dagger})$ to deduce the duality for a general dagger affinoid $X$. 
  
\vskip.2cm
(i) {\em Duality for $X:={\rm Sp}(L[X_1, \cdots, X_d]^{\dagger})$.} As in the affine case, we can easily compute the compactly supported cohomology for the sheaf of differentials $\Omega^d$. This will allow us to define the trace map. 

\begin{lemma}
\label{tate-vanish}
We have $H^i_c(X, \Omega^d)=0$ if $i \neq d$ and for $i=d$ we have
\[ H^d_c(X, \Omega^d) \simeq \{ \sum_{ \nu >0} \frac{ a_{\nu}}{x^{\nu}} dx_1 \wedge \cdots \wedge dx_d \; | \; \lim_{|\nu| \to \infty}  \delta^{| \nu |}|a_{\nu}|=0 \text{ for all } \delta >1 \}\]
where $\nu=(\nu_1, \cdots, \nu_d) \in \Z^d$ and $x^{\nu}=x_1^{\nu_1}x_2^{\nu_2} \cdots x_d^{\nu_d}$. 
\end{lemma}

\begin{proof}
The proof is similar to the one of Lemma~\ref{aff-vanish}: if we write $X_{\delta}:={\rm Sp}(L[\delta^{-1}X_1, \cdots, \delta^{-1}X_d])$, for $\delta >1$, then we can use the long exact sequence: 
\[ \cdots \to \colim_{\delta} H^{i-1}(X_{\delta} \setminus \wh{X}, \Omega^d) \to H^i_c(X, \Omega^d) \to \colim_{\delta} H^i(X_{\delta}, \Omega^d) \to \colim_{\delta} H^{i-1}(X_{\delta}\setminus \wh{X}, \Omega^d) \to \cdots \]
to compute $H^i_c(X, \Omega^d)$ (using that $H^i(X_{\delta}, \Omega^d)=0$ for $i\neq0$). 
\end{proof}

We define the trace map for $X={\rm Sp}(L[X_1, \cdots, X_d]^{\dagger})$ as 
\begin{align}
\label{trace-dagger}
{\rm tr}_{\rm coh}: \qquad  H^d_c(X, \Omega^d) \quad& \to  L,\\
 \sum_{ \nu >0} \frac{ a_{\nu}}{x^{\nu}}dx  & \mapsto  a_{(1, \dots, 1)}, \notag
\end{align}
where $\nu=(\nu_1, \cdots, \nu_d) \in \Z^d$, $x^{\nu}=x_1^{\nu_1}x_2^{\nu_2} \cdots x_d^{\nu_d}$ and $dx:=dx_1 \wedge \cdots \wedge dx_d$.

\begin{proposition}
\label{dual-dagger1}
Theorem~\ref{serre-dagger} is true for $X={\rm Sp}(L[X_1, \cdots, X_d]^{\dagger})$. 
\end{proposition}

\begin{proof}
The trace map is given by the formula \eqref{trace-dagger}. The proof of the duality is the same as the one of Proposition~\ref{dual-aff}: it is enough to prove the result for $\FS=\OS_X$ and in that case the pairing can be written explicitly. 
\end{proof} 

\vskip.2cm
(ii) {\em Duality for dagger affinoids.} Let $X$ be a smooth dagger affinoid of dimension $d$. By definition, we can find a closed immersion $i$ of $X$ into the Tate algebra $Y:={\rm Sp}(L[X_1, \cdots, X_d]^{\dagger})$. All the results concerning the closed immersion $i$ from Section~\ref{brompt0} are still true in this case (but simpler, since the immersion that we have now preserves the dimension). We define the trace map for $X$ as the composition: 
\begin{align*}
H^d_c(X, \Omega^d_X) \xrightarrow{\sim} H^d_c(Y,& i_*\Omega_X^d)  \simeq {\rm Ext}^d_{\OS_X,c}(i_*\OS_X, \Omega_{Y}^d) \xrightarrow{\sim} \Hom_L(H^0(Y, i_*\OS_Y),L) \to L.   
\end{align*}
Then, Theorem~\ref{serre-dagger} follows from Proposition~\ref{dual-dagger1} the same way as Proposition~\ref{dual-stein} follows from Proposition~\ref{dual-aff}. 

In particular, we obtain: 
\begin{proposition}
\label{dagger-vanish}
Let $X$ be a smooth dagger affinoid of dimension $d$ over $L$, and let $\FS$ be a locally free  sheaf on $X$. Then $H^i_c(X, \FS)=0$ if $i \neq d$. 
\end{proposition}

Moreover, as in the rigid case, we have that the trace map is compatible with open immersions of dagger affinoids: 

\begin{lemma}
\label{tr-uniq00}
Let $X$ be a smooth dagger affinoid over $L$, of dimension $d$. Then,
\begin{enumerate}
\item The trace map does not depend of the chosen closed immersion $X \hookrightarrow {\rm Sp}(L[X_1, \cdots, X_d]^{\dagger})$.
\item If $X \to Y$ is an open immersion of dagger smooth affinoids over $L$ of dimension $d$, the composition 
\[ H^d_c(X, \Omega^d) \to H^d_c(Y, \Omega^d) \xrightarrow{{\rm tr}_{{\rm coh},Y}} L \] 
is equal to ${\rm tr}_{{\rm coh},X}: H^d_c(X, \Omega^d) \to L$.  
\end{enumerate}
\end{lemma}

\begin{proof}The proof can be reduced to the Stein case, i.e., to Lemma \ref{tr-uniq}.
\end{proof}
\begin{remark}({\em Topological issues.})
All the terms in Theorem \ref{serre-rigid} and Theorem \ref{serre-dagger} are equipped with a natural topology and the dualities are topological. See \cite[Th. 4.21]{Ch90} for statements and discussions concerning these facts. 
\end{remark}

\subsection{Coherent duality: condensed approach}\label{serre-condensed0} The classical coherent duality results surveyed above can be paraphrased in the condensed mathematics language.

\begin{corollary}\label{serre-condensed}Let $X$ be a smooth partially proper rigid analytic variety over $L$, of dimension $d$. Then:
\begin{enumerate}
\item There are trace maps in $L_{\Box}$ and $\sd(L_{\Box})$, respectively:
$$\tr_{\rm coh}: H^d_c(X,\Omega^d_X)\to L,\quad {\rm tr}_{\rm coh}: \R\Gamma_c(X,\Omega^d)[d] \to L.
$$
They are compatible with open immersions. 
\item If $X$ is Stein or proper, the pairing
$$
H^i(X,\Omega^j_X)\otimes^{\LL_{\Box}}_LH^{d-i}_c(X,\Omega^{d-j}_X)\to H^{d}_c(X,\Omega^{d}_X)\lomapr{\tr_{\rm coh}} L 
$$
is perfect, i.e., it yields the quasi-isomorphisms in $\sd(L_{\Box})$
\begin{align}\label{52-cafe}
& \Omega^j(X)\stackrel{\sim}{\to}\R\Hom_{L_{\Box}}(H^{d}_c(X,\Omega^{d-j}_X),L),\quad \Omega^j(X)\stackrel{\sim}{\to}\Hom_{L_{\Box}}(H^{d}_c(X,\Omega^{d-j}_X),L),\\
& H^d_c(X,\Omega^j_X)\stackrel{\sim}{\to}\Hom_{L_{\Box}}(\Omega^{d-j}(X),L).\notag
\end{align}
\item  The pairing
$$
\R\Gamma(X,\Omega^j_X)\otimes^{\LL_{\Box}}_L\R\Gamma_c(X,\Omega^{d-j}_X)\to \R\Gamma_c(X,\Omega^{d}_X)\lomapr{\tr_{\rm coh}} L [-d]
$$
is perfect, i.e., it yields the quasi-isomorphism in $\sd(L_{\Box})$
\begin{align*}
\R\Gamma(X,\Omega^j_X)\simeq \R\Hom_{L_{\Box}}(\R\Gamma_c(X,\Omega^{d-j}_X)[d],L)
\end{align*}

\end{enumerate}
\end{corollary}
\begin{proof} For claim (1), assume first that $X$ is Stein. Recall that the locally convex spaces $H^i_c(X,\Omega^j)$ and $H^i(X,\Omega^j)$ are of compact type and nuclear Fr\'echet (both in the classical sense), respectively. We can see them inside the category of solid $L$-modules. Then the claim (1) follows from Proposition \ref{dual-stein}.   For a general  smooth and partially proper $X$  take an admissible covering $\{U_i\}$, $i\in I$,  of $X$ such that the $U_i$'s are Stein. We have the exact sequence:
\begin{align*}
 \bigoplus_{i,j \in I} H^d_c(U_i \cap U_j, \Omega^d) \to \bigoplus_{i\in I} H_c^d(U_i, \Omega^d) \to H^d_c(X, \Omega^d)  \to 0. 
\end{align*}
We define the trace map ${\rm tr}_{\rm coh}: H^d_c(X, \Omega_X^d) \to L$
using the above exact sequence  and the fact that the construction of the coherent trace for Stein varieties is compatible with the open immersions $U_i \cap U_j \to U_i$ (Lemma~\ref{tr-uniq}).  The derived trace ${\rm tr}_{\rm coh}: \R\Gamma_c(X,\Omega^d) \to L$ is defined as the composition
$$
{\rm tr}_{\rm coh}: \R\Gamma_c(X,\Omega^d)\stackrel{\sim}{\leftarrow}\tau_{\leq d}\R\Gamma_c(X,\Omega^d)\to  H^d_c(X, \Omega_X^d)[-d] \lomapr{{\rm tr}_{\rm coh}} L[-d].
$$
For the last statement of claim (1) we have the following:
\begin{lemma}\label{coherent-pp}
Let $j:X\to Y$ be an open immersion of smooth partially proper rigid analytic varieties of dimension $d$ over $L$. The coherent trace map  is compatible with $j_*$, i.e., 
the following diagram is commutative: 
$$
 \xymatrix{ H^d_c(X, \Omega_X^d) \ar[r]^{j_*} \ar[d]_{{\rm tr}_{{\rm coh},X}} & H^d_c(Y, \Omega_Y^d) \ar[dl]^{{\rm tr}_{{\rm coh},Y}} \\
  L. }
$$
\end{lemma}
\begin{proof}
If $X,Y$ are Stein this follows from Lemma \ref{tr-uniq}. 

 For a general partially proper $Y$, take an  admissible covering $\{U_i\}$ of $Y$ by Stein varieties. Recall that   the trace map $\tr_{Y}: H^d_c(Y, \Omega_Y^d) \to L$ comes from the maps $\tr_{U_i}: H^d_c(U_i, \Omega_{U_i}^d) \to L$. Consider the diagram: 
\[ \xymatrix{ 
\bigoplus_i H^d_c(U_i \times_Y X, \Omega^d) \ar[rrr]  \ar[rdd]_{\oplus_i {\rm tr}_{U_i\times X}} \ar@{->>}[rd] & & & \bigoplus_i H^d_c(U_i, \Omega^d) \ar@{->>}[ld] \ar[ldd]^{\oplus_i {\rm tr}_{U_i}} \\
& H^d_c(X, \Omega^d) \ar[r] \ar[d]_{{\rm tr}_X}& H^d_c(Y, \Omega^d) \ar[d]^{{\rm tr}_Y}& \\
& C \ar@{=}[r] & C &}. \]
We want to show that the central square is commutative. Since the maps from the direct sums to $H^d_c(X, \Omega^d)$ and $H^d_c(Y, \Omega^d)$ are surjective, it suffices to show that the upper and outer trapezoids as well as the right and left triangles are commutative. 

The right triangle is commutative by definition of the trace map for $Y$; for the upper trapezoid, it is by definition of the pushforward map.  For the left triangle and the outer trapezoid take an admissible covering $\{U_{i,j}\}$ of each $U_i\times_Y X$ by Stein varieties and thus reduce to $U_i\times_Y X$ itself being Stein; then commutativity  follows from the previous point. 
\end{proof}

   For the claim (2), we again use the cited proposition and the fact that 
we have $$\R\Hom_{L_{\Box}}(H^{2d-i}_c(X,\Omega^{d-j}_X),L)\stackrel{\sim}{\leftarrow} \Hom_{L_{\Box}}(H^{2d-i}_c(X,\Omega^{d-j}_X),L)$$ because $H^i_c(X,\Omega^j)$ is of compact type.

 For claim (3),   we cover $X$  with Stein varieties $\{X_i\}, i\in I$, and consider the associated \v{C}ech hypercovering $Y_{\jcdot}$ of $X$.  Then we compute in 
 $\sd(L_{\Box})$
   \begin{align*}
{\R\Gamma}(X,\Omega^j)& \simeq\R\lim_{i} {\R\Gamma}(Y_{i},\Omega^j)\simeq\R\lim_{i}\R\Hom_{L_{\Box}}({\R\Gamma}_{c}(Y_{i},\Omega^{d-j})[d],L)\\
& \simeq\R\Hom_{L_{\Box}}(\colim_{i}{\R\Gamma}_{c}(Y_{i},\Omega^{d-j})[d],L)\simeq\R\Hom_{L_{\Box}}({\R\Gamma}_{c}(X,\Omega^{d-j})[d],L),
   \end{align*}
as wanted. Here, the second quasi-isomorphism follows from the duality for Stein varieties stated in claim (2).

\end{proof}

  We have an analogous duality result  in the overconvergent setting:
\begin{corollary}\label{serre-condensed-over}Let $X$ be a smooth dagger variety over $L$, of dimension $d$. Then:
\begin{enumerate}
\item There is a trace map of solid $L$-modules
$$\tr_{\rm coh}: H^d_c(X,\Omega^d_X)\to L. $$
\item If $X$ is an affinoid then the pairing
$$
H^i(X,\Omega^j)\otimes^{\LL_{\Box}}_LH^{2d-i}_c(X,\Omega^{d-j}_X)\to H^{2d}_c(X,\Omega^{d}_X)\lomapr{\tr_{\rm coh}} L 
$$
is perfect, i.e., it yields the quasi-isomorphisms in $\sd(L_{\Box})$
\begin{align}\label{52-cafe-over}
& H^i_c(X,\Omega^j)\stackrel{\sim}{\to}\R\Hom_{L_{\Box}}(H^{2d-i}(X,\Omega^{d-j}_X),L),\quad H^i_c(X,\Omega^j)\stackrel{\sim}{\to}\Hom_{L_{\Box}}(H^{2d-i}(X,\Omega^{d-j}_X),L),\\
& H^i(X,\Omega^j)\stackrel{\sim}{\to}\Hom_{L_{\Box}}(H^{2d-i}_c(X,\Omega^{d-j}_X),L).\notag
\notag
\end{align}
\item  If $X$ is quasi-compact, the pairing
$$
\R\Gamma(X,\Omega^j_X)\otimes^{\LL_{\Box}}_L\R\Gamma_c(X,\Omega^{d-j}_X)\to \R\Gamma_c(X,\Omega^{d}_X)\lomapr{\tr_{\rm coh}} L [-d]
$$
is perfect, i.e., it yields the quasi-isomorphism in $\sd(L_{\Box})$
\begin{align*}
\R\Gamma_c(X,\Omega^j_X)\simeq \R\Hom_{L_{\Box}}(\R\Gamma(X,\Omega^{d-j}_X)[d],L)
\end{align*}

\end{enumerate}
\end{corollary}
\begin{proof}The proof of (1) and (2) is the mirror image of the proof of Corollary \ref{serre-condensed}, using Lemma \ref{tr-uniq00} and  the fact that if $X$ is an affinoid then   the locally convex spaces $H^i(X,\Omega^j)$ and $H^i_c(X,\Omega^j)$ are of compact type and nuclear Fr\'echet (both in the classical sense), respectively.

  For claim (3),   we cover $X$  with affinoid dagger  varieties $\{X_i\}, i\in I$, and consider the associated \v{C}ech hypercovering $Y_{\jcdot}$ of $X$.  Then we compute in 
 $\sd(L_{\Box})$
   \begin{align*}
{\R\Gamma}_c(X,\Omega^j)& \simeq\colim_{i} {\R\Gamma_c}(Y_{i},\Omega^j)\simeq\colim_{i}\R\Hom_{L_{\Box}}({\R\Gamma}(Y_{i},\Omega^{d-j})[d],L)\\
& \simeq\R\Hom_{L_{\Box}}(\R\lim_{i}{\R\Gamma}(Y_{i},\Omega^{d-j})[d],L)\simeq\R\Hom_{L_{\Box}}({\R\Gamma}(X,\Omega^{d-j})[d],L),
   \end{align*}
as wanted. Here, the second quasi-isomorphism follows from the duality for affinoid dagger  varieties stated in claim (2) and the third quasi-isomorphism  from the fact that the limits and colimits are finite. 
 \end{proof}

 \subsection{de Rham dualities}
 We state the dualities satisfied by  de Rham and Hyodo-Kato cohomologies of  smooth partially proper rigid analytic varieties and smooth dagger affinoids.  We deduce the de Rham duality from coherent duality and the Hyodo-Kato duality is  obtained from the de Rham duality  using the Hyodo-Kato isomorphism.

 \subsubsection{Rigid analytic setting}
We can now prove the duality result satisfied by the de Rham cohomology:
\begin{theorem}{\rm ({\em de Rham duality})}
 \label{dualdR-rigid} 
Let $L=K,C$ and $X$ be a smooth partially proper  variety of dimension $d$ over $L$. Then: 
\begin{enumerate}
\item
There are natural trace maps in $\sd(L_{\Box})$ and $L_{\Box}$, respectively:
$${\rm tr}_{\dr}: \, \R\Gamma_{\dr,c}(X)[2d]\to L,\quad {\rm tr}_{\dr}: \, H^{2d}_{\dr,c}(X)\to L. 
$$
\item
The  pairing
$$
\R\Gamma_{\dr}(X)\otimes^{\LL_{\Box}}_L \R\Gamma_{\dr,c}(X)[2d]\to \R\Gamma_{\dr,c}(X)[2d]\xrightarrow{{\rm tr}_{\dr}} L
$$
is a perfect duality, i.e., we have induced quasi-isomorphism in $\sd(L_{\Box})$
\begin{align*}
& \R\Gamma_{\dr}(X)\stackrel{\sim}{\to} \R\Hom_{L_{\Box}} (\R\Gamma_{\dr,c}(X)[2d],L).
\end{align*}
\item The  pairing
$$
H^i_{\dr}(X)\otimes^{\LL_{\Box}}_L H^{2d-i}_{\dr,c}(X)\to H^{2d}_{\dr,c}(X)\xrightarrow{{\rm tr}_{\dr}} L
$$
is a perfect duality of solid $L$-vector spaces, i.e., we have induced isomorphisms in $L_{\Box}$
\begin{align*}
& H^i_{\dr}(X)\stackrel{\sim}{\to} \Hom_{L_{\Box}} (H^{2d-i}_{\dr,c}(X),L),\\
& H^i_{\dr,c}(X)\stackrel{\sim}{\to} \Hom_{L_{\Box}} (H^{2d-i}_{\dr}(X),L).
\end{align*}
\end{enumerate}
\end{theorem}
\begin{proof} In the case $X$ is Stein,  
the trace map ${\rm tr}_{\dr}: H^{2d}_{\dr,c}(X)\to L$ can be  induced by the coherent trace map ${\rm tr}_{\rm coh}: H^d_c(X, \Omega_X^d) \to L$ from 
Corollary \ref{serre-condensed}, noting that ${\rm tr}_{\rm coh}(\partial)=0$ where $\partial: H^{d}_c(X, \Omega_X^{d-1}) \to H^d_c(X, \Omega_X^d)$ (see \cite[Lemma 2.1.4]{Bey97}). Its derived version ${\rm tr}_{\dr}: \R\Gamma_{\dr,c}(X)[2d]\to L$ is defined as the composition
$$
{\rm tr}_{\dr}:\quad \R\Gamma_{\dr,c}(X)[2d]\stackrel{\sim}{\leftarrow}\tau_{\leq 2d}\R\Gamma_{\dr,c}(X)[2d]\to H^{2d}_{\dr,c}(X)\lomapr{{\rm tr}_{\rm dR}}L.
$$
In the case of general partially proper $X$ we reduce to the Stein case as in the proof of Corollary~\ref{serre-condensed}.

 In the case $X$ is Stein,  claim (2) follows immediately from the sequence representing $\R\Gamma_{\dr}(X)$:
 \begin{equation}\label{buis1}
 \so(X)\to \Omega^1(X)\to\cdots\to\Omega^d(X),
 \end{equation}
   the sequence \eqref{buis2} representing $\R\Gamma_{\dr,c}(X)$, and derived Serre duality \eqref{52-cafe}. For a general partially proper $X$,  we cover it with Stein varieties $\{X_i\}, i\in I$, and consider the associated \v{C}ech hypercovering $Y_{\jcdot}$ of $X$.  Then we compute in $\sd(L_{\Box})$
   \begin{align*}
{\R\Gamma}_{\dr}(X)& \simeq\R\lim_{i} {\R\Gamma}_{\dr}(Y_{i})\simeq\R\lim_{i}\R\Hom_{L_{\Box}}({\R\Gamma}_{\dr,c}(Y_{i})[2d],L)\\
& \simeq\R\Hom_{L_{\Box}}(\colim_{i}{\R\Gamma}_{\dr,c}(Y_{i})[2d],L)\simeq\R\Hom_{L_{\Box}}({\R\Gamma}_{\dr,c}(X)[2d],L),
   \end{align*}
as wanted. Here, the second quasi-isomorphism follows from the duality for Stein varieties.
   
   The first isomorphism in claim (3) follows from claim (2) by taking cohomology. We use here the fact that $\Ext^j_{L_{\Box}}(H^{m}_{\dr,c}(X),L)=0$, for $j\geq 1$, because $H^{m}_{\dr,c}(X)$ is of compact type (in fact a countable direct sum of finite rank vector spaces).  The second isomorphism follows from the computation: 
\begin{align*}
 \Hom_{L_{\Box}} (H^{2d-i}_{\dr}(X),L)\stackrel{\sim}{\leftarrow}  \Hom_{L_{\Box}} (\Hom_{L_{\Box}} (H^{i}_{\dr,c}(X),L),L)\stackrel{\sim}{\to} H^{i}_{\dr,c}(X),
\end{align*}
where the first isomorphism was just proven and the second one follows from reflexivity of $H^{i}_{\dr,c}(X)$.
\end{proof}
\begin{lemma}\label{dR-pp}
Let $j:X\to Y$ be an open immersion of smooth partially proper rigid analytic varieties of dimension $d$ over $L$. The de Rham trace map  is compatible with $j_*$, i.e., 
the following diagram is commutative: 
$$
 \xymatrix{ H^{2d}_{\dr,c}(X) \ar[r]^{j_*} \ar[d]_{{\rm tr}_{\dr,X}} & H^{2d}_{\dr,c}(Y) \ar[dl]^{{\rm tr}_{\dr,Y}} \\
  L. }
$$
\end{lemma}
\begin{proof} This follows immediately from the definition of the de Rham trace map and an  analogous claim for coherent traces (see Lemma \ref{coherent-pp}).
\end{proof}
 
 \begin{corollary}{\rm ({\em Filtered de Rham duality})}
Let $L=K,C$ and $X$ be a smooth partially proper variety of dimension $d$ over $L$.
Let $r,r^{\prime}\in\N, r+r^{\prime}=d$. The  pairing in $\sd(L_{\Box})$
$$
(\R\Gamma_{\dr}(X)/F^{r^{\prime}})\otimes^{\LL_{\Box}}_L F^{r+1}\R\Gamma_{\dr,c}(X)[2d]\to \R\Gamma_{\dr,c}(X)[2d]\xrightarrow{{\rm tr}_{\dr}} L
$$
is a perfect duality, i.e., we have induced quasi-isomorphisms in $\sd(L_{\Box})$
\begin{align*}
& \R\Gamma_{\dr}(X)/F^{r^{\prime}}\stackrel{\sim}{\to} \R\Hom_{L_{\Box}} (F^{r+1}\R\Gamma_{\dr,c}(X)[2d],L),\\
& F^{r^{\prime}}\R\Gamma_{\dr}(X)\stackrel{\sim}{\to} \R\Hom_{L_{\Box}} (\R\Gamma_{\dr,c}(X)/F^{r+1}[2d],L).
\end{align*}
\end{corollary}
\begin{proof} Assume first that $X$ is Stein. 
We compute: 
\begin{align*}
\R\Hom_{L_{\Box}}(F^{r^{\prime}}\R\Gamma_{\dr,c}(X),L)  & \simeq
 (H^d_c(X,\Omega^{r^{\prime}})^*\leftarrow H^d_c(X,\Omega^{r^{\prime}+1})^*\leftarrow \cdots \leftarrow H^d_c(X,\Omega^d)^*)[-d-r^{\prime}]\\
& \simeq \so(X)\to\Omega^1(X)\to\cdots\to \Omega^{r}(X)\\
&\simeq \R\Gamma_{\dr}(X)/F^{r+1}.
\end{align*}
We wrote here $(-)^*$ for $\Hom_{L_{\Box}}(-,L)$ and used the fact that $$\Hom_{L_{\Box}}(H^d_c(X,\Omega^{i}),L)\stackrel{\sim}{\to}\R\Hom_{L_{\Box}}(H^d_c(X,\Omega^{i}),L)$$ because the cohomology groups with compact support are of compact type.
The second quasi-isomorphism follows from Corollary \ref{serre-condensed}. 

 The case of $X$ partially proper follows from the Stein case by a limit/colimit argument as in the proof of Theorem \ref{dualdR-rigid}.
\end{proof}

   The duality for $\B^+_{\dr}$-cohomology has a different form. For  $r\geq d$, a natural trace map in $\sd(\B^+_{\dr,\Box})$ can be defined by the composition
 $$
 {\rm tr}_{\B^+_{\dr}}:\quad F^{r}\R\Gamma_{\dr}(X_C/\B^+_{\dr})[2d]\to \R\Gamma_c(X,\Omega^d)\otimes_K^{\LL_{\Box}}F^{r-d}\B^+_{\dr}\verylomapr{{\rm tr}_{\rm coh}\otimes{\rm Id}} F^{r-d}\B^+_{\dr}.
 $$
\begin{corollary}{\rm ({\em Filtered $\B^+_{\dr}$-duality})}\label{bdrduality} Let  $X$ be a smooth partially proper  variety of dimension $d$ over $K$, $r,r^{\prime}\geq d$ and $s=r+r^{\prime}-d$. Then we have a natural quasi-isomorphism in $\sd(\B^+_{\dr,\Box})$: 
$$
F^r\R\Gamma_{\dr}(X_C/\B^+_{\dr})\simeq \R\Hom_{\B^+_{\dr,\Box}}(F^{r^{\prime}}\R\Gamma_{\dr,c}(X_C/\B^+_{\dr}),F^s\B^+_{\dr}).
$$
\end{corollary}
\begin{proof} Assume first that $X$ is Stein. 
This follows from the  computation:
\begin{align}\label{kolobrzeg2}
\R\Hom_{\B^+_{\dr,\Box}}(F^{r^{\prime}}& \R\Gamma_{\dr,c}(X_C/\B^+_{\dr}),F^s\B^+_{\dr})  \simeq \Hom_{\B^+_{\dr,\Box}}(F^{r^{\prime}}\R\Gamma_{\dr,c}(X_C/\B^+_{\dr}),F^s\B^+_{\dr})\\
& \simeq (H^d_c(X,\Omega^d)^*\otimes^{\Box}_K t^{r}\B^+_{\dr}\to H^d_c(X,\Omega^{d-1})^*\otimes^{\Box}_K t^{r-1}\B^+_{\dr}\to \cdots \to H^d_c(X,\so)^*\otimes^{\Box}_K t^{r-d}\B^+_{\dr})[-d],\notag\\
& \simeq \so(X)\otimes^{\Box}_K t^r\B^+_{\dr}\to\Omega^1(X)\otimes^{\Box}_K t^{r-1}\B^+_{\dr}\to\cdots\to \Omega^d(X)\otimes^{\Box}_K t^{r-d}\B^+_{\dr}\notag\\
&\simeq F^r\R\Gamma_{\dr}(X_C/\B^+_{\dr}).\notag
\end{align}
We used here the quasi-isomorphisms
\begin{align*}
\R\Hom_{\B^+_{\dr,\Box}}(H^d_c(X,\Omega^i)\otimes^{\LL_{\Box}}_K t^{r^{\prime}-i}\B^+_{\dr},F^s\B^+_{\dr}) \simeq H^d_c(X,\Omega^i)^*\otimes^{\Box}_K t^{r-i}\B^+_{\dr},
\end{align*}
which hold because we have:
\begin{align*}
\R\Hom_{\B^+_{\dr,\Box}}(H^d_c(X,\Omega^i)\otimes^{\LL_{\Box}}_K t^{r^{\prime}-i}\B^+_{\dr},t^s\B^+_{\dr}) &  \simeq \R\Hom_{K_{\Box}}(H^d_c(X,\Omega^i),t^{s-r^{\prime}+i}\B^+_{\dr})\\
  & \simeq \R\Hom_{K_{\Box}}(H^d_c(X,\Omega^i),K)\otimes_K^{\LL_{\Box}}F^{s-r^{\prime}+i}\B^+_{\dr}\\
  & \simeq H^d_c(X,\Omega^i)^*\otimes^{\Box}_K t^{r-d+i}\B^+_{\dr}.
\end{align*}
Here, the first quasi-isomorphism is by adjunction;  the second quasi-isomorphism follows from the fact that $H^d_c(X,\Omega^i)$ is of compact type, $\B^+_{\dr}$ is Fr\'echet, and we have \cite[Th. 3.40]{RR}.

  The case of $X$ partially proper follows from the Stein case by a limit/colimit argument as in the proof of Theorem \ref{dualdR-rigid}.
\end{proof}

\subsubsection{Overconvergent setting}

By the same arguments as in the proof of Theorem \ref{dualdR-rigid} we deduce the Poincar\'e duality for de Rham cohomology: 
\begin{theorem}{\rm(de Rham duality)}
 \label{dualdR-ov} 
Let $L=K,C$ and $X$ be a smooth  dagger variety of dimension $d$ over $L$. Then:  
\begin{enumerate}
\item
There are natural trace maps in $\sd(L_{\Box})$ and $L_{\Box}$, respectively:
$${\rm tr}_{\dr}: \, \R\Gamma_{\dr,c}(X)[2d]\to L,\quad {\rm tr}_{\dr}: \, H^{2d}_{\dr,c}(X)\to L. 
$$
\item If $X$ is quasi-compact, the  pairing
$$
\R\Gamma_{\dr}(X)\otimes^{\LL_{\Box}}_L \R\Gamma_{\dr,c}(X)[2d]\to \R\Gamma_{\dr,c}(X)[2d]\xrightarrow{{\rm tr}_{\dr}} L
$$
is a perfect duality, i.e., we have an  induced quasi-isomorphism in $\sd(L_{\Box})$
\begin{align*}
& \R\Gamma_{\dr}(X)\stackrel{\sim}{\to} \R\Hom_{L_{\Box}} (\R\Gamma_{\dr,c}(X)[2d],L).
\end{align*}
\item If $X$ is quasi-compact, 
the  pairing
$$
H^i_{\dr}(X)\otimes^{\LL_{\Box}}_L  H^{2d-i}_{\dr,c}(X)\to H^{2d}_{\dr,c}(X)\xrightarrow{{\rm tr}_{\dr}} L
$$
is a perfect duality of solid $L$-vector spaces, i.e., we have an induced quasi-isomorphism in $\sd(L_{\Box})$
\begin{align*}
& H^i_{\dr}(X)\stackrel{\sim}{\to} \Hom_{L_{\Box}} (H^{2d-i}_{\dr,c}(X),L),\\
& H^i_{\dr,c}(X)\stackrel{\sim}{\to} \Hom_{L_{\Box}} (H^{2d-i}_{\dr}(X),L).
\end{align*}
\end{enumerate}
\end{theorem}
\begin{proof}
In the case $X$ is affinoid,  
the trace map ${\rm tr}_{\dr}: H^{2d}_{\dr,c}(X)\to L$ can be  induced by the coherent trace map ${\rm tr}_{\rm coh}: H^d_c(X, \Omega_X^d) \to L$ from 
Corollary \ref{serre-condensed-over}, noting that ${\rm tr}_{\rm coh}(\partial)=0$ where $\partial: H^{d}_c(X, \Omega_X^{d-1}) \to H^d_c(X, \Omega_X^d)$ (reduce to the Stein case and use \cite[Lemma 2.1.4]{Bey97}). Its derived version ${\rm tr}_{\dr}: \R\Gamma_{\dr,c}(X)[2d]\to L$ is defined as the composition
$$
{\rm tr}_{\dr}:\quad \R\Gamma_{\dr,c}(X)[2d]\stackrel{\sim}{\leftarrow}\tau_{\leq 2d}\R\Gamma_{\dr,c}(X)[2d]\to H^{2d}_{\dr,c}(X)\lomapr{{\rm tr}_{\rm dR}}L.
$$
In the case of general smooth  $X$ we reduce to the affinoid case as in the proof of Corollary~\ref{serre-condensed-over}.

 In the case $X$ is affinoid,  claim (2) follows immediately from the sequence representing $\R\Gamma_{\dr}(X)$:
 $$
 \so(X)\to \Omega^1(X)\to\cdots\to\Omega^d(X),
 $$
   the analog of the sequence \eqref{buis2} representing $\R\Gamma_{\dr,c}(X)$, and derived Serre duality \eqref{52-cafe}. For a general quasi-compact  $X$,  we cover it with 
   dagger affinoid  varieties $\{X_i\}, i\in I$, and consider the associated \v{C}ech hypercovering $Y_{\jcdot}$ of $X$.  Then we compute in $\sd(L_{\Box})$
   \begin{align*}
{\R\Gamma}_{\dr,c}(X)& \simeq\colim_{i} {\R\Gamma}_{\dr,c}(Y_{i})\simeq\colim_{i}\R\Hom_{L_{\Box}}({\R\Gamma}_{\dr}(Y_{i})[2d],L)\\
& \simeq\R\Hom_{L_{\Box}}(\R\lim_{i}{\R\Gamma}_{\dr}(Y_{i})[2d],L)\simeq\R\Hom_{L_{\Box}}({\R\Gamma}_{\dr}(X)[2d],L),
   \end{align*}
as wanted. Here, the second quasi-isomorphism follows from the duality for affinoid varieties and the third quasi-isomorphism from the fact that limits/colimits are finite.
   
    The first isomorphism in claim (3) follows from claim (2) by taking cohomology. We use here the fact that $\Ext^j_{L_{\Box}}(H^{m}_{\dr}(X),L)=0$, for $j\geq 1$, because $H^{m}_{\dr}(X)$ is of finite rank.  The second isomorphism follows from the computation: 
\begin{align*}
 \Hom_{L_{\Box}} (H^{2d-i}_{\dr,c}(X),L)\stackrel{\sim}{\leftarrow}  \Hom_{L_{\Box}} (\Hom_{L_{\Box}} (H^{i}_{\dr}(X),L),L)\stackrel{\sim}{\to} H^{i}_{\dr}(X),
\end{align*}
where the first isomorphism was just proven and the second one follows from (trivial) reflexivity of $H^{i}_{\dr}(X)$.
\end{proof}

  \subsection{Hyodo-Kato duality} \label{poinc-hk-rig}
Using the Hyodo-Kato isomorphisms  we can  deduce Poincar\'e duality for the Hyodo-Kato cohomology of a smooth partially proper rigid analytic variety or a smooth dagger affinoid over $L=K,C$ from the one for de Rham cohomology. 
\subsubsection{Rigid analytic setting} 
\begin{theorem}
\label{dualHK-rigid} 
Let $L=K,C$ and $X$ be a smooth partially proper variety of dimension $d$ over $L$. Set $F_L=F, F^{\rm nr}, \breve{F}$ and $*=-,\breve{F}$. Then: 
\begin{enumerate}
\item
There is a natural trace map in the category of solid  $(\phi,N)$-modules over $F_L$
$${\rm tr}_{\hk}: \, H^{2d}_{\hk,*,c}(X)\to F_L\{-d\},  
$$
which is compatible with the de Rham trace map via the Hyodo-Kato morphism and Galois equivariant in the case $L=C$ and $X$ is defined over $K$.
\item
The  pairing of solid $(\phi,N)$-modules over $F_L$
$$
H^i_{\hk,*}(X)\otimes^{\Box}_{F_L} H^{2d-i}_{\hk,*,c}(X)\to H^{2d}_{\hk,*,c}(X)\xrightarrow{{\rm tr}_{\hk}} F_L\{-d\}
$$
is a perfect duality, i.e., we have  induced isomorphisms of solid $(\phi,N)$-modules over $F_L$
\begin{align*}
& H^i_{\hk,*}(X)\stackrel{\sim}{\to} \Hom_{F_{L,\Box}} (H^{2d-i}_{\hk,*,c}(X),F_L\{-d\}),\\
& H^i_{\hk,*,c}(X)\stackrel{\sim}{\to} \Hom_{F_{L,\Box}} (H^{2d-i}_{\hk,*}(X),F_L\{-d\}).
\end{align*}
These isomorphisms are  Galois equivariant if $L=C$ and $X$ is defined over $K$.
\end{enumerate}
\end{theorem}

\begin{proof}
Assume first that we can define a trace map ${\rm tr}_{\hk} : H^{2d}_{\HK,*,c}(X) \to F_L\{-d\}$
such that the following diagram commutes:
\begin{equation}
\label{diagRond}
 \xymatrix{ 
H^i_{\HK,*}(X)  \ar[d]^{\iota_{\HK}} \otimes^{\Box}_{F_L}   H^{2d-i}_{\HK,*,c} (X) \ar[d]^{\iota_{\HK}}    \ar[r]  & H^{2d}_{\HK,*,c}(X) \ar[d]^{\iota_{\HK}} \ar[r]^-{{\rm tr}_{\hk}} & F_L\{-d\} \ar[d]  \\
H^i_{\dr}(X)  \otimes^{\Box}_L  H^{2d-i}_{\dr,c}(X)   \ar[r]  &H^{2d}_{\dr,c}(X) \ar[r]^-{{\rm tr}_{\dr}} & L}.
\end{equation}
Then, since the map from the bottom row is a perfect pairing, via the Hyodo-Kato isomorphisms,  the top row is a perfect pairing as well and we obtain the duality for the Hyodo-Kato cohomology.

  Hence it suffices to prove claim (1). We will show that the de Rham trace map factorises through $F_L$ when restricted to $H^{2d}_{\HK,c}(X)$ yielding the Hyodo-Kato trace map with the required properties. We may assume that $X$ is connected. We start with the case $L=C$ and $*=\breve{F}$. 
Consider the following commutative diagram
\begin{equation}\label{china1}
\xymatrix{
H^{2d}_{\HK,\breve{F},c}(X) \ar[r]^{\iota_{\HK}}& H^{2d}_{\dr,c}(X)\ar[r]^-{{\rm tr}_{\dr}}_-{\sim} & L\\
H^{2d}_{\HK,\breve{F},c}(U)  \ar[d]^{\wr}\ar[u]^{\wr} \ar[r]^{\iota_{\HK}}& H^{2d}_{\dr,c}(U)  \ar[d]^{\wr}\ar[u]^{\wr}\ar[r]^-{{\rm tr}_{\dr}} _-{\sim}& L\ar@{=}[u]\\
H^{2d}_{\HK,\breve{F},c}({\mathbb P}^d_L)  \ar[r]^{\iota_{\HK}}& H^{2d}_{\dr,c}({\mathbb P}^d_L) \ar[r]^-{{\rm tr}_{\dr}} _-{\sim}& L\ar@{=}[u]\\
H^{2d}_{\HK,\breve{F},c}({\mathbb P}^d_{\breve{F}}) \ar[u]^{\wr} \ar[r]^{\iota_{\HK}}_{\sim}& H^{2d}_{\dr,c}({\mathbb P}^d_{\breve{F}}) \ar[u]\ar[r]^-{{\rm tr}_{\dr}}_-{\sim} & \breve{F}\{-d\}\ar[u]
}
\end{equation}
Here $U\subset X$ is an open ball of dimension $d$ inside $X$. The bottom de Rham trace map is compatible with Frobenius as can be easily seen by passing to the affine space and using  formula \eqref{trace-aff}. 
The above diagram shows that the top horizontal composition factors through $\breve{F}$,  defining the trace map ${\rm tr}_{\hk}: H^{2d}_{\HK,\breve{F},c}(X)\to \breve{F}\{-d\}$, which is clearly compatible with the de Rham trace map and Frobenius. Moreover,  the monodromy on 
$H^{2d}_{\HK,\breve{F},c}(X)$ is trivial (since it is so on $H^{2d}_{\HK,\breve{F},c}({\mathbb P}^d_{\breve{F}})$) hence the trace map ${\rm tr}_{\hk}$ is compatible with monodromy. 
Same for the residual Galois action.

 The same argument works for $L=K$ (and $F_L=F$). For $L=C$ and $F_L=F^{\nr}$ (and $*=-$) we need to modify it slightly.  We replace the left column in diagram \ref{china1} with $F^{\nr}$-cohomology and then add the following commutative diagram to the bottom
 $$
\xymatrix{
H^{2d}_{\HK,{F}^{\nr},c}({\mathbb P}^d_L)  \ar[r]^{\iota_{\HK}}& H^{2d}_{\dr,c}({\mathbb P}^d_L) \ar[r]^-{{\rm tr}_{\dr}} _-{\sim}& L\\
H^{2d}_{\HK,{F}^{\nr},c}({\mathbb P}^d_{\Q_p}) \ar[u] \ar[r]^{\iota_{\HK}}_{\sim}& H^{2d}_{\dr,c}({\mathbb P}^d_{\Q_p}) \ar[u]\ar[r]^-{{\rm tr}_{\dr}}_-{\sim} & \Q_p\{-d\}\ar[u]}
$$
Now we   $F^{\nr}$-linearize  the bottom row. This induces an isomorphism in the left column and our claim follows. 
 \end{proof}
\begin{lemma}\label{HK-pp}
Let $j:X\to Y$ be an open immersion of smooth partially proper rigid analytic varieties of dimension $d$ over $L$. The Hyodo-Kato  trace map  is compatible with $j_*$, i.e., 
the following diagram is commutative: 
$$
 \xymatrix{ H^{2d}_{\hk,*,c}(X) \ar[r]^{j_*} \ar[d]_{{\rm tr}_{\hk,X}} & H^{2d}_{\hk,*,c}(Y) \ar[dl]^{{\rm tr}_{\hk,Y}} \\
  F_L\{-d\}. }
$$
\end{lemma}
\begin{proof} We will argue for $F^{\nr}$-cohomology; the argument for other $F_L$-cohomologies   is analogous.
  Consider the following diagram: 
\[ 
\xymatrix{ 
H^{2d}_{{\rm dR},c}(X) \ar[ddd]_{{\rm tr}_{X}} \ar[rrr]& & & H^{2d}_{{\rm dR},c}(Y) \ar[ddd]^{{\rm tr}_{Y}} \\
& H^{2d}_{\hk,c}(X) \ar[r] \ar[d]_{{\rm tr}_{X}} \ar[ul] &  H^{2d}_{\hk,c}(Y)  \ar[d]^{{\rm tr}_{Y}} \ar[ru]& \\
& F^{\rm nr} \ar@{=}[r] \ar@{^(->}[dl] & F^{\rm nr} \ar@{^(->}[dr] & \\
C \ar@{=}[rrr] & & & C} \]
We want  show that  the central square commutes. For that, since the arrows from $F^{\rm nr}$ to $C$ are injective, it suffices to check the commutativity of the four trapezoids and the outer square. Commutativity  is clear for the lower trapezoid. For the left and right trapezoids, it comes from the construction of the trace maps. For the upper one, it follows from the compatibility of the Hyodo-Kato morphism with the pushforward. It remains to show that the trace map for the de Rham cohomology is compatible with pushforward. But this is Lemma \ref{dR-pp}. 
\end{proof}
\subsubsection{Overconvergent setting} 
\begin{theorem}
\label{dualHK-ov} 
Let $L=K,C$ and $X$ be a smooth dagger variety of dimension $d$ over $L$. Set $F_L=F, F^{\rm nr}, \breve{F}$ and $*=-,\breve{F}$. Then: 
\begin{enumerate}
\item
There is a natural trace map in the category of solid  $(\phi,N)$-modules over $F_L$
$${\rm tr}_{\hk}: \, H^{2d}_{\hk,*,c}(X)\to F_L\{-d\},
$$
which is compatible with the de Rham trace map via the Hyodo-Kato morphism and Galois equivariant in the case $L=C$ and $X$ is defined over $K$.
\item If $X$ is quasi-compact, 
the  pairing of solid $(\phi,N)$-modules over $F_L$
$$
H^i_{\hk,*}(X)\otimes_{F_L}^{\Box} H^{2d-i}_{\hk,*,c}(X)\to H^{2d}_{\hk,*,c}(X)\xrightarrow{{\rm tr}_{\hk}} F_L\{-d\}
$$
is a perfect duality of solid $F_L$-vector spaces, i.e., we have induced isomorphisms of solid $(\phi,N)$-modules over $F_L$
\begin{align*}
& H^i_{\hk,*}(X)\stackrel{\sim}{\to} \Hom_{F_L,\Box} (H^{2d-i}_{\hk,*,c}(X),L\{-d\}),\\
& H^i_{\hk,*,c}(X)\stackrel{\sim}{\to} \Hom_{F_L,\Box} (H^{2d-i}_{\hk,*}(X),L\{-d\}).
\end{align*}
These isomorphisms are  Galois equivariant if $L=C$ and $X$ is defined over $K$.
\end{enumerate}
\end{theorem}
\begin{proof}  For the trace map in the affinoid case, we use dagger presentations and the trace map for Stein varieties. 
We have an analog of Lemma \ref{HK-pp} in this setting, i.e., for an open immersion of smooth dagger affinoids of dimension $d$ over $K$. This can be easily seen from the definition of the Hyodo-Kato trace  via a reduction to Lemma \ref{HK-pp}. This yields, as in the proof of Corollary \ref{serre-condensed-over}, the trace map for general smooth dagger varieties.

  For claim (2), note
first 
that the following diagram commutes:
$$
 \xymatrix{ 
H^i_{\HK,*}(X)  \ar[d]^{\iota_{\HK}} \otimes^{\Box}_{F_L}   H^{2d-i}_{\HK,*,c} (X) \ar[d]^{\iota_{\HK}}    \ar[r]  & H^{2d}_{\HK,*,c}(X) \ar[d]^{\iota_{\HK}} \ar[r]^-{{\rm tr}_{\hk}} & F_L\{-d\} \ar[d]  \\
H^i_{\dr}(X)  \otimes^{\Box}_L  H^{2d-i}_{\dr,c}(X)   \ar[r]  &H^{2d}_{\dr,c}(X) \ar[r]^-{{\rm tr}_{\dr}} & L}.
$$
That is, that the Hyodo-Kato and de Rham trace maps are compatible: which follows from the fact the de Rham trace map can also be defined using dagger presentations and Stein trace maps.
Then, since the map from the bottom row is a perfect pairing by Theorem \ref{dualdR-ov}, via the Hyodo-Kato isomorphisms,  the top row is a perfect pairing as well and we obtain the duality for the Hyodo-Kato cohomology. 
 \end{proof}

\subsection{Examples}

We will compute the de Rham and Hyodo-Kato cohomologies of the affine space and the torus and check they satisfy the previous Poincar\'e duality results.

\subsubsection{Affine space}\label{affine}
Let  $X={\mathbb A}_L^d$ be the affine dagger space over $L=K,C$. Recall that we have  
$$
 H^i_{\dr}(X)\simeq  \begin{cases} L \text{ for } i = 0,\\ 0 \text{ for } i\neq 0, \end{cases} \text{ and  }  H^i_{\hk}(X)\simeq \begin{cases} F_L, \text{ pour } i = 0,\\ 0, \text{ for  } i\neq 0,\end{cases}
$$
where $F_C=F^{\nr}$ and $F_K=F$. 

It follows that, by Poincar\'e duality, we have: 
\begin{equation}
\label{mdm2}
 H^i_{\dr,c}(X)\simeq  \begin{cases} 0 \text{ for } i \neq 2d, \\ L \text{ for } i=2d, \end{cases} \text{ and }  H^i_{\hk,c}(X)\simeq  \begin{cases} 0 \text{ for } i \neq 2d, \\ F_L\{ -d \} \text{ for } i=2d. \end{cases} 
\end{equation}
The monodromy on $ H^i_{\hk,c}(X)$ is trivial and so is the action of the Galois group $\Gal(\overline{K}/F^{\rm nr})$ (in the case $L=C$).

  We can however obtain this result more directly. Recall that \cite[proof of Th. 4.17]{CDN3}, there exists  a semistable weak formal scheme $\sa^d$ over $\so_K$ such that $\sa^d_K\simeq {\mathbb A}^d_K$. Moreover
 there exist closed (resp. open) subschemes $Y_s, s\in\N,$ (resp. $U_s, s\in \N$) of the special fiber $Y:=\sa^d_0$ such that 
\begin{enumerate}
\item each $Y_s$ is a finite union of irreducible components,
\item $Y_s\subset U_s\subset Y_{s+1}$ and their union is $Y$,
\item the tubes  $\{\overline{U}_s:=]U_{s}[_{X}\},s\in\N,$ form  a Stein covering of $\sa^d_K$ by closed discs.
\end{enumerate}
We will call the covering $\{U_s\},s\in\N,$ a Stein covering of $Y$.
The schemes $U_s, Y_s$ inherit their log-structure from $Y$ (which is canonically a log-scheme log-smooth over $k^{0}$). The log-schemes $Y_s$ are not log-smooth (over $k^0$) but they are ideally log-smooth, i.e., they have a canonical idealized log-scheme structure and are ideally log-smooth for this structure.

\vskip.2cm
$\bullet$ {\em Computation of the de Rham cohomology with compact support.}
We have the distinguished triangle 
  \begin{align}
  \label{pijany1}
 &  \R\Gamma_{\dr,c}(X)\to  \R\Gamma_{\dr}(X)\to  \R\Gamma_{\dr}(\partial X), \mbox{where }\\
  & \R\Gamma_{\dr}(\partial X)\simeq \colim_s\R\Gamma_{\dr}(X\setminus \overline{U}_s).\notag
  \end{align}
 Since $X\setminus \overline{U}_s$ is a punctured open disc, we have 
 $$
H^i_{\dr}(\partial X)\simeq \begin{cases} 
K & \mbox{ for } i=0, \\
K & \mbox{ for } i =2d-1, \\
0 & \mbox{ for } i\neq0, 2d-1.\\
\end{cases}
 $$
 The computation in (\ref{mdm2}) now easily follows from that fact and  the distinguished triangle (\ref{pijany1}).

 \vskip.2cm
$\bullet$ {\em Computation of the Hyodo-Kato cohomology with compact support.}
Similarly, we can use the distinguished triangle (everything is over $F$ now)
 \begin{align}
  \label{pijany2}
 &  \R\Gamma_{\hk,c}(X)\to  \R\Gamma_{\hk}(X)\to  \R\Gamma_{\hk}(\partial X), \mbox{where }\\
  & \R\Gamma_{\hk}(\partial X)\simeq \colim_s\R\Gamma_{\hk}(X\setminus U^0_s),\notag
\end{align}
where $U^0_s$ is the naive  interior of  $\overline{U}_s$, an open disc. Since $X\setminus \overline{U}^0_s$ has a semistable model over $F$, we have: 
\begin{enumerate}
\item The Hyodo-Kato morphism induces an isomorphism $\iota_{\hk}: H^i_{\hk}(\partial X) \wh{\otimes}_F K \stackrel{\sim}{\to}H^i_{\dr}(\partial X)$, for all $i \ge 0$, and so we obtain 
 $$
H^i_{\hk}(\partial X)\simeq \begin{cases} 
F & \mbox{ for } i=0, \\
F & \mbox{ for } i =2d-1, \\
0 & \mbox{ for } i\neq0, 2d-1.\\
\end{cases}
 $$
\item Frobenius:  we have the isomorphism $H^{2d}_{\hk,c}(X)\stackrel{\sim}{\to} H^{2d}_{\hk,c}({\mathbb P}_K^d)$ and the Frobenius on the target group has slope $d$.
\item The monodromy is trivial on $H^i_{\hk}(X\setminus U^0_s)$ (since the integral model has no loops) hence on $H^i_{\hk}(\partial X)$ as well.
\end{enumerate}
We then recover the results from (\ref{mdm2}) using (\ref{pijany2}).

\subsubsection{Torus} \label{torus}

Let  $X:=\G_{m,K}^d$ be a rigid analytic torus of dimension $d$. Recall  that  $H^i_{\dr}(\G_{m,K}^d)$ is a  $K$-vector space of dimension  $\binom{d}{i}$, generated by symbols  $c_1^{\dr}(z_{j_1}) \dots c_1^{\dr}(z_{j_i})$ for $j_1 < j_2 < \dots < j_i$, where  $c_1^{\dr}(z)= \frac{dz}{z}$ is the de Rham Chern class. Similarly, $H^i_{\HK}(\G_{m,K}^d)$ is a  $F$-vector space of dimension  $\binom{d}{i}$, generated by symbols $c_1^{\HK}(z_{j_1}) \dots c_1^{\HK}(z_{j_i})$. The Chern classes $c_1^{\dr}(z)$ and  $c_1^{\HK}(z)$ are compatible with the  Hyodo-Kato isomorphism. Moreover $\varphi(c_1^{\HK}(z))=p c_1^{\HK}(z)$ and  $N(c_1^{\HK}(z))=0$.
 
 Using Poincar\'e duality, we obtain that $H^i_{\dr,c}(\G_{m,K}^d)$ is a  $K$-vector space of dimension  $\binom{d}{2d-i}$, generated by classes  $(c_1^{\dr}(z_{j_1}) \dots c_1^{\dr}(z_{j_{2d-i}}))^*$ for $j_1 < j_2 < \dots < j_{2d-i}$.
  Similarly, $H^i_{\HK,c}(\G_{m,K}^d)$ is a  $F$-vector space of dimension  $\binom{d}{2d-i}$, generated by classes $(c_1^{\HK}(z_{j_1}) \dots c_1^{\HK}(z_{j_{2d-i}}))^*$.  Monodromy acts trivially on these classes and Frobenius by multiplication by $p^{i-d}$.

\section{Geometric syntomic cohomology with compact support}

In this section, we define geometric compactly supported syntomic cohomology for both smooth rigid analytic varieties and smooth dagger varieties. In each case, the previous recipes apply and similar properties are obtained. We then show that compactly supported syntomic cohomology can be computed in terms of compactly supported (twisted) Hyodo-Kato cohomology and (quotients of) compactly supported $\B_{\dr}^+$-cohomology. In particular, this will allow us to compare the rigid definition with the dagger one for partially proper varieties. 

\subsection{Rigid analytic setting}

Recall (see \cite[Section 4.1]{CN3}) that geometric syntomic cohomology is defined by $\eta$-\'etale descent from the crystalline syntomic cohomology of Fontaine-Messing $\rg_{\synt}(\XS, \Q_p(r)):=[F^r\rg_{\crr}(\XS) \xrightarrow{p^r- \varphi} \rg_{\crr}(\XS)]$. By regarding the crystalline sheaves as condensed abelian groups, we can define the sheaf $\AS_{\synt}(r)$ on ${\rm Sm}_{C,\eet}$ as a sheaf with values in ${\cal D}({\Q_{p,\Box}})$ and set $\rg_{\synt}(X,\Q_p(r)):= \rg_{\eet}(X, \AS_{\synt}(r))$ for $X \in {\rm Sm}_{\eet, C}$. Following the usual process, we define the compactly supported cohomology $\rg_{\synt,c}(X,\Q_p(r)) \in {\cal D}(\Q_{p,\Box})$ of a smooth variety $X$ over $C$. It satisfies all the properties from Section~\ref{compact1}.  

Recall that we have:
\begin{theorem}\cite[Corollary 5.5]{CN4}
\label{syn-HK-dR}
Let $X \in {\rm Sm}_{C}$ and $r \ge 0$. There is a natural quasi-isomorphism in ${\cal D}({\Q_{p,\Box}})$: 
\[ \rg_{\synt}(X,\Q_p(r)) \xrightarrow{\sim} [ [\rg_{\hk}(X) \dsolid_{F^{\rm nr}} \Bstp]^{N=0,\varphi=p^r} \xrightarrow{\iota_{\hk} \otimes \iota} (\rg_{\dr}(X/ \B_{\dr}^+)/F^r], \]
where 
\[ [\rg_{\hk}(X) \dsolid_{F^{\rm nr}} \Bstp]^{N=0,\varphi=p^r}:=  \left[\vcenter{\xymatrix @C=2cm{ 
\rg_{\hk}(X) \dsolid_{F^{\rm nr}}\Bstp \ar[r]^-{(1- \varphi/p^r)} \ar[d]^{N} & \rg_{\hk}(X) \dsolid_{F^{\rm nr}}\Bstp \ar[d]^N  \\
\rg_{\hk}(X) \dsolid_{F^{\rm nr}} \Bstp \ar[r]^{(1- \varphi/p^{r-1})} & \rg_{\hk}(X) \dsolid_{F^{\rm nr}} \Bstp }} \right]. \]  
\end{theorem} 

\begin{proof}
The proof is given in \cite[Section 5.1]{CN4}. Let us recall how it goes to see that  it still works in the condensed setting: We work locally. The quasi-isomorphism for the local statement is induced by the quasi-isomorphisms $\iota_{\rm BK}^1$ and $\iota_{\rm BK}^2$ given by the following commutative diagram: 
  \begin{equation*}
\xymatrix{
[\R\Gamma_{\hk}({X})\dsolid_{F^{\nr}}{\B}^+_{\st}]^{N=0}\ar[d]^{\epsilon^{\hk}_{\st}}_{\wr} \ar@/_80pt/[dd]^{\wr}_{\iota_{\rm BK}^1}
\ar[r]^-{\iota_{\hk}\otimes \iota}   &   
   \rg_{\dr}({X}/\B^+_{\dr})/F^r
 \ar@/^80pt/[dd]^{\iota_{\rm BK}^2}_{\wr}\\
     [\R\Gamma_{\crr}({X})\dsolid_{\B^+_{\crr}}{\B}^+_{\st}]^{N=0}\\
         \R\Gamma_{\crr}({X})\ar[u]^{\wr}\ar[r]^{\can}  &        \R\Gamma_{\crr}({X})/F^r\ar[uu]_{\rotatebox{90}{$\sim$}}^{\iota_{{\rm BK},r}}
}
\end{equation*} 
where the quasi-isomorphism $\epsilon^{\hk}_{\st}$ is the quasi-isomorphism from the proof of Proposition~\ref{kit-2} and $\iota_{{\rm BK},r}$ is the one from \eqref{Boa}.  
\end{proof}

We will prove below in Proposition~\ref{ciud-3} that we have an analogous result for the compactly supported cohomology. Let us first define and prove some properties satisfied by the twisted Hyodo-Kato cohomology and the quotients of $\B_{\dr}^+$-cohomology.   We take $X\in {\rm Sm}_C$. 

\vskip.2cm
(i) {\em Twisted Hyodo-Kato cohomology.}
We set  
\begin{align*} 
\HK_*(X,r) := [\R\Gamma_{\HK,*}(X) \dsolid_{F^{\rm nr}} \Bstp]^{N=0,\varphi=p^r}.
\end{align*}

\begin{lemma}
\label{ciud-1}
There is a natural quasi-isomorphism in $\sd(\Q_{p,\Box})$:
\[ \HK_c(X,r) \stackrel{\sim}{\to} [ \HK(X,r) \to \HK(\partial X, r)] \]
where $\HK(\partial X,r):= \colim_{Z\in\Phi} \HK(X \setminus Z,r)$. 
\end{lemma}

\begin{proof}
We need to show that we have a quasi-isomorphism
\[ [(\colim_{Z\in\Phi} \rg_{\hk}(X\setminus Z)) \dsolid_{F^{\rm nr}} \Bstp]^{N=0, \varphi=p^r} \stackrel{\sim}{\to}  \colim_{Z\in\Phi} ([\rg_{\hk}(X\setminus Z) \dsolid_{F^{\rm nr}} \Bstp]^{N=0, \varphi=p^r}).\] 
But this follows from the fact that taking the colimit commutes with $\dsolid_{F^{\rm nr}}$.
\end{proof}

\vskip.2cm
(ii) {\em $\B^+_{\dr}$-cohomology and its quotients.}
Recall that we set 
\begin{align*} 
\DR_*(X,r):= \R\Gamma_{\dr,*}(X/\B^+_{\dr})/F^r .
\end{align*}
Arguing as in the proof of Lemma \ref{ciud-1} we get:
\begin{lemma}
\label{ciud-2}
There is a natural quasi-isomorphism in $\sd(\B^+_{\dr,\Box})$:
\[ \DR_c(X,r) \stackrel{\sim}{\to} [ \DR(X,r) \to \DR(\partial X, r)] \]
where $\DR(\partial X,r):= \colim_{Z\in\Phi} \DR(X \setminus Z,r)$. 
\end{lemma}

\vskip.2cm
(iii) {\em Comparison with  syntomic cohomology.} As in the classical case, we can compute compactly supported syntomic cohomology in terms of compactly supported Hyodo-Kato and $\B_{\dr}^+$- cohomologies.  

\begin{proposition}
\label{ciud-3}
There is a natural quasi-isomorphism in $\sd(\Q_{p,\Box})$: 
\[ \rg_{\synt,c}(X,\Q_p(r)) \xrightarrow{\sim} [ \HK_c(X,r) \to \DR_c(X,r) ] \]
such that for all quasi-compact open $Z \subset X$, the following diagram commutes in $\sd(\Q_{p,\Box})$:
\[ \xymatrix{ 
\rg_{\synt,c}(X,\Q_p(r)) \ar[r] &  \HK_c(X,r) \ar[r] & \DR_c(X,r) \\
\rg_{\synt,Z}(X,\Q_p(r)) \ar[r]\ar[u] & \HK_Z(X,r) \ar[r] \ar[u] & \DR_Z(X,r). \ar[u]
}\] 
\end{proposition}

\begin{proof} 
We deduce from Lemma~\ref{ciud-1} and Lemma~\ref{ciud-2} that we have a commutative diagram in which rows are distinguished triangles:  
\[ \xymatrix{
\rg_{\synt,c}(X,\Q_p(r)) \ar[r] \ar[d] & \rg_{\synt}(X,\Q_p(r)) \ar[r] \ar[d] & \rg_{\synt}(\partial X,\Q_p(r)) \ar[d]  \\
\HK_c(X,r) \ar[r] \ar[d] & \HK(X,r) \ar[r] \ar[d] & \HK(\partial X,r) \ar[d] \\
\DR_c(X,r) \ar[r] & \DR(X,r) \ar[r] &\DR(\partial X,r) } \]
The second column is a distinguished triangle by Theorem~\ref{syn-HK-dR}, so it suffices to check that the third column is a distinguished triangle as well. But for $Z \subset X$ a quasi-compact open, we have: 
\[\rg_{\synt}(X \setminus Z, \Q_p(r)) \stackrel{\sim}{\to} [ \HK(X \setminus Z,r) \to \DR(X\setminus Z,r) ] \]
and we deduce the result by taking the colimit over all $Z\in\Phi$.
\end{proof}

\subsection{Overconvergent setting} The definition of the compactly supported syntomic cohomology of a dagger variety is the same as the one of the pro-\'etale cohomology: we first give a local definition using presentation of dagger affinoids, and then (co)sheafify it.  

\subsubsection{Definition}

 If $X$ is a dagger affinoid over $C$ with a presentation $\{X_h\}_{h \in \N}$ we define in $\sd(\Q_{p,\Box})$
  \begin{align*}
 \rg^{\natural}_{\synt}(\partial X, \Q_p(r)) & := \rg_{\synt}(\partial \{X\}_h,\Q_p(r)):=\colim_h\rg_{\synt}(X_h\setminus \wh{X},\Q_p(r)),\\
\rg^{\natural}_{\synt,c}(X,\Q_p(r)) & := \rg_{\synt,c}(\{X\}_h,\Q_p(r)):=[\rg^{\natural}_{\synt}(X,r)\to  \rg^{\natural}_{\synt}(\partial X,\Q_p(r)) ], \quad \text{ for $r \in \N$}.
  \end{align*}
For a general smooth dagger variety $X$ over $L$, using the natural equivalence
  \begin{equation*}
  {\rm CSh}({\rm SmAff}^{\dagger}_{L,{\rm an}})\stackrel{\sim}{\to} {\rm CSh}({\rm Sm}^{\dagger}_{L,{\rm an}})\
  \end{equation*}
  we define the cosheaf $\sa_{\synt,c}(r)$ on $X_{\an}$ as the cosheaf associated to the precosheaf defined by $U\mapsto \rg^{\natural}_{\synt,c}(U,\Q_p(r)), U\in  {\rm SmAff}^{\dagger}_{L}$, $U\to X$  an admissible open. We set  $$
  \rg_{\synt,c}(X,\Q_p(r)):=\rg(X,\sa_{\synt,c}(r)) \in {\cal D}({\Q_{p,\Box}}).
  $$
 \begin{lemma}\label{herbata2-synt} Let $X$ be in ${\rm Sm}_C^{\dagger}$. 
  \begin{enumerate} 
  \item  {\rm (Local-global compatibility)} If $X$ is  a  dagger affinoid then the natural map in $\sd(\Q_{p,\Box})$
  $$ \rg_{\synt,c}(X,\Q_p(r))\to \rg^{\natural}_{\synt,c}(X,\Q_p(r))$$
   is a quasi-isomorphism. 
  \item {\rm (Passing to completion)} If $X$ is partially proper then there is  natural quasi-isomorphism  in $\sd(\Q_{p,\Box})$
  $$
  \rg_{\synt,c}({X},\Q_p(r))\stackrel{\sim}{\to} \rg_{\synt,c}(\wh{X},\Q_p(r)).
  $$
  \end{enumerate}
  \end{lemma}
  
  \begin{proof}
   For the first claim, 
  it suffices to show that for any open covering $X=X_1\cup X_2$ by dagger affinoids, we have the Mayer-Vietoris distinguished triangle
  $$
  \rg^{\natural}_{\synt,c}(X,\Q_p(r))\leftarrow \rg^{\natural}_{\synt,c}(X_1,\Q_p(r))\oplus \rg^{\natural}_{\synt,c}(X,\Q_p(r))\leftarrow \rg^{\natural}_{\synt,c}(X_1\cap X_2,\Q_p(r)).
  $$
  Choosing a presentation $\{ X_h \}$ of $X$, an $h\in \N$, and using excision we can write this triangle as
  $$
   \rg_{\synt,\wh{X}}(X_h,\Q_p(r))\leftarrow \rg_{\synt,\wh{X}_1}(X_h,\Q_p(r))\oplus \rg_{\synt,\wh{X}_2}(X_h,\Q_p(r))\leftarrow \rg_{\synt\wh{X_1\cap X_2}}(X_h,\Q_p(r)).
  $$
  That this triangle is distinguished can be seen by passing to twisted Hyodo-Kato and $\B^+_{\dr}$-cohomologies.

  Let us now prove claim (2). 
 If $X$ is Stein, take $\{ U_n \}$ an increasing covering of $X$ by admissible dagger affinoids. We have a commutative diagram:
\[ \xymatrix{ \rg_{\synt,c}(\wh{X},\Q_p(r)) \ar[r]  & \rg_{\synt}(\wh{X},\Q_p(r)) \ar[r]  & \colim_{n}  \rg_{\synt}(\wh{X}\setminus \wh{U}_n,\Q_p(r))  \\
\rg_{\synt,c}(X,\Q_p(r)) \ar[r] \ar@{-->}[u] & \rg_{\synt}(X,\Q_p(r)) \ar[r]  \ar[u]_{\rotatebox{90}{$\sim$}}& \colim_{n}  \rg_{\synt}(X\setminus {U}_n,\Q_p(r)) \ar[u]_{\rotatebox{90}{$\sim$}} } \] 
where the second row is the distinguished triangle     \eqref{tri-overcv-synt}. (We note that there is no circular reasoning here: Remark \ref{finish1} uses only the first claim of this lemma.) The second and third vertical arrows are quasi-isomorphisms by the usual comparison results between overconvergent cohomologies and their rigid analytic analogues (see \cite{CN3}).  We obtain that the first vertical map (defined by the diagram) is a quasi-isomorphism as well. 

   Since a general partially proper variety $X$ is locally Stein we can conclude by analytic codescent.
  \end{proof}
\begin{remark}\label{finish1}As for pro-\'etale cohomology, if $X$ is smooth and Stein and $\{U_n\}$ an increasing quasi-compact covering of $X$ with $U_n \Subset U_{n+1}$, we define the cohomology of the boundary of $X$ as 
\[ \rg_{\synt}(\partial X,\Q_p(r)):= \colim_{U_n} \rg_{\synt}(X \setminus U_n,\Q_p(r)). \]
The same proof as for \eqref{tri-overcv} shows that we have a distinguished triangle in ${\cal D}({\Q_{p,\Box}})$:
\begin{equation}
\label{tri-overcv-synt} 
\rg_{\synt,c}(X,\Q_p(r)) \to \rg_{\synt}(X,\Q_p(r)) \to \rg_{\synt}(\partial X,\Q_p(r)). 
\end{equation}  
\end{remark}
 \subsubsection{Passage to the Bloch-Kato version}
As in the rigid analytic, we can express compactly supported syntomic cohomology in terms of compactly supported de Rham and Hyodo-Kato cohomologies. We define the twisted Hyodo-Kato cohomology and the quotients of the $\B_{{\rm dR}}^+$-cohomology as in the rigid case:
 \begin{align*} 
& \HK_c(X,r) := [\R\Gamma_{\HK,c}(X) \dsolid_{F^{\rm nr}} \Bstp]^{N=0,\varphi=p^r},\\
& \DR_c(X,r):= \R\Gamma_{\dr,c}(X/\B^+_{\dr})/F^r.
\end{align*} 
The first one lives in $\sd(\Q_{p,\Box})$, the second one in $\sd(\B^+_{\dr,\Box})$.

\begin{proposition}
\label{coho-hk11}
For all $i\ge 0$ and $r\ge 0$, we have in ${\Q_{p,\Box}}$
\begin{align*}
H^i \HK_*(X,r) \simeq (H^i_{\hk,*}(X) \solid_{F^{\rm nr}} \Bstp)^{N=0, \varphi=p^r}.
\end{align*}
\end{proposition}

\begin{proof}
The result for the compactly supported cohomology follows from the result in the classical case working with dagger affinoids and using the long exact sequence defining the cohomology with compact support. So it suffices to prove the result for $\HK(X,r)$. An analogous result in the classical setting was proved in \cite[Example 3.16]{CDN3} (and extended in \cite[Prop. 4.24]{CN4}). The proof in the condensed setting is similar,  the main difference is that $\B_{\st}^+$ is already flat for the solid tensor product $-\solid_{F^{\rm nr}} -$.

  We first prove that we have 
\begin{equation}
\label{TeMaw0}
H^i (\rg_{\hk}(X) \dsolid_{F^{\rm nr}} \Bstp) \simeq H^i_{\hk}(X) \solid_{F^{\rm nr}} \Bstp, \quad \text{for all } i \ge 0. 
\end{equation}
Since $\Bstp$ is flat over $F$, we have a natural quasi-isomorphism in $\sd(\Q_{p,\Box})$
\[ H^0_{\hk}(X) \dsolid_{F^{\rm nr}} \Bstp \simeq H^0_{\hk}(X) \solid_{F^{\rm nr}} \Bstp. \]
This yields the distinguished triangle: 
\[ (H^0_{\hk}(X) \solid_{F^{\rm nr}} \Bstp)[0] \to \rg_{\hk}(X) \dsolid_{F^{\rm nr}} \Bstp \to (\tau_{\le 1} \rg_{\hk}(X)) \dsolid_{F^{\rm nr}} \Bstp\]
and, by taking the cohomology, we get
\begin{align*}
& H^0_{\hk}(X) \solid_{F^{\rm nr}} \Bstp \simeq H^0(\rg_{\hk}(X) \dsolid_{F^{\rm nr}} \Bstp), \\
& H^i(\rg_{\hk}(X) \dsolid_{F^{\rm nr}} \Bstp) \simeq H^i((\tau_{\le 1} \rg_{\hk}(X)) \dsolid_{F^{\rm nr}} \Bstp), \quad \text{for all } i \ge 1.
\end{align*} 
We repeat the same computation to obtain the isomorphism~\eqref{TeMaw0} for all $i$. 

We now assume that $X:= \XS_C$ for $\XS$ in the category of basic semistable formal models $\MS_C^{\dagger, \rm ss,b}$. The proposition in that case follows from the isomorphism~\eqref{TeMaw0} and from the two exact sequences: 
\begin{align*}
& 0 \to M \solid_{F^{\rm nr}} \B_{\crr}^+ \to M \solid_{F^{\rm nr}} \B_{\st}^+ \xrightarrow{N} M \solid_{F^{\rm nr}} \B_{\st}^+ \to 0, \\
& 0 \to (M \solid_{F^{\rm nr}} \B_{\crr}^+)^{\varphi=p^r} \to M \solid_{F^{\rm nr}} \B_{\crr}^+ \xrightarrow{p^r- \varphi} M \solid_{F^{\rm nr}} \B_{\crr}^+ \to 0
\end{align*} 
for $M$ a finite $(\varphi,N)$-module (the first one is proved by recurrence on the integer $m$ such that $N^m=0$ and the second one is proved in the same way as in \cite[Remark 2.30]{CN1}). 

In the general case, take $\US_{\bullet}$ an \'etale hypercovering of $X$ from $\MS_C^{\dagger, \rm ss,b}$. By the previous paragraph, the proposition is true for each $\US_{i,C}$. We then use the two following spectral sequences to conclude: 
\begin{align*}
& E_2^{i,j}= H^i(\rg_{\hk}(\US_{j,C}) \dsolid_{F^{\rm nr}} \Bstp) \Rightarrow  H^{i+j}(\rg_{\hk}(X) \dsolid_{F^{\rm nr}} \Bstp) \\
& E_2^{i,j}= H^i_{\hk}(\US_{j,C}) \solid_{F^{\rm nr}} \Bstp \Rightarrow  H^{i+j}_{\hk}(X) \solid_{F^{\rm nr}} \Bstp
\end{align*}
(to get the second one, we use the flatness of $\Bstp$ and the fact that the differentials are $\Bstp$-linear). 
\end{proof}

As in the rigid case, we can write syntomic cohomology in terms of Hyodo-Kato and de Rham cohomologies:
\begin{proposition}
\label{ciud-4}
Let $X$ be a smooth dagger variety over $C$. There is a natural  quasi-isomorphism in $\sd(\Q_{p,\Box})$: 
\[ \rg_{\synt,c}(X,\Q_p(r)) \xrightarrow{\sim} [ \HK_c(X,r) \to \DR_c(X,r) ]. \]
\end{proposition}

\begin{proof}
It suffices to show the result for the $\natural$-cohomologies. Let $X$ be an affinoid dagger space and $\{ X_h \}$ be a presentation of $X$. Consider the following commutative diagram:
\[ \xymatrix{
\rg^{\natural}_{\synt,c}(X,\Q_p(r)) \ar[r] \ar[d] & \rg^{\natural}_{\synt}(X,\Q_p(r)) \ar[r] \ar[d] & \rg^{\natural}_{\synt}(\partial X,\Q_p(r)) \ar[d]  \\
\HK^{\natural}_c(X,r) \ar[r] \ar[d] & \HK^{\natural}(X,r) \ar[r] \ar[d] & \HK^{\natural}(\partial X,r) \ar[d] \\
\DR^{\natural}_c(X,r) \ar[r] & \DR^{\natural}(X,r) \ar[r] &\DR^{\natural}(\partial X,r), } \]
where we set
 \begin{align*} 
& \HK^{\natural}_c(X,r) := [\R\Gamma^{\natural}_{\HK,c}(X) \dsolid_{F^{\rm nr}} \Bstp]^{N=0,\varphi=p^r},\\
& \DR^{\natural}_c(X,r):= \R\Gamma^{\natural}_{\dr,c}(X/\B^+_{\dr})/F^r.
\end{align*} 
The first row in the above diagram  is a distinguished triangle by definition, the second and third ones are  clearly distinguished triangles. The second and third column are distinguished triangles by Proposition~\ref{syn-HK-dR}. It follows that the first column is a distinguished triangle as well. 
\end{proof}

  We will need the following:
\begin{lemma}
\label{comm-synHK2}
Let $j:X \hookrightarrow Y$ be an open immersion of smooth dagger varieties over $C$, then the following diagram  commutes: 
\[ \xymatrix{
\rg_{\synt,c}(X,\Q_p(r)) \ar[r]^{j_*} \ar[d] & \rg_{\synt,c}(Y,\Q_p(r)) \ar[d] \\
\HK_c(X,r) \ar[r]^{j_*} & \HK_c(Y,r).
 }\]
 Similarly, for the Hyodo-Kato map $\iota_{\hk}: \R\Gamma_{\hk,c}(-)\to \R\Gamma_{\dr,c}(-)$.
\end{lemma}

\begin{proof}
We can work locally.  Let $X \hookrightarrow Y$ be an open immersion of dagger affinoids and take $\{ X_h \}$ and $\{ Y_h \}$ presentations of $X$ and $Y$ respectively, such that we have open immersions $X_h \hookrightarrow Y_h$.
Recall that the maps $j_*: \rg^{\natural}_{\synt,c}(X,r) \to \rg^{\natural}_{\synt,c}(Y,r)$ and $\HK_c(X,r) \to \HK_c(Y,r)$ are induced by the compositions:
\begin{align*}
\rg^{\natural}_{\synt,c}(X,r) & = \colim_{h} \rg^{\natural}_{\synt,\wh{X}}(X_h,r) \xleftarrow{\sim} \colim_{h} \rg^{\natural}_{\synt,\wh{X}}(Y_h,r) \to  \colim_{h} \rg^{\natural}_{\synt,\wh{Y}}(Y_h,r)=  \rg^{\natural}_{\synt,c}(Y,r), \\
\HK^{\natural}_{c}(X,r) &\stackrel{\sim}{\to} \colim_{h} \HK^{\natural}_{\wh{X}}(X_h,r) \xleftarrow{\sim} \colim_{h} \HK^{\natural}_{\wh{X}}(Y_h,r) \to  \colim_h \HK^{\natural}_{\wh{Y}}(Y_h,r) \xleftarrow{\sim} \HK^{\natural}_{c}(Y,r).
\end{align*}
The commutativity of the diagram in the lemma is now clear. The argument for the Hyodo-Kato map is analogous.
\end{proof}

\subsection{Period morphism} We show that we have a period isomorphism comparing syntomic  cohomologies with  pro-\'etale  cohomologies with compact support from the previous section. This will allow us to compute explicitly $p$-adic pro-\'etale cohomologies in the next chapter of the paper.
\subsubsection{Rigid analytic period morphism} 
Recall that, for all $X \in {\rm Sm}_C$, we have a period morphism in $\sd(\Q_{p,\Box})$
$$
\alpha_{r}: \R\Gamma_{\synt}(X,\Q_p(r))\to \R\Gamma_{\proeet}(X,\Q_p(r))
$$
that is a quasi-isomorphism after truncation $\tau_{\leq r}$ (see   \cite[Corollary 7.3]{CN3}). By functoriality, it induces a period morphism  in $\sd(\Q_{p,\Box})$
\begin{equation}
\label{alpha-c}
\alpha_{r,c}: \R\Gamma_{\synt,c}(X,\Q_p(r))\to \R\Gamma_{\proeet,c}(X,\Q_p(r)).
\end{equation}

\begin{theorem}
\label{per-synet}
For  $X$ in  $\Sm_C$ and  $r\ge0$, the map $\alpha_{r,c}$ is a quasi-isomorphism in ${\cal D}({\Q_{p,\Box}})$:
\[ \alpha_{r,c}: \tau_{\le r} \R\Gamma_{\synt,c}(X,\Q_p(r)) \xrightarrow{\sim} \tau_{\le r} \R\Gamma_{\proeet,c}(X, \Q_p(r)). \]
\end{theorem}

\begin{proof} This follows from the map of distinguished triangles: 
\[ \xymatrix{ 
\R\Gamma_{\synt,c}(X, r) \ar[d] \ar[r] & R\Gamma_{\synt}(X, r) \ar[d]^{\wr}_{\tau_{\leq r}} \ar[r] & \R\Gamma_{\synt}(\partial X, r) \ar[d]^{\wr}_{\tau_{\leq r}} \\
\R\Gamma_{\proeet,c}(X, \Q_p(r)) \ar[r] & R\Gamma_{\proeet}(X, \Q_p(r))   \ar[r] & \R\Gamma_{\proeet}(\partial X, \Q_p(r)).  \\ 
} \] 
The second and the third vertical maps are quasi-isomorphisms (after truncation  $\tau_{\le r}$). Note here that truncation commutes with colimits. Hence  the first vertical map is a quasi-isomorphism after truncation $\tau_{\leq r}$ as well.
\end{proof}

 \subsubsection{Overconvergent period morphism}
 
Let $X$ be a dagger affinoid over $C$ and let $\{X_h\}$ be a presentation of $X$. We have compatible period maps in $\sd(\Q_{p,\Box})$
\[ \rg_{\synt}(X_h,\Q_p(r)) \to \rg_{\proeet}(X_h, \Q_p(r)), \quad \text{ and } \quad \rg_{\synt}(X_h\setminus \wh{X},\Q_p(r)) \to \rg_{\proeet}(X_h\setminus \wh{X}, \Q_p(r)) \]
that are quasi-ismorphisms after truncation $\tau_{\le r}$. They induce a local period morphism for the compactly supported cohomology:
 \[ \alpha_{r,c}^{\natural}: \rg^{\natural}_{\synt,c}(X,\Q_p(r)) \to \rg^{\natural}_{\proeet,c}(X, \Q_p(r)), \] 
and this map is a quasi-isomorphism in $\sd(\Q_{p.\Box})$ after taking the truncation.

 We can know globalize the above construction and we obtain:
\begin{theorem}
\label{alpha-c-dagger}
For  $X$ in  $\Sm^{\dagger}_C$ and  $r\ge0$, there is a quasi-isomorphism in ${\cal D}(\Q_{p,\Box})$:
\[ \alpha^{\dagger}_{r,c}: \tau_{\le r} \R\Gamma_{\synt,c}(X,\Q_p(r)) \xrightarrow{\sim} \tau_{\le r} \R\Gamma_{\proeet,c}(X, \Q_p(r)). \]
\end{theorem}

\subsubsection{Geometrization of the period morphism} 
We define ${\rm TVS}$ (for Topological Vector Spaces) as  the $\infty$ derived category of $\underline{\Q}_p$-modules in presheaves of solid  abelian groups  on the category ${\rm Perf}_C$ of perfectoid spaces over $C$. The usual Fontaine's period rings $$\B_{\dr}^+,\B_{\crr}^+, \B_{\st}^+, \B_{\dr},\B_{\crr}, \B_{\st}$$ have canonical  structures of ${\rm TVS}$'s that we will denote by  ${\bb B}_{\dr}^+,{\bb B}_{\crr}^+, {\bb B}_{\st}^+, {\bb B}_{\dr},{\bb B}_{\crr}, {\bb B}_{\st}$, respectively. As in~\cite[\S 7]{CN4}, we will see that we can endow the cohomologies defined in the previous sections with a ${\rm TVS}$-structure and that the period morphism can be lifted to  this category.

\vskip.2cm 
(i) {\em Rigid analytic setting.} If $X$ is a smooth rigid analytic variety over $C$, we write ${\bb H}^i_{\proeet}(X, \Q_p), i\in\N$, (respectively ${\bb R}_{\proeet}(X, \Q_p)$) for the functor on ${\rm Perf}_C$ defined as
\[ \Lambda \mapsto H^i_{\proeet}(X_{\Lambda}, \Q_p) \quad \text{(respectively, } \Lambda \mapsto \rg_{\proeet}(X_{\Lambda} \Q_p) ), \]
where $X_{\Lambda}$ is the base change $X\times_C \Lambda$. These define objects in ${\rm TVS}$.

On the other hand, the syntomic cohomology can be geometrized by tensoring $\rg_{\hk}(X) \otimes^{\LL_{\Box}}_{\breve{F}}\B_{\st}^+$ with $- \otimes_{\B_{\st}^+}^{\LL_{\Box}} {\bb B}_{\st}^+$ and $F^r\rg_{\dr}(X/ \B_{\dr}^+)$ with $- \otimes_{\B_{\dr}^+}^{\LL_{\Box}} {\bb B}_{\dr}^+$ to obtain ${\bb R}_{\rm syn}(X, \Q_p(r)), r \ge 0$. 

Recall that the period morphism can be lifted to ${\rm TVS}$: 

\begin{proposition}{\rm ({\em Colmez-Nizio{\l},  \cite[Th.~7.3]{CN4}})}
\label{geometrization}
Let $X\in {\rm Sm}_C$ and $r \ge 0$. The functorial period map in $\DS(\Q_{p, \Box})$
\[ \alpha_r : \rg_{\rm syn}(X, \Q_p(r)) \to  \rg_{\proeet}(X, \Q_p(r)) \]
lifts to a functorial map of TVS’s (with values in $\DS(\Q_{p, \Box} )$) :
\[      {\bbalpha}_r : {\bb R}_{\rm syn}(X, \Q_p(r)) \to {\bb R}_{\proeet}(X, \Q_p(r)), \]
which is a quasi-isomorphism after truncation $\tau_{\le r}$.
\end{proposition} 

Similarly, for $X$ a smooth rigid analytic variety over $C$, we endow $\rg_{\proeet,c}(X, \Q_p(r))$ and $\rg_{{\rm syn},c}(X, \Q_p(r))$ with a ${\rm TVS}$-structure by setting
\begin{align*}
&{\bb R}_{\proeet,c}(X, \Q_p(r)) :  \Lambda \mapsto \rg_{\proeet}(X_{\Lambda}, \Q_p(r)), \\
&{\bb R}_{{\rm syn},c}(X, \Q_p(r)): \Lambda \mapsto [ [\rg_{\hk,c}(X)\otimes^{\LL_{\Box}}_{\breve{F}} {\bb B}_{\st}^+(\Lambda)]^{\varphi=p^r,N=0} \to (\rg_{\dr,c}(X/ \B_{\dr}^+)\otimes_{\B_{\dr}^+}^{\LL_{\Box}} {\bb B}_{\dr}^+(\Lambda))/F^r ].
\end{align*} 
There are canonical quasi-isomorphisms: 
\begin{align*}
{\bb R}_{\proeet,c}(X, \Q_p(r)) &\simeq [ {\bb R}_{\proeet}(X, \Q_p(r)) \to {\bb R}_{\proeet}(\partial X, \Q_p(r)], \\ 
{\bb R}_{{\rm syn},c}(X, \Q_p(r)) &\simeq [ {\bb R}_{\rm syn}(X, \Q_p(r)) \to {\bb R}_{\rm syn}(\partial X, \Q_p(r)], 
\end{align*}
where ${\bb R}_{*}(\partial X, \Q_p(r)):= \colim_{Z \in \Phi} {\bb R}_*(X \setminus Z, \Q_p(r))$, for $* \in \{ \proeet, {\rm syn}\}$. It follows easily  from Proposition~\ref{geometrization} that the period morphism for compactly supported cohomology can be lifted to the category ${\rm TVS}$:
\begin{proposition} There exists a natural map of ${\rm TVS}$ (with values in $\DS(\Q_{p, \Box} )$) :
\begin{equation}\label{geometrization2} { \bbalpha}_{r,c} : {\bb R}_{{\rm syn},c}(X, \Q_p(r)) \to {\bb R}_{\proeet,c}(X, \Q_p(r)), \end{equation} 
lifting the morphism~\eqref{alpha-c}. Moreover, it is a quasi-isomorphism after truncation $\tau_{\le r}$.
\end{proposition}

\vskip.2cm
(ii){\em The overconvergent setting.}  We have a similar picture for dagger varieties. Let $X$ be a dagger affinoid over $C$, let $\{X_h\}_{h\in\N}$ be a presentation, and $\widehat{X}$ the associated rigid analytic variety. We define the ${\rm TVS}$:
\[ {\bb R}_{*,c}^{\dagger}(X, \Q_p(r)) :=  \colim_h [{\bb R}_{*}(X_h, \Q_p(r)) \to {\bb R}_{*}(X_h\setminus \widehat{X}, \Q_p(r))]  , \quad r\ge 0, \quad \text{ for } * \in \{ \proeet, {\rm syn}\}. \]
We then extend this definition to any smooth dagger variety $X$ over $C$, by \'etale cosheafification and get ${\rm TVS}$'s:  ${\bb R}_{*,c}(X, \Q_p(r))$. 

 Now, for smooth dagger affinoids $X$ over $C$, using the rigid analytic period morphism $\bbalpha_{r,c}$ from \eqref{geometrization2}, we get a period morphism of ${\rm TVS}$'s
\[  { \bbalpha}^{\dagger}_{r,c} : {\bb R}_{{\rm syn},c}(X, \Q_p(r)) \to {\bb R}_{\proeet,c}(X, \Q_p(r)), \quad r \ge 0, \]
which globalizes to a period morphism 
that lifts the period morphism $\alpha^{\dagger}_{r,c}$ from Theorem~\ref{alpha-c-dagger}. It is a quasi-isomorphism after truncation $\tau_{\le r}$. To sum up: 
\begin{proposition} There exists a natural map of ${\rm TVS}$ (with values in $\DS(\Q_{p, \Box} )$) :
\begin{equation}\label{geometrization3} { \bbalpha}^{\dagger}_{r,c} : {\bb R}_{{\rm syn},c}(X, \Q_p(r)) \to {\bb R}_{\proeet,c}(X, \Q_p(r)), \end{equation} 
lifting the morphism~\eqref{alpha-c-dagger}. It is a quasi-isomorphism after truncation $\tau_{\le r}$.
\end{proposition}

\section{Computations} 
We present here  computations of  geometric compactly supported $p$-adic pro-\'etale cohomology of some Stein varieties using the period morphisms from Section~\ref{proetale-c} and the duality results from Section~\ref{dualities}. In particular, our computations show the existence of trace maps that will be used to prove  Poincar\'e duality for geometric $p$-adic pro-\'etale cohomology in \cite{CGN2}. 

\subsection{Affine space}
Let us first consider the case of the affine rigid space of dimension $d$ over $K$. Using the distinguished triangle from Proposition~\ref{ciud-3}, we can compute its compactly supported syntomic cohomology.   

\begin{proposition} \label{SYNcAff}Let $X:={\mathbb A}_K^d$, $d\geq 1$, be the rigid analytic affine space of dimension $d$ over $K$. 
\begin{enumerate}
\item Let $r\geq 0$. 
 We have the following isomorphisms in $\Q_{p,\Box}$
\begin{align*}
&  H^{i}_{\synt,c}(X_C,\Q_p(r)) =0, \quad i<d, i>2d;\\
& H^{i}_{\synt,c}(X_C,\Q_p(r)) \simeq 
 (H_c^d(X_C, \Omega^{i-d-1})/\kker d)\otimes_K^{\Box}{\rm gr}_F^{r-i+d}\B^+_{\dr}, \quad    d \le i \le 2d-1. 
\end{align*}
\item If $0\leq r<d$ then $H_{\synt,c}^{2d}(X_C, \Q_p(r)) =0$. If $r\geq d$ then we have an exact sequence  in $\Q_{p,\Box}$
$$ 0 \to (H_c^d(X_C, \Omega^{d-1})/ \kker d) (r-d) \to H_{\synt,c}^{2d}(X_C, \Q_p(r)) \to \Q_p(r-d) \to 0. 
$$
\end{enumerate}
\end{proposition}
\begin{proof} Clearly   $H^{i}_{\synt,c}(X_C,\Q_p(r)) =0$, for $i<d$. 

 Let   $0\leq i\leq d-1$. Since 
$ H^{i+d}_{\hk,c}(X_C,r)= H^{i+d-1}_{\hk,c}(X_C,r)=0, $ 
we obtain from  Lemma \ref{DRcAff} that
\[ H^{i+d}_{\synt,c}(X_C,\Q_p(r)) \simeq H^{i+d-1}\DR_c(X_C,r) \simeq
 (H_c^d(X_C, \Omega^{i-1})/\kker d)\otimes^{\Box}_K{\rm gr}_F^{r-i}\B^+_{\dr}.
 \]
For the same reasons $H^{i+d}_{\synt,c}(X_C,\Q_p(r)) =0$ if $i> d+1$. 

     For  $i=d$, by Section \ref{affine}, we have
\[
H^{2d}\hk_c(X_C,r)  \simeq (\breve{F}\{-d \} \solid_{\breve{F}} \Bstp)^{N=0,\varphi=p^r} = (\Bstp)^{N=0,\varphi=p^{r-d}}=(\Bcrp)^{\varphi=p^{r-d}}. 
\] 
And, by Lemma \ref{DRcAff}, we have  
\[ H^{2d} \DR_c(X_C,r) \simeq \B^+_{\dr}/ F^{r-d}.
\]
Hence we get the exact sequence: 
\[ 0 \to  H^{2d-1} \DR_c(X_C,r)\to H_{\synt,c}^{2d}(X_C, \Q_p(r)) \to ( \Bcrp)^{\varphi=p^{r-d}} \to \B^+_{\dr}/F^{r-d} \to H_{\synt,c}^{2d+1}(X_C, \Q_p(r))\to 0 . \]
If $r<d$,  this yields
$$
H_{\synt,c}^{2d}(X_C, \Q_p(r))=H_{\synt,c}^{2d+1}(X_C, \Q_p(r))=0.
$$
If $r\geq d$, we have a fundamental exact sequence of $p$-adic Hodge Theory:
$$
0\to \Q_p(r-d) \to ( \Bcrp)^{\varphi=p^{r-d}} \to \B^+_{\dr}/F^{r-d} \to 0.
$$
Hence $H_{\synt,c}^{2d+1}(X_C, \Q_p(r))=0$ and we get the exact sequence
$$ 0 \to (H_c^d(X_C, \Omega^{d-1})/ \kker d) (r-d) \to H_{\synt,c}^{2d}(X_C, \Q_p(r)) \to \Q_p(r-d) \to 0, $$
 as wanted.
\end{proof}

 The following corollary follows immediately from Proposition \ref{SYNcAff} and the period isomorphism from Theorem~\ref{per-synet}.
\begin{corollary}
Let $X:={\mathbb A}_K^d$, $d\geq 1$, be the rigid analytic affine space of dimension $d$ over $K$. 
\begin{enumerate}
\item 
 We have the following isomorphisms  in $\Q_{p,\Box}$
\begin{align*}
&  H^{i}_{\proeet,c}(X_C,\Q_p) =0, \quad i<d, i>2d;\\
& H^{i}_{\proeet,c}(X_C,\Q_p(i-d)) \simeq 
 H_c^d(X_C, \Omega^{i-d-1})/\kker d \mbox{ if }  d\le i \le 2d-1.
\end{align*}
\item  We have an exact sequence  in $\Q_{p,\Box}$
$$ 0 \to H_c^d(X_C, \Omega^{d-1})/ \kker d\to H_{\proeet,c}^{2d}(X_C, \Q_p(d)) \to \Q_p \to 0. 
$$
\end{enumerate}
\end{corollary}
\begin{example}[Dimension 1] \label{dim1-aff}
For $d=1$ and $r\geq 0$, Proposition \ref{SYNcAff} yields:
\begin{enumerate}[label=\roman*)]
\item $H^0_{\synt,c}({\mathbb A}_C^1, \Q_p(r))=0$, $H^i_{\synt,c}({\mathbb A}_C^1, \Q_p(r))=0$, for $i\geq 3$.
\item  $H^1_{\synt,c}({\mathbb A}_C^1, \Q_p(r))=0.$ 
\item $H^2_{\synt,c}({\mathbb A}_C^1, \Q_p)=0$.
\item For $r\geq 1$, an exact sequence of solid $\Q_p$-modules : \[ 0 \to H^1_c({\mathbb A}_C^1,\so) (r-1) \to H^2_{\synt,c}({\mathbb A}_C^1, \Q_p(r)) \to \Q_p(r-1) \to 0. \] 
\end{enumerate}
In the last point we have used the fact that the map $H^1_c({\mathbb A}_C^1,\so)\to H^1_c({\mathbb A}_C^1,\Omega^1)$ is injective (since $H^1_{\dr,c}({\mathbb A}_C^1)=0$).

  This yields:
  \begin{enumerate}
  \item  $H^i_{\proeet,c}({\mathbb A}_C^1, \Q_p)=0$, for $i\neq 2$.
  \item an exact sequence of solid $\Q_p$-modules
   \[ 0 \to H^1_c({\mathbb A}_C^1,\so) \to H^2_{\proeet,c}({\mathbb A}_C^1, \Q_p(1)) \to \Q_p\to 0. \] 
 \end{enumerate}
\end{example}

\subsection{Stein curves} We now compute the compactly supported $p$-adic cohomology of Stein curves. Let $X$ be a smooth Stein variety over $K$ of dimension $d= 1$.
  For $i\geq 0$, we have an exact sequence
$$\xymatrix{
H^{i-1}{\rm HK}_c(X_C,i)\ar[r] &  H^{i-1}{\rm DR}_c(X_C,i)\ar[r]& H^i_{\proeet,c}(X_C,\Q_p(i))\ar[r] & 
  H^{i}{\rm HK}_c(X_C,i)\ar[r]&
  H^{i}{\rm DR}_c(X_C,i)
}
$$
 We have in $\sd(\Q_{p,\Box})$: 
\begin{align*}
\rg_{\dr,c}(X)\dsolid_{K}\B^+_{\dr} & \simeq (H^1_c(X,\so_X)\solid_{K}\B^+_{\dr}\to H^1_c(X,\Omega^1_X)\solid_{K}\B^+_{\dr})[-1],\\
F^i(\rg_{\dr,c}(X)\dsolid_{K}\B^+_{\dr}) & \simeq  (H^1_c(X,\so_X)\solid_{K}F^i\B^+_{\dr}\to H^1_c(X,\Omega^1_X)\solid_{K}F^{i-1}\B^+_{\dr})[-1],\\
{\rm DR}_c(X_C,i)& \simeq (H^1_c(X,\so_X)\solid_{K}(\B^+_{\dr}/F^i)\to H^1_c(X,\Omega^1_X)\solid_{K}(\B^+_{\dr}/F^{i-1}))[-1].
\end{align*}
In particular $H^j{\rm DR}_c(X_C,i)=0$, for $j\neq 1,2.$

We also have (see Proposition~\ref{coho-hk11}): 
\begin{equation}
\label{hk-camb}
 H^{j}{\rm HK}_c(X_C,i)\simeq (H^j_{{\rm HK},c}(X_C)\solid_{F^{\rm nr}}\B^+_{\st})^{N=0,\phi=p^i}.
\end{equation}

By the Poincar\'e duality
$$
H^{j}_{{\rm HK},c}(X_C)\simeq H^{2-j}_{{\rm HK}}(X_C)^*\{-1\},\quad H^{j}_{{\rm dR},c}(X_C)\simeq H^{2-j}_{{\rm dR}}(X_C)^*\{-1\},
$$
where the first dual is the $F^{\rm nr}$-dual and the second one -- the  $C$-dual; the twist $\{n\}$ on de Rham cohomology shifts the filtration by $n$, i.e., $F^i_{\rm new}=F^{i+n}_{\rm old}$. It follows that 
\begin{align}
\label{camb1}
 & H^2_{{\rm HK},c}(X_C)\simeq F^{\rm nr}\{-1\}; \quad H^i_{{\rm HK},c}(X_C)=0, \mbox{ for } i\neq 1,2;\\
 & H^2_{{\rm dR},c}(X_C)\simeq C\{-1\}; \quad H^i_{{\rm dR},c}(X_C)=0, \mbox{ for } i\neq 1,2.\notag
\end{align}
 This implies  that  $H^i_{\proeet,c}(X_C,\Q_p)=0$ for $i\neq 1,2,3$.

 \vskip.2cm
    $\bullet$ {\em The group $H^1_{\proeet,c}(X_C,\Q_p)$.}
   We have the exact sequence of solid $\Q_p$-modules
    $$
  \xymatrix{
H^{0}{\rm HK}_c(X_C,1) \ar[r] &  H^{0}{\rm DR}_c(X_C,1)\ar[r] & H^1_{\proeet,c}(X_C,\Q_p(1))\ar[r] & H^{1}{\rm HK}_c(X_C,1)\ar[r]&
  H^{1}{\rm DR}_c(X_C,1)
}$$
By \eqref{camb1}, this yields the exact sequence
\begin{equation}
\label{vend000}
  \xymatrix{
0 \ar[r]  & H^1_{\proeet,c}(X_C,\Q_p(1)) \ar[r] & H^{1}{\rm HK}_c(X_C,1)\ar[r]&
  H^{1}{\rm DR}_c(X_C,1).
}
\end{equation}
From Lemma~\ref{DRc}, we obtain an isomorphism 
$ H^1\DR_c(X_C,1) \simeq H^1_c(X_C, \so)   $
and from Proposition~\ref{coho-hk11}, $H^1\HK_c(X_C,1) \simeq (H^1_{\hk,c}(X_C)\solid_{F^{\rm nr}} \Bstp)^{N=0, \varphi=p}$. But the map 
$$ (H^1_{\hk,c}(X_C)\solid_{F^{\rm nr}} \Bstp)^{N=0, \varphi=p} \to H^1_c(X_C, \so)  $$
factorizes as 
\begin{equation}
\label{vend001}
 (H^1_{\hk,c}(X_C)\solid_{F^{\rm nr}} \Bstp)^{N=0, \varphi=p} \to H^1_{\dr, c}(X_C)  \to H^1_c(X_C, \so) 
 \end{equation}
where the second morphism is injective. So the kernel of the second arrow in \eqref{vend000} is isomorphic to the kernel of the first arrow in \eqref{vend001}. This is equal to $ (H^1_{\hk,c}(X_C)\solid_{F^{\rm nr}} \Bstp)^{N=0, \varphi=1}$ (see \cite[Lemma 3.39]{CDN3} or Lemma~\ref{HKc} below) and so we get 
\[ H^1_{\proeet,c}(X_C,\Q_p(1)) \xrightarrow{\sim} (H^1_{\hk,c}(X_C)\solid_{F^{\rm nr}} \Bstp)^{N=0, \varphi=1}. \]

 \vskip.2cm
 $\bullet$ {\em The group $H^2_{\proeet,c}(X_C,\Q_p)$.} The long exact sequence above gives the exact sequence: 
  $$
  \xymatrix{
H^{1}{\rm HK}_c(X_C,2)\ar[r] &  H^{1}{\rm DR}_c(X_C,2)\ar[r] & H^2_{\proeet,c}(X_C,\Q_p(2))\ar[r] &  H^{2}{\rm HK}_c(X_C,2)\ar[r]&
  H^{2}{\rm DR}_c(X_C,2)
}$$
Using \eqref{hk-camb} and \eqref{camb1} we get the exact sequence
\begin{equation}
\label{camb2}
\xymatrix{
H^{1}{\rm HK}_c(X_C,2) \ar[r] &  H^{1}{\rm DR}_c(X_C,2)\ar[r]& H^2_{\proeet,c}(X_C,\Q_p(2))\ar[r] &  \B_{\crr}^{+,\phi=p}\ar[r]&
\B^+_{\dr}/F^1 
}
\end{equation}
Recall now that we have a fundamental exact sequence
$$
0\to\Q_p(1)\to  \B_{\crr}^{+,\phi=p}\to 
\B^+_{\dr}/F^1 \to 0.
$$
Using it we get
 the exact sequence
\begin{equation}
\label{vend0}
  \xymatrix{
H^{1}{\rm HK}_c(X_C,2) \ar[r] &  H^{1}{\rm DR}_c(X_C,2)\ar[r] & H^2_{\proeet,c}(X_C,\Q_p(2))\ar[r] & \Q_p(1)\ar[r] & 0
}
\end{equation}
where the first term is isomorphic to $(H^1_{\hk,c}(X_C) \solid_{F^{\rm nr}} \Bstp)^{N=0, \varphi=p^2}$.

\begin{remark}
By duality, the slopes of the Frobenius on $H^1_{\hk,c}(X_C)$ are between $0$ and $1$. If we assume they are $0$, then, we can show that we have an exact sequence   
\begin{equation}
\label{vend00}
0 \to (H^1_c(X_C, \so)/ \kker d)(1) \to H^2_{\proeet,c}(X_C,\Q_p(2)) \to \Q_p(1) \to 0.
\end{equation}
To see that first note that, from Lemma \ref{DRc}, we get an exact sequence
\begin{equation}
\label{vend1}
0\to H^1_c(X_C,\so)(1)\to H^{1}{\rm DR}_c(X_C,2)\to H^1_{\dr,c}(X_C)\to 0
\end{equation}
and since the slopes of the Frobenius on $H^1_{\hk,c}(X_C)$ are $\le 1$, we have the following exact sequence (see \cite[Lemma 3.39]{CDN3} or Lemma~\ref{HKc} below)
\begin{equation}
\label{vend2}
0 \to (H^1_{\hk,c}(X_C) \solid_{F^{\rm nr}} \Bstp)^{\varphi=p, N=0} \xrightarrow{t} (H^1_{\hk,c}(X_C) \solid_{F^{\rm nr}} \Bstp)^{\varphi=p^2, N=0} \to H^1_{\dr,c}(X_C) \to 0.   
\end{equation}
Combining \eqref{vend1} and \eqref{vend2}, we see that the cokernel of the first map in \eqref{vend0} is isomorphic to the cokernel of the map 
$$  (H^1_{\hk,c}(X_C) \solid_{F^{\rm nr}} t\Bstp)^{\varphi=p^2, N=0} \to  H^1_c(X_C,\so)(1)$$
 which in turn factorizes as
$$ (H^1_{\hk,c}(X_C) \solid_{F^{\rm nr}} t\Bstp)^{\varphi=p^2, N=0} \to H^1_{\dr,c}(X_C) (1) \to H^1_c(X_C,\so)(1).$$
But, if the slopes of the Frobenius are zero then the first map is surjective (see Lemma~\ref{HKc} below) so the cokernel we want is isomorphic to the cokernel of the map on the right, which gives the exact sequence \eqref{vend00}. 
\end{remark} 

  \vskip.2cm
  $\bullet$ {\em The group  $H^3_{\proeet,c}(X_C,\Q_p)$.}
  We claim that  $H^3_{\proeet,c}(X_C,\Q_p)=0$. Indeed,   we have the exact sequence
    $$
  \xymatrix{
 H^{2}{\rm HK}_c(X_C,3)\ar[r] &  H^{2}{\rm DR}_c(X_C,3)\ar[r] & H^3_{\proeet,c}(X_C,\Q_p(3))\ar[r] &  H^{3}{\rm HK}_c(X_C,3)
}$$
The last term on the right  is zero  for dimension reasons.   By  Lemma \ref{DRc} the first map is isomorphic to the Hyodo-Kato map
$$
(H^2_{\hk,c}(X_C)\otimes^{\Box}_{F^{\nr}} \B_{\st}^+)^{N=0, \phi=p^3}\to H^2_{\dr,c}(X)\otimes^{\Box}_KC
$$
Since $N=0$ and $\phi=p$ on $H^{2}_{\rm HK,c}(X_C)$ and we have compatible isomorphisms $H^{2}_{\rm HK,c}(X_C)\simeq F^{\nr}$, $H^{2}_{\dr,c}(X)\simeq K$, this map can be written as
$$
\B^{+,\phi=p^2}_{\crr}\stackrel{\iota}{\longrightarrow} \B^+_{\dr}/F^1,
$$
which is surjective, as wanted. 
\subsection{Trace maps} \label{trace-maps}
The purpose of this section is to define trace maps for the $p$-adic pro-\'etale cohomology of smooth partially proper varieties and dagger affinoids of dimension $d$ over $K$. To do that, we use the period isomorphism to express $H^{2d}_{\proeet,c}(X, \Q_p(r))$ in terms of Hyodo-Kato and de Rham cohomologies and then we use the trace morphisms from Section~\ref{dualities}. We then show that these trace maps satisfy various compatibilities.  

\subsubsection{Partially proper varieties and dagger affinoids}

Let $X$ be a smooth partially proper  variety or a smooth dagger affinoid over $K$. Assume that it is of dimension $d\geq 1$ and geometrically irreducible.
 We have  an exact sequence of solid $\Q_p$-modules
 \begin{equation}
 \label{Stein-tr}
    \xymatrix{
\ar[r] &  H^{2d}_{\proeet,c}(X_C,\Q_p(2d))\ar[r] &  (H^{2d}_{\rm HK,c}(X_C)\solid_{F^{\rm nr}}\B^+_{\st})^{N=0,\phi=p^{2d}} \ar[r]&
  H^{2d}{\rm DR}_c(X_C,2d)  \ar[r] & 
}\end{equation}
Since $N=0$ and $\phi=p^d$ on $H^{2d}_{\rm HK,c}(X_C)$, we get an exact sequence
$$
  \xymatrix{
\ar[r] &  H^{2d}_{\proeet,c}(X_C,\Q_p(2d))\ar[r] &  H^{2d}_{\rm HK,c}(X_C)\solid_{F^{\rm nr}}\B_{\crr}^{+,\phi=p^{d}} \ar[r]&
  H^{2d}_{\rm dR,c}(X)\solid_K\B^+_{\dr}/F^d & 
}$$
Via the compatible trace maps 
$$
{\rm tr}_{\hk,X_C}:  H^{2d}_{\rm HK,c}(X_C)\stackrel{\sim}{\to} F^{\nr},\quad  {\rm tr}_{\dr,X}:  H^{2d}_{\rm dR,c}(X)\stackrel{\sim}{\to} K
$$
this yields an exact sequence
$$
  \xymatrix{
\ar[r] &  H^{2d}_{\proeet,c}(X_C,\Q_p(2d))\ar[r] &  \B_{\crr}^{+,\phi=p^{d}} \ar[r]^{\iota}&
 \B^+_{\dr}/F^d\ar[r]  & 
}$$
Hence a natural trace map of solid $\Q_p$-modules
$$
{\rm tr}_{\proeet,X_C}: H^{2d}_{\proeet,c}(X_C,\Q_p(2d))\to \Q_p(d).
$$
This map is surjective. We will denote in the same way the derived trace map  in $\sd(\Q_{p,\Box})$
$$
{\rm tr}_{\proeet,X_C}: \R\Gamma_{\proeet,c}(X_C,\Q_p(2d))[2d]\to \Q_p(d)
$$
obtained as the composition
\begin{align*}
{\rm tr}_{\proeet,X_C}: \R\Gamma_{\proeet,c}(X_C,\Q_p(2d))[2d] & \stackrel{\sim}{\leftarrow} \tau_{\leq 2d}\R\Gamma_{\proeet,c}(X_C,\Q_p(2d))[2d]\\
 & \to H^{2d}_{\proeet,c}(X_C,\Q_p(2d))\lomapr{{\rm tr}_{\proeet,X_C}} \Q_p(d).
\end{align*}

\begin{remark}{\rm (Proper varieties)}  Assume that $X$ is proper in the above. 
Then we see that  the first map in the exact sequence \eqref{Stein-tr} is injective and it follows that the trace map is an isomorphism
$$
{\rm tr}_{\proeet, X_C}: H^{2d}_{\proeet,c}(X_C,\Q_p(2d))\stackrel{\sim}{\to} \Q_p(d).
$$
\end{remark}
\subsubsection{Properties of trace maps} The trace morphisms defined above satisfy the following compatibilities. 

\begin{lemma}\label{properties}The trace map ${\rm tr}_{\proeet,X_C}$  is compatible with 
\begin{enumerate}
\item the corresponding Hyodo-Kato and de Rham trace maps, i.e., the following diagrams commute in $\Q_{p,\Box}$
\begin{align*}
\xymatrix{
H^{2d}_{\proeet,c}(X_C,\Q_p(2d))\ar[d]^{{\rm tr}_{\proeet,X_C}} \ar[r] & (H^{2d}_{{\rm HK},c}(X_C)\solid_{F^{\rm nr}}\B_{\st}^+)^{N=0}\ar[d]^{{\rm tr}_{\hk,X_C}\otimes{\rm Id}} \\
\Q_p(d)\ar[r] & \B_{\crr}^{+}\{-d\},
}\quad 
\xymatrix{
H^{2d}_{\proeet,c}(X_C,\Q_p(2d))\ar[d]^{{\rm tr}_{\proeet,X_C}} \ar[r] & H^{2d}_{{\rm dR},c}(X)\solid_{K}\B_{\dr}^{+}\ar[d]^{{\rm tr}_{\dr,X}\otimes{\rm Id}} \\
\Q_p(d)\ar[r] & \B^{+}_{\dr}
}
\end{align*}
\item open immersions $j:X\hookrightarrow Y$, where $X, Y$ are or partially proper varieties or  dagger affinoids over $K$. 
\end{enumerate}
\end{lemma}
\begin{proof}
The compatibility (1) is clear from the construction. For  compatibility (2), we  first assume that $X$ and $Y$ are partially proper and consider the following diagram of solid $\Q_p$-modules: 
\begin{equation*}
\xymatrix{
(H^{2d}_{{\rm HK},c}(X_C)\solid_{F^{\rm nr}}\B_{\st}^{+})^{N=0} \ar[rrr]^{j_*} \ar[rddd]_{{\rm tr}_{X_C}\otimes{\rm Id}} & & & (H^{2d}_{{\rm HK},c}(Y_C)\solid_{F^{\rm nr}}\B_{\st}^{+})^{N=0} \ar[lddd]^{{\rm tr}_{Y_C}\otimes{\rm Id}} \\
& H^{2d}_{\proeet,c}(X_C,\Q_p(2d)) \ar[r]^{j_*} \ar[lu] \ar[d]^{{\rm tr}_{X_C}}&  H^{2d}_{\proeet,c}(Y_C,\Q_p(2d)) \ar[ru] \ar[d]^{{\rm tr}_{Y_C}} & \\
& \Q_p(d) \ar@{=}[r]  \ar@{^(->}[d] & \Q_p(d) \ar@{^(->}[d] & \\ 
& \B_{\crr}^{+}\{-d\} \ar@{=}[r] & \B_{\crr}^{+}\{-d\} &
}
\end{equation*}
We want to show that the upper square is commutative.  We know that the upper trapezoid is commutative. The right and left triangles are commutative by claim (1). Since the arrows from $\Q_p(d)$ to $\B_{\crr}^{+, \phi=p^d}$ are injective, it suffices then to prove the commutativity of the outer trapezoid. or that the geometric Hyodo-Kato trace is compatible with the open immersion $X\hookrightarrow Y$. But this was shown in Lemma \ref{HK-pp}.

   The proof is similar when $X$ and $Y$ are both dagger affinoids, using the proof of Theorem \ref{dualHK-ov}.
   \end{proof}

\subsubsection{Comparison with Berkovich's trace map} This section will not be used in the rest of the paper. Let $X$ be a smooth partially proper rigid analytic variety of dimension $d$, over $K$.  Recall (see Lemma \ref{proet-et-entier}) that we have a canonical map in $\sd(\Q_{p,\Box})$
$$
\R\Gamma_{\proeet,c}(X_C,\Q_p)\to \R\Gamma_{\eet,c}(X_C,\Q_p),
$$
where $\R\Gamma_{\eet,c}(X_C,\Q_p)=(\lim_n\R\Gamma_{\eet,c}(X_C,\Z/p^n))\otimes^{\Box}_{\Z_p}\Q_p$. We also have (compatible with $n$) Berkovich's trace maps (see  \cite[Sec. 7.2]{Ber93}) 
$$
{\rm tr}^B_{\eet,X_C,n}: \R\Gamma_{\eet,c}(X_C,\Z/p^n(d))[2d]\to \Z/p^n
$$
that induce a trace map in $\sd(\Q_{p,\Box})$
$$
{\rm tr}^B_{\eet,X_C}: \R\Gamma_{\eet,c}(X_C,\Q_p(d))[2d]\to \Q_p.
$$
We will call the composition 
$$
{\rm tr}^B_{\proeet, X_C}:\R\Gamma_{\proeet,c}(X_C,\Q_p(d))[2d]\to \R\Gamma_{\eet,c}(X_C,\Q_p(d))[2d]\verylomapr{{\rm tr}^B_{\eet,X_C}}\Q_p
$$
the {\em pro-\'etale Berkovich  trace map}.
\begin{proposition}
The two trace maps in $\sd(\Q_{p,\Box})$ 
$$\tr_{\proeet,X_C}, \tr^{B}_{\proeet,X_C}: \R\Gamma_{\proeet,c}(X_C,\Q_p(2d))[2d]\to \Q_p(d)$$
are equal. 
\end{proposition}
\begin{proof} By the spectral sequence \eqref{yunnan1},  we have a surjection in $\Q_{p,\Box}$
$$
\bigoplus_{i\in I}H^{2d}_{\proeet,c}(U_{i,C},\Q_p(2d)) \twoheadrightarrow H^{2d}_{\proeet,c}(X_C,\Q_p(2d)),
$$
where the index set $ \{U_i\}$ runs over a covering family of open discs. Hence we may assume that $X$ is an open disc. Then, since both trace maps are compatible with
open immersions of partially proper rigid analytic varieties (by Lemma \ref{properties} and \cite[Th. 5.3.3]{Zav}, respectively),  we may assume that $X$ is the projective space: $X={\mathbb P}^{d,{\rm an}}_C$.

   Now we can drop the compact support everywhere and then  pass to the algebraic category. It suffices thus   to compare the trace maps
\begin{equation}\label{finish2}
\wt{\tr}_{\eet, X_C}, \tr_{\eet, X_C}: H^{2d}_{\eet}({\mathbb P}^{d,{\rm alg}}_C,\Q_p(2d))\to \Q_p(d),
\end{equation}
where the first trace map is the algebraic analog of $\tr_{\proeet, X_C}$ (so defined via comparison theorems) and the second trace map is the classical algebraic \'etale trace map. That this passage works this way follows from GAGA-type theorems in \cite{Shao} and the compatibility of the algebraic and analytic \'etale trace maps. Now, the equality of the two trace maps in \eqref{finish2}   follows from the known fact that
 the classical comparison theorems are compatible with trace maps (this reduces to a comparison of Chern class maps).
\end{proof}

\section{Stein varieties}
We  construct  here a conditional  version of a fundamental diagram for geometric compactly supported $p$-adic pro-\'etale cohomology of  Stein varieties and  apply it to determine such  cohomology of tori and Drinfeld's spaces.  The proofs are similar to the ones of \cite{CDN3}, but more complicated due to the shift by $d$ that appears in the term  coming from   the compactly supported de Rham cohomology. 

\subsection{Fundamental diagram} \label{ex-stein} 
Let $X$ be a smooth dagger Stein variety of dimension $d$ over $K$. The purpose of this section is to study the compactly supported $p$-adic pro-\'etale cohomology of $X_C$. To do that, we need to extend some of the results of \cite{CDN3} and, in particular, their Lemma 3.39: 

\begin{lemma}
\label{HKc}
Let $j$, $k$ and $l$ be in $\N$ such that $k \le j$. There is an exact sequence in $\Q_{p,\Box}$:  
\[ 0 \to (H^l_{\hk,c}(X_C) \solid_{F^{\rm nr}} \Bstp)^{N=0, \varphi=p^{j-k}} \xrightarrow{t^{k+1}} (H^l_{\hk,c}(X_C) \solid_{F^{\rm nr}} \Bstp)^{N=0, \varphi=p^{j+1}} \xrightarrow{\pi^{l,j+1}_{k+1}} H^{l}_{\dr,c}(X) \solid_K (\B^+_{\dr}/F^{k+1}).  \]
Moreover, the right arrow is surjective if the slopes of the Frobenius on $H^l_{\hk,c}(X_C)$ are $\le j-k$. 
\end{lemma}

\begin{proof}
The case $k=0$ was proved in \cite[Lemma~3.39]{CDN3}. For $k \neq 0$, it is clear that the multiplication by $t^{k+1}$ is injective, it remains to show the exactness in the middle. To do that, as in \cite[Lemma~3.39]{CDN3}, it suffices to show that for all $\delta$, $(F^{k+1}\Bcrp)^{\varphi=p^{j+1- \delta}}=(t^{k+1} \Bcrp)^{\varphi=p^{j+1- \delta}}$, where we set $F^k \Bcrp=\Bcrp \cap F^k \B_{\dr}^+$. Clearly, $(t^{k+1}\Bcrp)^{\varphi=p^{j+1- \delta}}\subset(F^{k+1} \Bcrp)^{\varphi=p^{j+1- \delta}}.$
We prove the other inclusion  by induction. For $k=0$, we have 
$$(F^{1}\Bcrp)^{\varphi=p^{j+1- \delta}} \subset \{ x \in \Bcrp \; | \; \varphi^{n}(x) \in F^1\Bcrp \text{ for all } n >0 \}=t \Bcrp.$$ 
Let $k \ge 0$. By induction, it suffices to prove that for all $\delta$, $(F^{k+1}\Bcrp)^{\varphi=p^{j+1-\delta}}\subset (t F^k\Bcrp)^{\varphi=p^{j+1-\delta}}$. Let $y$ be in $(F^{k+1}\Bcrp)^{\varphi=p^{j+1-\delta}}$, then $y=t^{k+1}y_0$ with $y_0$ in $\B_{\dr}^+$ and we have
$$y \in \{ x \in \Bcrp \; | \; \varphi^{n}(x) \in F^1\Bcrp \text{ for all } n >0 \}=t \Bcrp.$$ 
We obtain that $t^ky_0$ is in $\Bcrp \cap F^k \B_{\dr}^+=F^k \Bcrp$, which concludes the proof of the first statement of the lemma.

Let us prove that if the slopes of the Frobenius on $H^l_{\hk,c}(X_C)$ are $\le j-k$, then $\pi^{l,j+1}_{k+1}$ is surjective. We show by induction that for all $1   \le m \le k+1$, the map \[ \pi^{l,j+1}_{m} : (H^l_{\hk,c}(X_C) \solid_F \Bstp)^{N=0, \varphi=p^{j+1}} \to H^{l}_{\dr,c}(X) \solid_K \B^+_{\dr}/F^{m} \] is surjective. For $m=1$, the result follows from \cite[Lemma~3.39]{CDN3} (the slopes of the Frobenius are less or equal to $ j-k$ so they are less or equal to $j$). Assume that $\pi^{l,j+1}_{m-1}$ is surjective. The surjectivity of $\pi^{l,j+1}_{m}$ will follow from the commutative diagram: 
\[ \xymatrix@C=10pt{
  H^{l}_{\dr,c}(X) \solid_K (F^{m-1}\B^+_{\dr}/F^{m}) \ar@{^(->}[r] & H^{l}_{\dr,c}(X) \solid_K (\B^+_{\dr}/F^{m}) \ar@{->>}[r] & H^{l}_{\dr,c}(X) \solid_K (\B^+_{\dr}/F^{m-1})  \\
 \ar[u] (H^l_{\hk,c}(X_C) \solid_{F^{\rm nr}} t^{m-1}\Bstp)^{N=0, \varphi=p^{j+1}} \ar@{^(->}[r] & (H^l_{\hk,c}(X_C) \solid_{F^{\rm nr}} \Bstp)^{N=0, \varphi=p^{j+1}} \ar[u]^{\pi^{l,j+1}_{m}} \ar@{=}[r] & (H^l_{\hk,c}(X_C) \solid_{F^{\rm nr}} \Bstp)^{N=0, \varphi=p^{j+1}} \ar[u]^{\pi^{l,j+1}_{m-1}}  \\
}\]
The first row is exact, the vertical arrow $\pi^{l,j+1}_{m-1}$ is surjective by the inductive assumption,  and the left vertical arrow is surjective since the morphism
\[ (H^l_{\hk,c}(X_C) \solid_{F^{\rm nr}} \Bstp)^{N=0, \varphi=p^{j+1-(m-1)}} \to  H^{l}_{\dr,c}(X) \solid_K (\B^+_{\dr}/F^1) \] 
is surjective (the slopes of the Frobenius are less or equal to $j-k$ so as long as $m$ is less than $k+1$, the slopes are less or equal to $j-m+1$). We deduce that $\pi^{l,j+1}_{m}$ is surjective, as wanted. \end{proof}

\begin{remark}Since the slopes of Frobenius on $H^l_{\hk}(X_C)$ are between $0$ and $i$, by duality, the slopes of the Frobenius on $H^l_{\hk,c}(X_C)$ are between $d-(2d-l)=l-d$ and $d-0=d$. 
\end{remark}

\begin{proposition}
\label{bal}
Assume that the maps $\pi^{r-1,r}_{d}$ and $\pi^{r-1,r-d}_1$ are surjective. Then there is an exact sequence in $\Q_{p,\Box}$:
\[ 0 \to (H_c^d(X_C, \Omega^{r-d-1})/ \kker d)(d) \to H_{\synt,c}^{r}(X_C, \Q_p(r)) \to (H^{r}_{\hk,c}(X_C) \solid_{F^{\rm nr}} t^d \Bstp)^{\varphi=p^{r}, N=0} \to 0. \]
\end{proposition}

\begin{proof}
Consider the following commutative diagram:
{\small \[ \xymatrix@C=10pt{
0 \ar[r] & (H^d_c(\Omega^{r-d})/{\rm Im}(d))(d-1) \ar[r] & H^r \DR_c(X_C,r) \ar[r] & H^r_{\dr,c}(X) \solid_K (\B_{\dr}^+/F^{d-1}) \ar[r] & 0 \\
0 \ar[r] & (\HK^r_c(X_C) \solid_{F^{\rm nr}} t^{d-1} \Bstp)^{N=0, \varphi=p^{r}} \ar[r] \ar[u]^{\gamma_r} & (\HK^r_c(X_C) \solid_{F^{\rm nr}} \Bstp)^{N=0, \varphi=p^{r}} \ar[r] \ar[u] & H^r_{\dr,c}(X) \solid_K (\B_{\dr}^+/F^{d-1}) \ar@{=}[u] & \\
& & H^r_{\synt,c}(X_C, \Q_p(r)) \ar[u] & & \\
0 \ar[r] & (H^d_c(\Omega^{r-d-1})/{\rm Im}(d))(d) \ar[r] & H^{r-1} \DR_c(X_C,r) \ar[u] \ar[r] & H^{r-1}_{\dr,c}(X) \solid_K (\B_{\dr}^+/F^{d}) \ar[r] & 0 \\
0 \ar[r] & (\HK^{r-1}_c(X_C) \solid_{F^{\rm nr}} t^{d} \Bstp)^{N=0, \varphi=p^{r}} \ar[r] \ar[u]^{\gamma_{r-1}} & (\HK^{r-1}_c(X_C) \solid_{F^{\rm nr}} \Bstp)^{N=0, \varphi=p^{r}} \ar[r]^-{\pi^{r-1,r}_d} \ar[u] & H^{r-1}_{\dr,c}(X) \solid_K (\B_{\dr}^+/F^{d}) \ar@{=}[u] \ar[r] & 0 
}\]}
where $\HK^*_c(X_C):=H^*_{\hk,c}(X_C)$, $H^d_c(\Omega^{*}):=H^d_c(X_C, \Omega^*)$ and the maps $\gamma_{r-1}$, $\gamma_r$ are obtained by factorization. The rows are exact sequences by Lemma~\ref{DRc} and Lemma~\ref{HKc} (and by assumption on $\pi_d^{r-1,r}$ for the bottow row). The column in the middle is an exact sequence as well.    
We deduce that we have an exact sequence: 
\[ 0 \to \mathrm{coker}(\gamma_{r-1}) \to H^r_{\synt,c}(X_C, \Q_p(r)) \to \kker(\gamma_r) \to 0 \] 
and it remains to compute $\kker(\gamma_r)$ and $\coker(\gamma_{r-1})$. 
\vskip.2cm
($\bullet$) {\em Kernel of $\gamma_r$.} The map $\gamma_r$ factorizes as 
\[ (H^{r}_{\HK,c}(X_C)\solid_{F^{\rm nr}} t^{d-1} \Bstp)^{N=0, \varphi=p^{r}} \xrightarrow{\gamma_{r}^{(1)}} H^{r}_{\dr,c}(X_C) (d-1) \xrightarrow{\gamma_{r}^{(2)}} (H_c^d(X_C,\Omega^{r-d})/ {\rm Im}(d)) (d-1) \]
with $\gamma_{r}^{(2)}$ injective. Then,
$$ \kker(\gamma_r) \simeq \kker(\gamma_r^{(1)}) \simeq (H^{r}_{\HK,c}(X_C)\solid_{F^{\rm nr}} t^{d} \Bstp)^{N=0, \varphi=p^{r}},$$
as wanted. 

\vskip.2cm
($\bullet$) {\em Cokernel of $\gamma_{r-1}$.} As above, the map $\gamma_{r-1}$ factorizes as    
\[ (H^{r-1}_{\HK,c}(X_C)\solid_{F^{\rm nr}} t^{d} \Bstp)^{N=0, \varphi=p^{r}} \xrightarrow{\gamma_{r-1}^{(1)}} H^{r-1}_{\dr,c}(X_C) (d) \xrightarrow{\gamma_{r-1}^{(2)}} (H_c^d(X_C,\Omega^{r-d-1})/ {\rm Im}(d))(d) \]
with $\gamma_{r-1}^{(2)}$ injective. The hypothesis that $\pi_1^{r-1,r-d}$ is surjective implies that $\gamma_{r-1}^{(1)}$ is surjective as well. So 
$$\coker(\gamma_{r-1}) \simeq \coker(\gamma_{r-1}^{(2)}) \simeq (H_c^d(X_C,\Omega^{r-d-1})/ \kker d) (d) $$
and we obtain the exact sequence from the proposition. 
\end{proof}


Using the previous lemma, we obtain a bicartesian diagram analogous to the one of \cite[Example 3.35]{CDN3} to compute the compactly supported syntomic cohomology of Stein varieties: 
\begin{theorem}
\label{diagSYNc}
Assume that the maps $\pi^{r-1,r}_{d}$ and $\pi^{r-1,r-d}_1$ are surjective. Then, there is a map of exact sequences in $\Q_{p,\Box}$:  
{\small\[ \xymatrix{
0 \ar[r] & (H_c^d(X_C, \Omega^{r-d-1})/ \kker d)(d) \ar[r] \ar@{=}[d] & H_{\synt,c}^{r}(X_C, \Q_p(r)) \ar[d]^{\alpha} \ar[r] & (H^{r}_{\HK,c}(X_C) \solid_{F^{\rm nr}} t^d \Bstp)^{\varphi=p^{r}, N=0} \ar[r] \ar[d]^{\beta} & 0 \\
0 \ar[r] & (H_c^d(X_C, \Omega^{r-d-1})/ \kker d)(d) \ar[r] & H_c^d(X_C, \Omega^{r-d})^{d=0}(d)  \ar[r] & H^r_{\dr,c}(X_C)(d) \ar[r] & 0 
}\]}
\end{theorem}

\begin{proof}
The first row is exact by the previous proposition and the exactness of the second one is implied by the definition of the de Rham cohomology $H^{i+d}_{\dr,c}(X_C)= H^i(H^d_c(X_C, \Omega^{\bullet}))$.  
The map $\alpha$ is given by the composition:
\[ H^r_{\synt,c}(X_C, \Q_p(r)) \to H^r(F^r(R\Gamma_{\dr,c}(X) \solid_K \B^+_{\dr})) \to H_c^d(X, \Omega_X^{r-d})^{d=0} \solid_K (F^d \B^+_{\dr}/F^{d+1}) \]
and $\beta$ is the morphism:
\[ (H^{r}_{\HK,c}(X_C) \solid_{F^{\rm nr}} t^d \Bstp)^{\varphi=p^{r}, N=0} \xrightarrow{\iota_{\HK} \otimes \iota} H^r_{\dr,c}(X) \solid_K F^d \B^+_{\dr} \to (H^r_{\dr,c}(X) \solid_K C)(d). \]
We then show that the diagram is commutative using the same argument as in the classical case \cite[Example 3.35]{CDN3}.
\end{proof}

\begin{corollary}
\label{proet-c-stein}
Let $X$ be a smooth Stein variety of dimension $d$ over K and $r\ge0$. Assume that the maps $\pi^{r-1,r}_{d}$ and $\pi^{r-1,r-d}_1$ are surjective.Then there is a bicartesian diagram in $\Q_{p,\Box}$: 
\[ \xymatrix{
H_{\proeet,c}^{r}(X_C, \Q_p(r)) \ar[r] \ar[d] &  (H^{r}_{\HK,c}(X_C) \solid_{F^{\rm nr}} t^d \Bstp)^{\varphi=p^{r}, N=0} \ar[d] \\ 
H_c^d(X_C, \Omega^{r-d})^{d=0} (d) \ar[r] &  H^r_{\dr,c}(X_C) (d)}. \]
\end{corollary}

\begin{remark}
If the slopes of the Frobenius on $H^{r-1}_{\hk,c}(X_C)$ are in $\{r-d-1,r-d\}$ then the map $\pi^{r-1,r}_{d}$ is surjective. If they are equal to $r-d-1$ then $\pi^{r-1,r-d}_1$ is surjective.
 \end{remark}
 
 \begin{remark}
 \label{diag-00}
We will now show that, if we drop the surjectivity assumption for $\pi^{r-1,r}_{d}$ (but we are still assuming that $\pi^{r-1,r-d}_1$ is surjective), then there is a short exact sequence:
\begin{equation}
\label{diag-01}
0 \to \coker(\pi^{r-1,r}_d) \to  H_{\synt,c}^{r}(X_C, \Q_p(r)) \to P(X_C,r) \to 0
\end{equation}
where $P(X_C,r)$ is the bicartesian diagram from the previous theorem, i.e., 
\[ P(X_C,r):= \kker((H^{r}_{\HK,c}(X_C) \solid_{F^{\rm nr}} t^d \Bstp)^{\varphi=p^{r}, N=0} \oplus H_c^d(X_C, \Omega^{r-d})^{d=0} (d) \to H^r_{\dr,c}(X_C) (d)). \]
We keep the same notations as in the proof of Propositon~\ref{bal} and we moreover denote by
$\iota_{\hk}^{r-1,r}$ the map $(H^{r-1}_{\hk,c}(X_C) \solid_{F^{\rm nr}}  \Bstp)^{\varphi=p^{r}, N=0} \to H^{r-1} \DR_c(X_C,r)$. Then, similar computations as before show that we have a commutative diagram:
\[ {\scriptscriptstyle \xymatrix{   
   & 0  & && \\
   & \coker(\pi^{r-1,r}_d) \ar[u] &&&  \\
0 \ar[r] & \coker(\iota_{\hk}^{r-1,r}) \ar[u] \ar[r] \ar@{-->}@/^1cm/[d]&  H_{\synt,c}^{r}(X_C, \Q_p(r)) \ar[r] \ar[d]^{\alpha} & (H^{r}_{\HK,c}(X_C) \solid_{F^{\rm nr}} t^d \Bstp)^{\varphi=p^{r}, N=0} \ar[d]^{\beta} \ar[r] & 0 \\
  0 \ar[r]  & ( H_c^d(X_C, \Omega^{r-d-1})/ \kker d) (d)  \ar[u] \ar[r]   &H_c^d(X_C, \Omega^{r-d})^{d=0} (d) \ar[r] & H^r_{\dr,c}(X_C) (d) \ar[r] & 0 \\
  & 0 \ar[u] & & & }}\]
  where the rows and the first column are exact (the vertical exact sequence follows from the snake lemma and the computation of $\gamma_{r-1}$ made in the proof of Proposition~\ref{bal}:
   $$\coker(\gamma_{r-1}) \simeq ( H_c^d(X_C, \Omega^{r-d-1})/ \kker d)(d) $$ using the assumption that $\pi^{r-1,r-d}_1$ is surjective). 
  The dotted arrow  is obtained by factorization of $\alpha$ and defines a section to the injection.
  
Now, the exact sequence \eqref{diag-01} follows from a diagram chase (note that the map on the left in the exact sequence is given by a section of the morphism from the cokernel of $\iota_{\hk}^{r-1,r}$ to the cokernel of $\pi^{r-1,r}_d$).   
 \end{remark}

\subsection{Examples} We will now use Theorem \ref{diagSYNc} to compute the compactly supported geometric $p$-adic pro-\'etale cohomology of tori and Drinfeld spaces. 
\subsubsection{Torus} 

We will use the previous formula to compute the compactly supported pro-\'etale cohomology of the torus $X:=\G_{m,K}^d$.   For $0\leq r\le d-1$, we have $H^r_{\proeet,c}(X_C,\Q_p)=0$ since $H^i_{\dr,c}(X_C)=0,$ for $i\leq d-1$. 
Assume thus that   $d \le r \le 2d$.  We start with the following

\begin{lemma}
For all $k \ge 0$, there is an isomorphism of solid $\Q_p$-modules
$$(H^{r-1}_{\HK,c}(X_C) \solid_{F^{\rm nr}} \Bstp)^{\varphi=p^{k}, N=0} \simeq  (\wedge^{2d-r+1} \Q_p^d )\solid_{\Q_p} (\Bcrp)^{\varphi=p^{k+d-r+1}}. $$  
\end{lemma}

\begin{proof}
By Section \ref{torus}, we have: 
\[(H^{r-1}_{\HK,c}(X_C) \solid_{F^{\rm nr}} \Bstp)^{\varphi=p^{k}, N=0} \simeq  ( \wedge^{2d-r+1} (F^{\rm nr})^d\{d-r+1\}  \solid_{F^{\rm nr}} \Bstp)^{\varphi=p^{k}, N=0}. \]  
The monodromy on $\wedge^{2d-r+1} (F^{\rm nr})^d$ is trivial. 
We compute
\begin{align*}
  (\wedge^{2d-r+1} (F^{\rm nr})^d\{d-r+1\}   \solid_{F^{\rm nr}} \Bstp)^{\varphi=p^{k}, N=0} 
  & \simeq (\wedge^{2d-r+1} (F^{\rm nr})^d\{d-r+1\}  )^{\varphi=p^{-(d-r+1)}} \solid_{\Q_p} (\Bcrp)^{\varphi=p^{k+d-r+1}} \\
  & \simeq (\wedge^{2d-r+1} \Q_p^d) \solid_{\Q_p} (\Bcrp)^{\varphi=p^{k+d-r+1}},
  \end{align*}
  which proves the lemma. 
  \end{proof}
Since we also have
\[ H^{r-1}_{\dr,c}(X)  \simeq (H_{\dr}^{2d-r+1}(X))^* \simeq \wedge^{2d-r+1} K^d, \]
we deduce that the maps
\begin{align*}
\pi_{1}^{r-1,r-d} : &  (\wedge^{2d-r+1} \Q_p^d) \solid_{\Q_p} (\Bcrp)^{\varphi=p} \to (\wedge^{2d-r+1} \Q_p^d )\solid_{\Q_p} C \\
\pi_d^{r-1,r} &( \wedge^{2d-r+1} \Q_p^d) \solid_{\Q_p} (\Bcrp)^{\varphi=p^{d+1}} \to (\wedge^{2d-r+1} \Q_p^d )\solid_{\Q_p} \B_{\dr}^+/F^d
\end{align*}
are surjective.
The same computation as in the proof of the lemma gives: 
\begin{align*}
(H^r_{\HK,c}(X_C) \solid_{F^{\rm nr}} t^d\Bstp)^{\varphi=p^{r}, N=0} \simeq  (\wedge^{2d-r} \Q^d_p)(d). \end{align*}
So, by Theorem \ref{diagSYNc} we obtain a  map of exact sequences in $\Q_{p,\Box}$: 
\begin{equation} 
\label{kol12}
\xymatrix{
0\ar[r] &H^d_c(X_C,\Omega^{r-d-1})/\kker d \ar@{=}[d]\ar[r] &H_{\proeet,c}^{r}(X_C, \Q_p(r-d)) \ar[r] \ar[d] &  \wedge^{2d-r} \Q_p^d  \ar[d] \ar[r] & 0\\ 
0\ar[r] &H^d_c(X_C,\Omega^{r-d-1})/\kker d \ar[r] &H_c^d(X_C, \Omega^{r-d})^{d=0}  \ar[r] &   \wedge^{2d-r} C^d \ar[r] & 0. 
}
\end{equation}
For comparison, we have $H_{\eet,c}^{r}(X_C, \Q_p(r-d))\simeq H_{\eet,c}^{r}(X^{\rm alg}_C, \Q_p(r-d))\simeq  \wedge^{2d-r} \Q_p^d$. 
\begin{remark} Recall that for the usual cohomology we have an exact sequence in $\Q_{p,\Box}$ (see~\cite[Section~4.3.2]{CDN3}),
\begin{equation*} \xymatrix{
0\ar[r] &\Omega^{r-1}(X_C)/\kker d \ar[r] &H_{\proeet}^{r}(X_C, \Q_p(r)) \ar[r] &  \wedge^{r} \Q_p^d   \ar[r] & 0
}
\end{equation*}
We claim that, for any smooth Stein variety $X$ over $C$,  we have an isomorphism in $C_{\Box}$
\begin{equation}\label{tea1}
(\Omega^{r-1}(X)/\kker d)^*\simeq H^d_c(X,\Omega^{d-r})/\kker d. 
\end{equation}
This observation together with diagram \eqref{kol12} was a starting point of the study  of  dualities  for $p$-adic pro-\'etale cohomology in \cite{CGN2}. To prove the isomorphism 
\eqref{tea1}, let $r^{\prime}=d-r$ and consider the exact sequence
\begin{equation}\label{cwir1}
0\to \Omega^{r-1}(X)/\kker d\to \Omega^r(X)^{d=0}\to H^r_{\dr}(X)\to 0.
\end{equation}
Dualizing, we get the exact sequence (note that all the terms in \eqref{cwir1} are nuclear Fr\'echet and use \cite[Th. 3.40]{RR})
$$
0\to H^r_{\dr}(X)^*\to\Omega^r(X)^{d=0,*}\to (\Omega^{r-1}(X)/\kker d)^*\to 0.
$$
Using Poincar\'e duality from Theorem~\ref{dualdR-rigid}, we obtain the exact sequence (where $r^{\prime}=d-r$)
$$
0\to H^{d+r^{\prime}}_{\dr,c}(X)\to H^d_c(X,\Omega^{r^{\prime}})/{\rm Im}\,d\to (\Omega^{r-1}(X)/\kker d)^*\to 0.
$$
Hence an  isomorphism \eqref{tea1}, as wanted. 
\end{remark}
\subsubsection{Drinfeld's space} We now treat the case of the Drinfeld space $X:=\mathbb{H}_K^d$. Recall that we have  natural isomorphisms of solid $K$- and $F^{\rm nr}(\phi,N)$-modules, respectively (see \cite[Lemma 5.11]{CDN3}):
 \[ H^{i}_{\dr}(\mathbb{H}_K^d) \simeq \Sp_{i}(K)^*, \quad H^i_{\hk}(\mathbb{H}_C^d) \simeq \Sp_i(F^{\rm nr})^*\{-i\}, \quad 0\leq i\leq d,  \]
 where $\Sp_i(-)$ is the generalized locally constant Steinberg representation (see \cite[Sec. 5.2.1]{CDN3}). 
We note that the monodromy on $ H^i_{\hk}(\mathbb{H}_C^d)$ is trivial. For $i >d$, $H^{i}_{\dr}(\mathbb{H}_K^d)=0$. Hence we have natural isomorphisms of solid $K$- and $F^{\rm nr}(\phi,N)$-modules, respectively:
$$
 H^{i}_{\dr,c}(\mathbb{H}_K^d) \simeq \Sp_{2d-i}(K), \quad H^i_{\hk,c}(\mathbb{H}_C^d) \simeq \Sp_{2d-i}(F^{\rm nr})\{d-i\},\quad i\geq d.
$$
In particular, the monodromy on $ H^i_{\hk,c}(\mathbb{H}_C^d)$ is trivial. For $i<d$, we have $H^{i}_{\dr,c}(\mathbb{H}_K^d)=0$, hence $H^{i}_{\proeet,c}(\mathbb{H}_K^d, \Q_p)=0$.

  Assume thus that $r\geq d$. 
 \begin{lemma}
For all $k \ge 0$, we have a natural isomorphism of solid $\Q_p$-modules
$$(H^{r-1}_{\HK,c}(X_C) \solid_{F^{\rm nr}} \Bstp)^{\varphi=p^{k}, N=0} \simeq  \Sp_{2d-r+1}(\Q_p) \solid_{\Q_p} (\Bcrp)^{\varphi=p^{k+d-r+1}}. $$  
\end{lemma}

\begin{proof}
From the above formulas  we get
\begin{align*}
(H^{r-1}_{\hk,c}(X_C) \solid_{F^{\rm nr}} \Bstp)^{\varphi=p^{k}, N=0} & \simeq  ( \Sp_{2d-r+1}(F^{\rm nr})\{d-r+1\}\solid_{F^{\rm nr}} \Bcrp)^{\varphi=p^{k}} \\
& \simeq  ( \Sp_{2d-r+1}(F^{\rm nr})\solid_{F^{\rm nr}} \Bcrp)^{\varphi=p^{k+d-r+1}} \\
 & \simeq \Sp_{2d-r+1}(\Q_p) \solid_{\Q_p} (\Bcrp)^{\varphi=p^{k+d-r+1}}.
 \end{align*}
 \end{proof}
 
And as in the case of torus, we obtain that the maps 
\begin{align*}
\pi_{1}^{r-1,r-d} : &  \Sp_{2d-r+1}(\Q_p) \solid_{\Q_p} (\Bcrp)^{\varphi=p} \to \Sp_{2d-r+1}(\Q_p) \solid_{\Q_p} C, \\
\pi_d^{r-1,r} & \Sp_{2d-r+1}(\Q_p) \solid_{\Q_p} (\Bcrp)^{\varphi=p^{d+1}} \to \Sp_{2d-r+1}(\Q_p) \solid_{\Q_p} (\B_{\dr}^+/F^d)
\end{align*}
are surjective. Since, computing as in the the proof of the above lemma, we also have 
$$
(H^r_{\HK,c}(X_C) \solid_{F^{\rm nr}} t^d\Bstp)^{\varphi=p^{r}, N=0} \simeq   \Sp_{2d-r}(\Q_p)(d),$$ from    Theorem \ref{diagSYNc}, we obtain the map of exact sequences 
 in $\Q_{p,\Box}$:  
\begin{equation} \xymatrix{
0\ar[r] &H^d_c(X_C,\Omega^{r-d-1})/\kker d \ar@{=}[d]\ar[r] &H_{\proeet,c}^{r}(X_C, \Q_p(r-d)) \ar[r] \ar[d] & \Sp_{2d-r}(\Q_p) \ar[d] \ar[r] & 0\\ 
0\ar[r] &H^d_c(X_C,\Omega^{r-d-1})/\kker d\ar[r] &H_c^d(X_C, \Omega^{r-d})^{d=0}  \ar[r] &  \Sp_{2d-r}(C) \ar[r] & 0.} 
\end{equation}
For comparison, we have an isomorphism (see \cite[Th. 1.2]{CDHN})
$$
H_{\eet,c}^{r}(X_C, \Q_p(r-d)) \simeq \Sp^{\rm cont}_{2d-r}(\Q_p), 
$$
where $\Sp^{\rm cont}_i(-)$ is the generalized continuous Steinberg representations (see \cite[Sec. 5.2.1]{CDN3}). Hence we should have a commutative diagram (we did not check compatibilities)
$$
\xymatrix{
H_{\proeet,c}^{r}(X_C, \Q_p(r-d)) \ar[r] \ar[d] & \Sp_{2d-r}(\Q_p) \ar@{^(->}[d] \\
H_{\eet,c}^{r}(X_C, \Q_p(r-d)) \ar[r]  & \Sp^{\rm cont}_{2d-r}(\Q_p).
}
$$


\begin{thebibliography}{9999}
\bibitem{ALBM}  J.~Ansch\"utz, A.-C.~Le Bras, L.~Mann,  {\em A 6-functor formalism for solid quasi-coherent sheaves on the Fargues-Fontaine curve.}  arXiv:2412.20968[math.AG].
 \bibitem{Ber93} V.~G.~Berkovich, {\em \'{E}tale cohomology for non-Archimedean analytic spaces}. Inst. Hautes \'{E}tudes Sci. Publ. Math. 78 (1994),5--161.
 \bibitem{B} V.~Berkovich, {\em Vanishing cycles for formal schemes}.
Invent. Math. 115 (1994), no. 3, 539--571.
 \bibitem{Bey97} P.~Beyer, {\em On Serre-duality for coherent sheaves on rigid-analytic spaces}. Manuscripta Math. 93 (1997), no. 2, 219--245.
 \bibitem{BS15} B.~Bhatt, P.~Scholze, {\em The pro-\'{e}tale topology for schemes}. Ast\'{e}risque 369 (2015), 99--201.
 \bibitem{Bos21} G.~Bosco, {\em On the $p$-adic pro-\'{e}tale cohomology of Drinfeld symmetric spaces}. arXiv:2110.10683v2 [math.NT]. 
 \bibitem{Bos22} G.~Bosco, {\em Rational $p$-adic Hodge theory for rigid-analytic varieties}. arXiv:2306.06100 [math.AG].
 \bibitem{Ch90} B.~Chiarellotto, {\em Duality in rigid analysis. $p$-adic analysis (Trento, 1989)}, 142--172, Lecture Notes in Math., 1454, Springer, Berlin, 1990.
\bibitem{CDN1}  P.~Colmez, G.~Dospinescu, W.~Nizio{\l}, {\em Cohomologie $p$-adique de la tour de Drinfeld: le cas de la dimension 1.} J. Amer.
Math. Soc. 33 (2020), 311--362.
\bibitem{CDN3}  P.~Colmez, G.~Dospinescu, W.~Nizio\l, {\em Cohomology of $p$-adic Stein spaces}.  Invent. Math. 219 (2020), 873--985.
\bibitem{CDHN} P.~Colmez, G.~Dospinescu, J.~Hauseux, W.~ Nizio\l, {\em $p$-adic \'etale cohomology of period domains.} Math. Ann. 381 (2021), no. 1-2, 105--180.
 \bibitem{CGN} P.~Colmez, S.~Gilles, W.~Nizio\l, {\em Arithmetic duality for $p$-adic pro-\'etale cohomology of analytic curves.}   arXiv:2308.07712v2 [math.NT], to appear in Compositio  Math.
 \bibitem{CGN2} P.~Colmez, S.~Gilles, W.~Nizio\l, {\em Duality for $p$-adic geometric pro-\'etale cohomology.}   arXiv:2411.12163 [math.AG].
   \bibitem{Haw}  P.~Colmez, S.~Gilles, W.~Nizio{\l}, {\em Une conjecture  $C_{\st}$ pour la cohomologie \`a support compact.}  arXiv:2511.14675  [math.AG].
  \bibitem{CN1} P.~Colmez, W.~Nizio\l, {\em Syntomic complexes and p-adic nearby cycles}. Invent. Math. 208 (2017), no.1, 1-108.                                
 \bibitem{CN3} P.~Colmez, W.~Nizio\l, {\em On $p$-adic comparison theorems for rigid analytic spaces, I}. M\"unster J. Math. 13 (2020) (Special Issue: In honor of Ch. Deninger), 445--507. 
 \bibitem{CN4} P.~Colmez, W.~Nizio\l, {\em On the cohomology of $p$-adic analytic spaces, I: The basic comparison theorem.} Journal of Algebraic Geometry 34 (2025), 1--108. 
 \bibitem{CN5} P.~Colmez, W.~Nizio{\l}, {\em On the cohomology of $p$-adic analytic spaces, II: The $C_{\st}$-conjecture.}  Duke Math. J. 174 (2025), no.11, 2203--2301.
  \bibitem{FF} { L.~Fargues} et {\sc J.-M.~Fontaine}, Courbes et fibr\'es vectoriels en th\'eorie de Hodge $p$-adique, Ast\'erisque~{\bf 406} (2018), 51--382.
 \bibitem{FujiwaraKato} K.~Fujiwara, F.~Kato, {\em Foundations of rigid geometry. I.} EMS Monogr. Math. European Mathematical Society (EMS), Z\"urich, 2018.
  \bibitem{GK1} E.~Grosse-Kl\"onne, {\em Rigid analytic spaces with overconvergent structure sheaf.} J. Reine Angew. Math. 519 (2000), 73--95. 
  \bibitem{Huber} R.~Huber, {\em \'Etale cohomology of rigid analytic varieties and adic spaces}. Aspects of Mathematics, E30. Friedr. Vieweg and Sohn, Braunschweig, 1996.
  \bibitem{Hub98} R.~Huber, {\em A comparison theorem for $\ell$-adic cohomology.} Compositio Math. 112 (1998), no. 2, 217--235.
  \bibitem{Jann88} U.~Jannsen, {\em Continuous \'etale cohomology.} Math. Ann. 280 (1988), no. 2, 207--245.
  \bibitem{Ka94-1}K.~ Kato, {\em Toric singularities.} Amer. J. Math. 116 (1994), no. 5, 1073--1099.
  \bibitem{LLZ} K.-W.~Lan, R.~Liu, X.~Zhu, {\em de Rham comparison and Poincar\'e duality for rigid varieties.} Peking Math. J. 6 (2023), no. 1, 143--216.
  \bibitem{ZLi} Z.~Li, {\em Duality for arithmetic $p$-adic pro-\'etale cohomology of analytic spaces.} arXiv:2412.11786 [math.AG].
  \bibitem{LurieHA}J.~Lurie, {\em Higher algebra.} Available at https://www.math.ias.edu/~lurie/papers/HA.pdf. 
  \bibitem{Lutk90} W.~L\"utkebohmert, {\em Formal-algebraic and rigid-analytic geometry.} Math. Ann. 286 (1990), no. 1-3, 341--371.
   \bibitem{NN16} J.~Nekov\'{a}\v{r}, W.~Nizio{\l}, {\em Syntomic cohomology and regulators for varieties over $p$-adic fields}, Algebra Number Theory 10 (2016), no. 8, 1695-1790.
  \bibitem{RR} J.~Rodrigues Jacinto, J. E.~Rodriguez Camargo, {\em Solid locally analytic representations of $p$-adic Lie groups.} Represent. Theory 26 (2022), 962--1024.
  \bibitem{Sch13} P.~Scholze, {\em $p$-adic Hodge theory for rigid-analytic varieties.} Forum Math. Pi 1 (2013), e1, 77 pp.
  \bibitem{Sch19} P.~Scholze, {\em Condensed Maths}. Available at http://www.math.uni-bonn.de/people/scholze/Condensed.pdf.
  \bibitem{Shao} X.~Shao, {\em Hyodo-Kato cohomology in rigid geometry: some foundational results.}  	arXiv:2502.13315 [math.AG].
  \bibitem{Ta67} {J.\,Tate}, 
{ $p$-divisible groups,} {\it Proc. of a
conference on local fields}, Nuffic Summer School at Driebergen, 
158--183, Springer 1967.
  \bibitem{VdP92} M.~van der Put, {\em Serre duality for rigid analytic spaces.} Indag. Math. (N.S.) 3 (1992), no. 2, 219--235.
   \bibitem{VRo78} A.~C.~M.~van Rooij, {\em Non archimedean functional analysis.} Monographs and Textbooks in pure and applied mathematics 51(1978), Marcel Dekker, Inc, Basel.
     \bibitem{Vez} A.~Vezzani, {\em The Monsky-Washnitzer and the overconvergent realizations}. Int. Math. Res. Not. IMRN 2018, no. 11, 3443--3489.
  \bibitem{YY} G.~Yamashita, S.~Yasuda, {\em $p$-adic \'etale cohomology and crystalline cohomology for the open varieties with semistable reduction.} Preprint.
  \bibitem{Zav} B.~Zavyalov, {\em Mod $p$ Poincar\'{e} duality in $p$-adic analytic geometry.}  	Ann. of Math. (2) 201 (2025), no. 3, 647--773.
    \bibitem{StacksProject} {\em The Stacks Project.} Available at http://stacks.math.columbia.edu/.
	
 \end{thebibliography}
 \end{document}